%% file: final.tex
\documentclass[11pt]{article}

\makeatletter
\newcommand\RedeclareMathOperator{%
  \@ifstar{\def\rmo@s{m}\rmo@redeclare}{\def\rmo@s{o}\rmo@redeclare}%
}
\newcommand\rmo@redeclare[2]{%
  \begingroup \escapechar\m@ne\xdef\@gtempa{{\string#1}}\endgroup
  \expandafter\@ifundefined\@gtempa
     {\@latex@error{\noexpand#1undefined}\@ehc}%
     \relax
  \expandafter\rmo@declmathop\rmo@s{#1}{#2}}
\newcommand\rmo@declmathop[3]{%
  \DeclareRobustCommand{#2}{\qopname\newmcodes@#1{#3}}%
}
\@onlypreamble\RedeclareMathOperator
\makeatother

\usepackage{witharrows}
\usepackage{tikz}
\usepackage{graphicx,wrapfig,lipsum}   

\usepackage{xspace}

\usepackage{framed}

\usepackage{amsmath,amssymb}
\usepackage{amsthm}
\usepackage{mathtools}
\usepackage[noend]{algorithmic}
\usepackage[ruled,vlined]{algorithm2e}
\usepackage{url}
\usepackage{fullpage}
\usepackage{makeidx}
\usepackage{enumerate}
\usepackage[top=1in, bottom=1.25in, left=0.75in, right=0.75in]{geometry}
\usepackage{graphicx,float,psfrag,epsfig,caption}

\usepackage{hyperref}
\usepackage{epstopdf}
\usepackage{color}
\usepackage{xr}
\usepackage{subcaption}
\usepackage[utf8]{inputenc}

\usepackage{scalerel,stackengine}

\usepackage{enumitem}
\usepackage{bbm}

\stackMath
\newcommand\reallywidehat[1]{%
\savestack{\tmpbox}{\stretchto{%
  \scaleto{%
    \scalerel*[\widthof{\ensuremath{#1}}]{\kern.1pt\mathchar"0362\kern.1pt}%
    {\rule{0ex}{\textheight}}
  }{\textheight}%
}{2.4ex}}%
\stackon[-6.9pt]{#1}{\tmpbox}%
}
\usepackage[mathscr]{euscript} 

\usepackage[T1]{fontenc}

\DeclareSymbolFont{rsfs}{U}{rsfs}{m}{n}
\DeclareSymbolFontAlphabet{\mathscrsfs}{rsfs}

\numberwithin{equation}{section}

\newtheoremstyle{myexample} 
    {\topsep}                    
    {\topsep}                    
    {\rm }                   
    {}                           
    {\bf }                   
    {.}                          
    {.5em}                       
    {}  

\newtheoremstyle{myremark} 
    {\topsep}                    
    {\topsep}                    
    {\rm}                        
    {}                           
    {\bf}                        
    {.}                          
    {.5em}                       
    {}  

\newtheorem{claim}{Claim}[section]
\newtheorem{lemma}[claim]{Lemma}

\newtheorem{theorem}{Theorem}
\newtheorem{proposition}[claim]{Proposition}
\newtheorem{corollary}[claim]{Corollary}
\newtheorem{definition}[claim]{Definition}

\theoremstyle{myremark}
\newtheorem{remark}{Remark}[section]

\theoremstyle{myremark}

\theoremstyle{myexample}
\newtheorem{example}[remark]{Example}

\definecolor{darkgreen}{rgb}{0.0, 0.5, 0.0}

\hypersetup{
  colorlinks   = true, 
  urlcolor     = blue, 
  linkcolor    = blue, 
  citecolor   = red 
}

\input{commands.tex}
\title{Tight Lipschitz Hardness for Optimizing Mean Field Spin Glasses}
\author{
    Brice Huang\thanks{Department of Electrical Engineering and Computer Science, Massachusetts Institute of Technology.}
    \and
    Mark Sellke\thanks{Department of Mathematics, Stanford University.}
}
\date{}

\begin{document}

\maketitle
\vspace{-4pt}

\input{tex/0-abstract}
\tableofcontents

\input{tex/1-intro}
\input{tex/2-results}
\input{tex/3-main-proof}

\input{tex/4-interpolation}
\input{tex/5-grand-hamiltonian-ub-spherical}
\input{tex/6-grand-hamiltonian-ub-ising}
\input{tex/7-necessity-of-branching-tree}
\input{tex/8-stable-algorithms}

\input{tex/ack}

\bibliographystyle{alpha}
\bibliography{all-bib}

\appendix
\input{tex/a1-deferred-proofs}
\input{tex/a2-spherical-minimizer}

\end{document}

%% file: commands.tex
\newcommand{\bea}{\begin{eqnarray}}
\newcommand{\eea}{\end{eqnarray}}
\newcommand{\<}{\langle}
\renewcommand{\>}{\rangle}

\newcommand{\wt}{\widetilde}
\newcommand{\op}{\text{op}}

\def\eps{{\varepsilon}}

\def\bh{\boldsymbol{h}}

\def\cuK{\mathscrsfs{K}}
\def\cuU{\mathscrsfs{U}}
\def\cuL{\mathscrsfs{L}}
\def\cuV{\mathscrsfs{V}}

\def\ind{{\mathbb I}}

\def\hzeta{\widehat{\zeta}}

\def\bSigma{{\boldsymbol{\Sigma}}}

\def\hq{{\widehat{q}}}

\def\hB{{\widehat{B}}}

\def\bg{{\boldsymbol{g}}}
\def\bx{{\boldsymbol{x}}}

\def\cX{{\mathcal X}}
\def\cZ{{\mathcal Z}}

\def\Z{{\mathbb Z}}

\def\op{\mbox{\tiny\rm op}}

\def\uu{\underline{u}}
\def\vk{\vec k}
\def\vp{\vec p}
\def\vv{\vec v}
\def\vx{\vec x}
\def\vq{\vec q}

\def\vB{\vec B}
\def\vV{\vec V}
\def\vX{\vec X}
\def\vW{\vec W}
\def\uvq{\underline{\vec q}}
\def\uvp{\underline{\vec p}}
\def\uvk{\underline{\vec k}}
\def\wtH{\wt{H}}
\def\wtB{\wt{B}}
\def\wtq{\wt{q}}
\def\vone{\vec 1}

\def\bsig{{\boldsymbol \sigma}}
\def\vbsig{{\vec \bsig}}
\def\uvbsig{\underline{\vec \bsig}}
\def\ubsig{\underline{\bsig}}
\def\hbsig{{\widehat \bsig}}

\def\brho{{\boldsymbol \rho}}

\def\vbrho{{\vec \brho}}

\def\bfeta{{\boldsymbol \eta}}
\def\veta{{\vec \eta}}
\def\vbfeta{{\vec \bfeta}}
\def\by{{\boldsymbol y}}
\def\vy{{\vec y}}
\def\vby{{\vec \by}}
\def\vbx{{\vec \bx}}

\def\R{{\mathbb R}}

\def\Z{{\mathbb Z}}

\def\bv{{\boldsymbol{v}}}

\def\bx{{\boldsymbol{x}}}

\def\bm{\boldsymbol{m}}

\def\vbm{\boldsymbol{\vec m}}
\def\vbh{\boldsymbol{\vec h}}

\def\Par{{\sf P}}
\def\de{{\rm d}}

\def\prob{{\mathbb P}}

\def\<{\langle}
\def\>{\rangle}
\def\Tr{{\sf Tr}}

\def\cH{{\cal H}}
\def\cM{{\cal M}}
\def\cN{{\cal N}}

\def\uD{\underline{D}}

\def\cY{{\cal Y}}

\def\by{{\boldsymbol{y}}}

\def\P{\mathbb{P}}

\def\b0{{\boldsymbol{0}}}

\def\hH{\widehat{H}}

\DeclareMathOperator*{\plim}{p-lim}

\def\OPT{{\sf OPT}}

\def\cA{{\mathcal A}}

\def\cS{{\mathcal S}}
\def\bn{{\boldsymbol n}}
\def\cB{{\mathcal B}}

\def\cK{{\mathcal K}}

\def\ozeta{\overline{\zeta}}

\def\uzeta{\underline{\zeta}}
\def\oI{\overline{I}}
\def\uI{\underline{I}}
\def\oJ{\overline{J}}
\def\uJ{\underline{J}}
\def\uK{\underline{K}}
\def\uQ{\underline{Q}}
\def\PA{\mathrm{pa}}

\renewcommand{\b}{\mathbf{b}}

\def\fr{\frac}
\def\lt{\left}
\def\rt{\right}

\def\la{\langle}
\def\ra{\rangle}

\def\eps{\varepsilon}

\def\bbE{{\mathbb{E}}}
\def\bbL{{\mathbb{L}}}
\def\bbN{{\mathbb{N}}}
\def\bbP{{\mathbb{P}}}
\def\bbR{{\mathbb{R}}}
\def\bbS{{\mathbb{S}}}
\def\bbT{{\mathbb{T}}}
\def\bbZ{{\mathbb{Z}}}

\def\ubbL{\underline{\mathbb{L}}}

\def\cA{{\mathcal{A}}}
\def\cB{{\mathcal{B}}}

\def\cK{{\mathcal{K}}}
\def\cN{{\mathcal{N}}}

\def\cQ{{\mathcal{Q}}}
\def\ucQ{\underline{\mathcal{Q}}}

\def\sH{{\mathscr{H}}}

\def\bg{{\mathbf{g}}}
\def\bh{{\boldsymbol{h}}}

\def\hbsig{{\boldsymbol {\hat \sigma}}}

\def\hcA{{\widehat \cA}}

\def\ALG{{\mathsf{ALG}}}
\def\OPT{{\mathsf{OPT}}}

\def\TV{{\mathrm{TV}}}
\def\GS{{\mathrm{GS}}}

\def\Ssolve{S_{\mathrm{solve}}}

\def\Soverlap{S_{\mathrm{overlap}}}

\def\Sogp{S_{\mathrm{ogp}}}
\def\Seigen{S_{\mathrm{eigen}}}

\def\smax{{s_{\max}}}

\DeclareMathOperator*{\E}{\bbE}
\RedeclareMathOperator*{\P}{\bbP}
\DeclareMathOperator*{\argmax}{\arg\,\max}

\def\Sp{\mathrm{Sp}}
\def\Is{\mathrm{Is}}

\newcommand{\Gp}[1]{\mathbf{G}^{(#1)}}
\newcommand{\gp}[1]{\mathbf{g}^{(#1)}}
\newcommand{\gb}[1]{\mathbf{g}^{[#1]}}
\newcommand{\wtHNp}[1]{{\wt{H}_N^{(#1)}}}

\newcommand{\HNp}[1]{{H_N^{(#1)}}}
\newcommand{\bgp}[1]{\bg^{(#1)}}

\newcommand{\Sum}{\mathrm{Sum}}

\newcommand{\diff}[1]{{\mathrm{d}#1}}

\newcommand{\norm}[1]{{\lt\|#1\rt\|}}

\newcommand{\He}{\mathrm{He}}
\newcommand{\tHe}{\widetilde \He}

\newcommand{\ucuL}{\underline{\cuL}}
\newcommand{\ocuL}{\overline{\cuL}}
\newcommand{\ocuK}{\overline{\cuK}}

\newcommand{\uk}{\underline{k}}

\newcommand{\pderiv}[2]{{\fr{\partial #1}{\partial #2}}}

\newcommand{\Span}{\mathrm{span}}

%% file: tex/0-abstract.tex
\begin{abstract}
    We study the problem of algorithmically optimizing the Hamiltonian $H_N$ of a spherical or Ising mixed $p$-spin glass. 
    The maximum asymptotic value $\OPT$ of $H_N/N$ is characterized by a variational principle known as the Parisi formula, proved first by Talagrand and in more generality by Panchenko.
    Recently developed approximate message passing algorithms efficiently optimize $H_N/N$ up to a value $\ALG$ given by an extended Parisi formula, which minimizes over a larger space of functional order parameters.
    These two objectives are equal for spin glasses exhibiting a \emph{no overlap gap} property.
    However, $\ALG < \OPT$ can also occur, and no efficient algorithm producing an objective value exceeding $\ALG$ is known.
    
    We prove that for mixed even $p$-spin models, no algorithm satisfying an \emph{overlap concentration} property can produce an objective larger than $\ALG$ with non-negligible probability. 
    This property holds for all algorithms with suitably Lipschitz dependence on the disorder coefficients of $H_N$. It encompasses natural formulations of gradient descent, approximate message passing, and Langevin dynamics run for bounded time and in particular includes the algorithms achieving $\ALG$ mentioned above. 
    To prove this result, we substantially generalize the overlap gap property framework introduced by Gamarnik and Sudan to arbitrary ultrametric forbidden structures of solutions.
\end{abstract}

%% file: tex/1-intro.tex
\section{Introduction}

In a \emph{random optimization problem}, one sets out to optimize an objective function generated from random data.
The computational complexity of these problems is not well understood, due to the fact that they are often both non-convex and high-dimensional.
Optimizing non-convex functions in high dimensions is well-known to be computationally intractable in the worst case; however, worst-case lower bounds rely on highly structured hard instances, and in average-case settings the picture is far less clear.

In this paper we obtain a sharp computational threshold for a natural class of random optimization problems, namely the Hamiltonians of mean-field spin glasses. These functions have been studied since \cite{sherrington1975solvable} as models for disordered magnetic systems. From a mathematical point of view, they are simply polynomials or power series in many variables with independent and identically distributed coefficients. Moreover as discussed below they are closely related to random combinatorial optimization problems such as $k$-SAT and MaxCut.
Our main result is a lower bound against a natural class of \textbf{stable} algorithms which exactly matches the best known algorithms for this problem.

Our problem is defined as follows. For each $p\in 2\bbN$, let $\Gp{p} \in \lt(\bbR^N\rt)^{\otimes p}$ be an independent $p$-tensor with i.i.d. $\cN(0,1)$ entries. Let $h \ge 0$, and set $\bh = (h,\ldots,h) \in \bbR^N$. Fix a sequence $(\gamma_p)_{p\in 2\bbN}$ with $\gamma_p\ge 0$ and $\sum_{p\in 2\bbN} 2^p \gamma_p^2 < \infty$.
The mixed even $p$-spin Hamiltonian $H_N$ is defined by
\begin{align}
    \label{eq:def-hamiltonian}
    H_N(\bsig) &= \la \bh, \bsig \ra + \wtH_N(\bsig), \quad \text{where}\\
    \label{eq:def-hamiltonian-no-field}
    \wtH_N(\bsig) &=\sum_{p\in 2\bbN} \fr{\gamma_p}{N^{(p-1)/2}} \la \Gp{p}, \bsig^{\otimes p} \ra.
\end{align}

We consider inputs $\bsig$ in either the sphere $S_N = \{\bsig \in \bbR^N : \norm{\bsig}_2^2 = N\}$ or the cube $\Sigma_N = \{-1, 1\}^N$.
These define, respectively, the \emph{spherical} and \emph{Ising} mixed $p$-spin glass models. The coefficients $\gamma_p$ are customarily encoded in the \emph{mixture function} $\xi(x) = \sum_{p\in 2\bbN} \gamma_p^2 x^p$. Note that $\wtH_N$ is equivalently described as the Gaussian process with covariance
\[
    \E \wtH_N(\bsig^1) \wtH_N(\bsig^2) = N \xi(\la \bsig^1,\bsig^2\ra/N)
\]
while the term $\la \bh, \bsig \ra$ represents an external field.
Our purpose is to shed light on a discrepancy between the in-probability limiting maximum values
\[
    \OPT^{\Sp}_{\xi,h}=\plim_{N\to\infty}\frac{1}{N}\max_{\bsig\in S_N}H_N(\bsig), \qquad 
    \OPT^{\Is}_{\xi,h}=\plim_{N\to\infty}\frac{1}{N}\max_{\bsig\in \Sigma_N}H_N(\bsig)
\]
and the maximum \emph{efficiently computable} values of $H_N$ over the same sets. 
We will write $\OPT^{\Sp} = \OPT^{\Sp}_{\xi,h}$ and $\OPT^{\Is} = \OPT^{\Is}_{\xi,h}$ when $\xi,h$ are clear from context.

\subsection{$\OPT$, $\ALG$, and the Parisi Functional}
The values $\OPT^{\Sp}$ and $\OPT^{\Is}$ are given by the celebrated Parisi formula \cite{parisi1979infinite} which was proved for even models by \cite{talagrand2006parisi,talagrand2006spherical} and in more generality by \cite{panchenko2014parisi}. 
While most often stated as a formula for the limiting free energy at inverse temperature $\beta$, the asymptotic maximum can be recovered as a $\beta\to\infty$ limit of the Parisi formula. 
Restricting for concreteness to the Ising case (we will state the analogous result for the spherical case in Section~\ref{sec:results}), the result can be expressed in the following form due to Auffinger and Chen \cite{auffinger2017parisi}.
Define the function space
\begin{equation}
    \label{eq:def-cuU}
    \cuU= \lt\{
        \zeta : [0,1) \to \bbR_{\ge 0}:
        \text{$\zeta$ is right-continuous and nondecreasing,}
        \int_0^1 \zeta(t) \de t < \infty
    \rt\}.
\end{equation}
For $\zeta\in \cuU$, define $\Phi_{\zeta}:[0,1]\times \bbR\to\bbR$ to be the solution of the following \emph{Parisi PDE}.
\begin{align}
    \label{eq:intro-1dParisiPDEdefn}
    \partial_t \Phi_{\zeta}(t,x)+\frac{1}{2}\xi''(t)\left(\partial_{xx}\Phi_{\zeta}(t,x)+\zeta(t)(\partial_x \Phi_{\zeta}(t,x))^2\right)&=0\\
    \Phi_{\zeta}(1,x)&=|x|.
\end{align}

Existence and uniqueness properties for this PDE are well established and are reviewed in Subsection~\ref{subsec:1d-parisi-review}. 
The Parisi functional $\Par^{\Is} = \Par^{\Is}_{\xi,h} : \cuU \to \bbR$ is given by
\begin{equation}
    \label{eq:def-parisi-functional-is}
    \Par^{\Is}(\zeta) 
    = 
    \Phi_{\zeta}(0,h) - \frac{1}{2}\int_{0}^1 t\xi''(t)\zeta(t) \diff{t}.
\end{equation}

\begin{theorem}[{\cite[Theorem 1]{auffinger2017parisi}}]
    \label{thm:opt-is}
    The following identity holds.
    \begin{equation}
        \label{eq:opt-variational-formula-is}
        \OPT^{\Is}
        = 
        \inf_{\zeta\in \cuU}
        \Par^{\Is}(\zeta).
    \end{equation}
\end{theorem}

The infimum over $\zeta\in\cuU$ is achieved at a unique $\zeta_*\in \cuU$ as shown in \cite{auffinger2017parisi,chen2018energy}, which can be obtained as an appropriately renormalized zero-temperature limit of the corresponding minimizers in the positive temperature Parisi formula. 
These positive temperature minimizers roughly correspond to cumulative distribution functions for the overlap $\la \bsig^1,\bsig^2\ra/N$ of two replicas $\bsig^1,\bsig^2$ sampled from the Gibbs measure $e^{\beta H_N}/Z_N(\beta)$; this is why the functions $\zeta$ considered in the Parisi formula are nondecreasing.

Efficient algorithms to find an input $\bsig$ achieving a large objective have recently emerged in a line of work initiated by \cite{subag2018following} and continued in \cite{mon18,ams20,sellke2021optimizing}. 
The main results of these works in the Ising case can be described as follows. 
For a function $f: \bbR\to\bbR$ and interval $J$, let $\norm{f}_{\TV(J)}$ denote the total variation of $f$ on $J$, expressed as the supremum over partitions: 
\[
    \norm{f}_{\TV(J)}
    =
    \sup_n
    \sup_{t_0 < t_1 < \cdots < t_n, t_i\in J}
    \sum_{i=1}^n |f(t_i) - f(t_{i-1})|.
\]
Let $\cuL\supseteq \cuU$ denote the set of functions given by
\begin{equation}
    \label{eq:def-cuL}
    \cuL = \lt\{
        \begin{array}{ll}
            \displaystyle
            \zeta : [0,1) \to \bbR_{\ge 0}:
            &\text{$\zeta$ right-continuous},
            \norm{\xi''\cdot \zeta}_{\TV[0,t]}<\infty~\text{for all $t\in [0,1)$}, \\
            &\displaystyle \int_0^1 \xi''(t) \zeta(t) \diff{t} < \infty
        \end{array}
    \rt\}.
\end{equation}
It turns out (see Subsection~\ref{subsec:1d-parisi-review}) that the definition of $\Par^{\Is}$ above extends from $\cuU$ to $\cuL$. 
Therefore we may define $\ALG^{\Is} = \ALG^{\Is}_{\xi, h}$ by
\begin{equation}
    \label{eq:def-alg-is}
    \ALG^{\Is} = \inf_{\zeta\in \cuL}\Par^{\Is}(\zeta).
\end{equation}
Note that $\ALG^{\Is}\leq \OPT^{\Is}$ trivially holds.
We have $\ALG^{\Is} = \OPT^{\Is}$ if the infimum in \eqref{eq:def-alg-is} is attained by some $\zeta \in \cuU$, and otherwise $\ALG^{\Is} < \OPT^{\Is}$.

\begin{theorem}[\cite{ams20,sellke2021optimizing}]
    \label{thm:ams20}
    Assume there exists $\zeta_*\in\cuL$ such that $\Par^{\Is}(\zeta_*)=\ALG^{\Is}$. Then for any $\eps>0$, there exists an efficient algorithm $\cA:\sH_N\to C_N$ such that
    \[
        \P[H_N(\cA(H_N))/N\geq \ALG^{\Is}-\eps]\geq 1-o(1),\quad c=c(\eps)>0.
    \] 
\end{theorem}

All of the algorithms in \cite{subag2018following,mon18,ams20,sellke2021optimizing} are computationally efficient. 
The latter three works use a class of iterative algorithms known as approximate message passing (AMP).
In particular they require only a constant number of queries of $\nabla H_N(\cdot)$; this results in computation time linear in the description length of $H_N$ when $\xi$ is a polynomial, assuming oracle access to $\zeta_*$ and the function $\Phi_{\zeta_*}$. 
AMP offers a great deal of flexibility, and the idea introduced in \cite{mon18} was to use it to encode a stochastic control problem which is in some sense dual to the Parisi formula. 
Based on this idea it was shown in \cite{ams20} that no AMP algorithm of this powerful but specific form can achieve asymptotic value $\ALG^{\Is}+\eps$ in the case $h=0$. 
The non-equality $\ALG^{\Is} < \OPT^{\Is}$ also has a natural interpretation in terms of the optimizer $\zeta_*$ of \eqref{eq:opt-variational-formula-is}. Namely, it implies that $\zeta_*$ is not strictly increasing; see \cite{sellke2021optimizing} for a more precise condition called ``optimizability'' therein. As explained in \cite[Section 6]{sellke2021optimizing}, in the case of even Ising spin glasses this non-equality exactly coincides with the presence of an \emph{overlap gap property} (discussed below) associated with forms of algorithmic hardness. 
It is therefore natural to conjecture that the aforementioned AMP algorithms achieve the best asymptotic energy possible for efficient algorithms.

This belief was also aligned with results on the ``critical point complexity'' of pure spherical spin glasses with $\xi(x) = x^p$ and $h=0$. 
In this case, the analogous value $\ALG^{\Sp}$ is the one obtained by \cite{subag2018following} and coincides with the onset of exponentially many bounded index critical points, as established in \cite{auffinger2013random,subag2017complexity}. 
In this case almost all local optima have energy value $\ALG^{\Sp}\pm o(1)$ with high probability, which suggests from another direction that exceeding the energy $\ALG^{\Sp}$ might be computationally intractable. 
On the other hand, this threshold (see \cite{arous2020geometry}) does not coincide with $\ALG^{\Sp}$ beyond the pure case.

It unfortunately seems difficult to establish any limitations on the power of general polynomial-time algorithms for such a task. 
However one might still hope to characterize the power of natural classes of algorithms that include gradient descent and AMP. 
To this end, we define the following distance on the space $\sH_N$ of Hamiltonians $H_N$.
We identify $H_N$ with its disorder coefficients $(\Gp{p})_{p\in 2\bbN}$, which we concatenate (in an arbitrary but fixed order) into an infinite vector $\bg(H_N)$. 
We equip $\sH_N$ with the (possibly infinite) distance
\[
    \norm{H_N - H'_N}_2 
    = 
    \norm{\bg(H_N) - \bg(H'_N)}_2.
\]
Let $B_N=\{\bsig \in \bbR^N : \norm{\bsig}_2^2 \le N\}$ and $C_N=[-1,1]^N$ be the convex hulls of $S_N$ and $\Sigma_N$, which we equip with the standard $\norm{\cdot}_2$ distance.
A consequence of our main result is that no suitably \emph{Lipschitz} function $\cA:\sH_N\to C_N$ can surpass the asymptotic value $\ALG^{\Is}$. 
(And similarly in the spherical case for $\cA : \sH_N \to B_N$ and an analogous $\ALG^{\Sp}$.) 

\begin{theorem}
    \label{thm:main-lip}
    Let $\tau,\eps>0$ be constants.
    For $N$ sufficiently large, any $\tau$-Lipschitz $\cA:\sH_N\to C_N$ satisfies
    \[
        \bbP\lt[
            H_N(\cA(H_N))/N \ge \ALG^{\Is}+\eps
        \rt]
        \le 
        \exp(-cN),
        \quad c=c(\xi,h,\eps,\tau)>0.
    \]
\end{theorem}

Note that the Lipschitz condition $\norm{\cA(H_N) - \cA(H'_N)}_2 \le \tau \norm{H_N - H'_N}_2$ holds vacuously when the latter distance is infinite.

The algorithms of \cite{mon18,ams20,sellke2021optimizing} are $O(1)$-Lipschitz in the sense above\footnote{Technically the algorithms in these papers round their outputs to the discrete set $\Sigma_N$ at the end, making them discontinuous. Removing the rounding step yields Lipschitz maps $\cA:\sH_N\to C_N$ with the same performance.}. While the approach of \cite{subag2018following} is not Lipschitz, its performance is captured by AMP as explained in \cite[Remark 2.2]{ams20}.\footnote{We also outline a similar impossibility result for a family of variants of \cite{subag2018following} in Subsection~\ref{subsec:subag-bopg-2}.} Hence in tandem with these constructive results, Theorem~\ref{thm:main-lip} identifies the exact asymptotic value achievable by Lipschitz functions $\cA:\sH_N\to C_N$ (assuming the existence of a minimizer $\zeta_*\in\cuL$ as required in Theorem~\ref{thm:ams20}). We also give an analogous result for spherical spin glasses, in which there is no question of existence of a minimizer on the algorithmic side. Let us remark that the rate $e^{-cN}$ in Theorem~\ref{thm:main-lip} is best possible up to the value of $c$, being achieved even for the trivial algorithm $\cA(H_N)=(1,1,\dots,1)$ which ignores its input entirely. 

Many natural optimization algorithms satisfy the Lipschitz property above on a set $K_N\subseteq \sH_N$ of inputs with $1 - \exp(-\Omega(N))$ probability; this suffices just as well for Theorem~\ref{thm:main-lip} thanks to the Kirszbraun extension theorem (see Subsection~\ref{subsec:approx-lip}).
As explained in Section~\ref{sec:overlap-conc-of-algs}, algorithms with this property include the following examples, all run for a constant (i.e. dimension-independent) number of iterations or amount of time.
\begin{itemize}
    \item Gradient descent and natural variants thereof;
    \item Approximate message passing;
    \item More general ``higher-order'' optimization methods with access to $\nabla^k H_N(\cdot)$ for constant $k$;
    \item Langevin dynamics for the Gibbs measure $e^{\beta H_N}$ with suitable reflecting boundary conditions and any positive constant $\beta$.
\end{itemize} 
In fact we will not require the full Lipschitz assumption on $\cA$, but only a consequence that we call overlap concentration. Roughly speaking, overlap concentration of $\cA$ means that given any fixed correlation between the disorder coefficients of $H_N^1$ and $H_N^2$, the overlap $\langle \cA(H_N^1),\cA(H_N^2)\rangle/N$ tightly concentrates around its mean. This property holds automatically for $\tau$-Lipschitz $\cA$ thanks to concentration of measure on Gaussian space. It also might plausibly be satisfied for some discontinuous algorithms such as the Glauber dynamics.

\subsection{The Overlap Gap Property as a Barrier to Algorithms}

As mentioned previously, mean field spin glasses are just one example of a random optimization problem; other examples include random constraint satisfaction problems, such as random $k$-SAT and MaxCut, as well as random perceptron models \cite{talagrand2011mean1,talagrand2011mean2,ding2019capacity}. 

The main heuristics proposed to understand computational hardness in random optimization problems have focused on geometric properties of the solution space.
One version of this connection was proposed in \cite{achlioptas2008phasetransitions,cojaoghlan2015independent} based on a \emph{shattering} phase transition: for suitable random instances of $k$-SAT, $q$-coloring, and maximum independent set, beyond a threshold constraint density the solution space breaks into exponentially many small components. 
Shattering defeats local search heuristics, suggesting that polynomial-time algorithms should not succeed.
Other predictions based on the \emph{clustering}, \emph{condensation} \cite{krzakala2007gibbs} and \emph{freezing} \cite{zdeborova2007phase} transitions have also been suggested.
While intuitively appealing, the hypothesis that some form of clustering is responsible for hardness has been shown incorrect in notable examples -- see Subsection~\ref{subsec:alg-signatures} for more discussion.

In the past several years, a line of work \cite{gamarnik2014limits, rahman2017independent, gamarnik2017performance, chen2019suboptimality, gamarnik2019overlap, gamarnik2020optimization, wein2020independent, gamarnik2021partitioning, bresler2021ksat, gamarnik2021circuit} on the Overlap Gap Property (OGP) has made substantial progress on rigorously linking solution geometry to hardness.
A survey can be found in \cite{gamarnik2021survey}.
Initiated by Gamarnik and Sudan in \cite{gamarnik2014limits}, this line of work links the absence of certain constellations of solutions in the super-level set $S_E(H_N) = \{\bsig \in C_N : H_N(\bsig) / N \ge E\}$ -- in its original form, a pair of solutions a medium distance apart -- with the failure of algorithms with certain stability properties.
Roughly, these works proceed by contradiction, arguing that any stable algorithm attaining value $E$ would be able to construct the forbidden constellation.
An important difference from the predictions above is that the shattering, clustering, and freezing transitions describe properties of a \emph{typical} solution, while OGP requires the forbidden constellation to \emph{not exist at all}.

Some of these works use a somewhat stronger claim, that such constellations $(\bsig^1,\bsig^2)$ are absent even from $S_E(\HNp{1}) \times S_E(\HNp{2})$, where $\HNp{1},\HNp{2}$ are two correlated copies of a Hamiltonian. 
That is, points $\bsig^i$ in the constellation can be input to different, correlated Hamiltonians.
We will use a similar construction in this work.

In many of these problems, the classic (2-solution, with or without correlated Hamiltonians) OGP shows the failure of stable algorithms above an intermediate value, smaller than the existential maximum but larger than the algorithmic limit.
The argument stalls because below this value, $S_E(H_N)$ does contain pairs of inputs at each possible distance.
To improve the lower bound, subsequent works have used ``multi-OGPs," which consider more complex forbidden structures; this is usually more difficult but often yields much sharper results.
Indeed, multi-OGPs have been used to show nearly-tight hardness results for finding maximum independent sets on $G(N,d/N)$ in the limit $N\to\infty$ followed by $d\to\infty$ \cite{rahman2017independent,wein2020independent} as well as for random $k$-SAT \cite{bresler2021ksat}. In both cases the threshold is attained by a simple local algorithm, which is shown to be optimal within the larger class of low degree polynomials.

The overlap gap property has also been applied previously to the spin glass Hamiltonians we consider.
For \emph{pure} spherical and Ising $p$-spin glasses where $h=0$ and $p\ge 4$ is even, $\ALG<\OPT$ always holds (recall \eqref{eq:opt-variational-formula-is}, \eqref{eq:def-alg-is}). In such models, \cite{gamarnik2020optimization} showed using the classic OGP that low degree polynomials cannot achieve some objective strictly smaller than $\OPT$, extending a similar hardness result of \cite{gamarnik2019overlap} for approximate message passing. \cite{gamarnik2021circuit} extended the conclusions of \cite{gamarnik2020optimization} to Boolean circuits of depth less than $\fr{\log n}{2 \log \log n}$.
As pointed out in \cite[Section 6]{sellke2021optimizing}, these results extend in the Ising case to any mixed even model where $\ALG^{\Is} < \OPT^{\Is}$.
In this paper, we will use a multi-OGP to show that overlap concentrated algorithms cannot reach any objective larger than $\ALG^\Is$, or in the spherical case $\ALG^\Sp$.

The design of our multi-OGP is a significant departure from previous work.
Previous works all use one of the following three forbidden structures, see Figure~\ref{fig:ogps}. 
\begin{itemize}
    \item Classic OGP: two solutions with medium overlap \cite{gamarnik2014limits, chen2019suboptimality, gamarnik2019overlap, gamarnik2020optimization, gamarnik2021circuit}. 
    \item Star OGP: several solutions with approximately the same pairwise overlap \cite{rahman2017independent, gamarnik2017performance, gamarnik2021partitioning}. 
    \item Ladder OGP: several solutions, where the $i$-th solution ($i\ge 2$) has medium ``multi-overlap'' with the first $i-1$ solutions, for a problem-specific notion of multi-overlap of one solution with several solutions \cite{wein2020independent, bresler2021ksat}.
\end{itemize}

\begin{figure}[h]
\begin{subfigure}[b]{.5\textwidth}
\centering
\includegraphics[width=.8\linewidth]{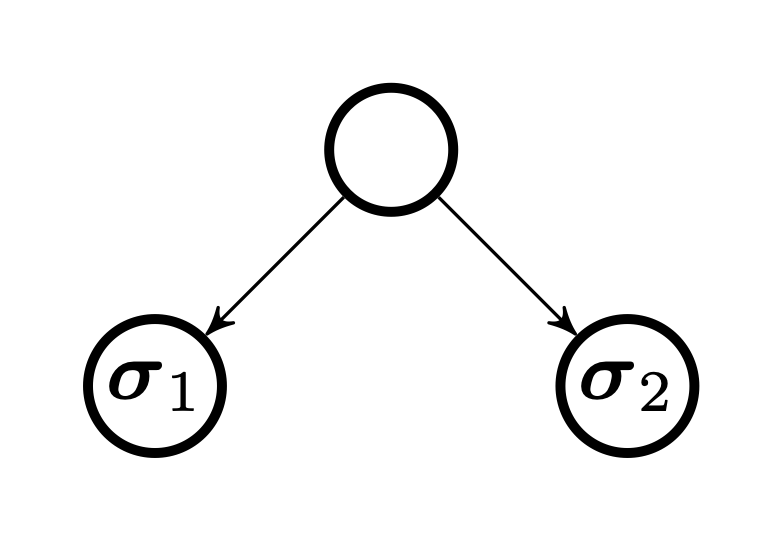}
\caption{\textbf{Classic OGP}: $\bsig_1,\bsig_2$ have medium overlap.}
\end{subfigure}
\hfill
\begin{subfigure}[b]{.5\textwidth}
\centering
\includegraphics[width=\linewidth]{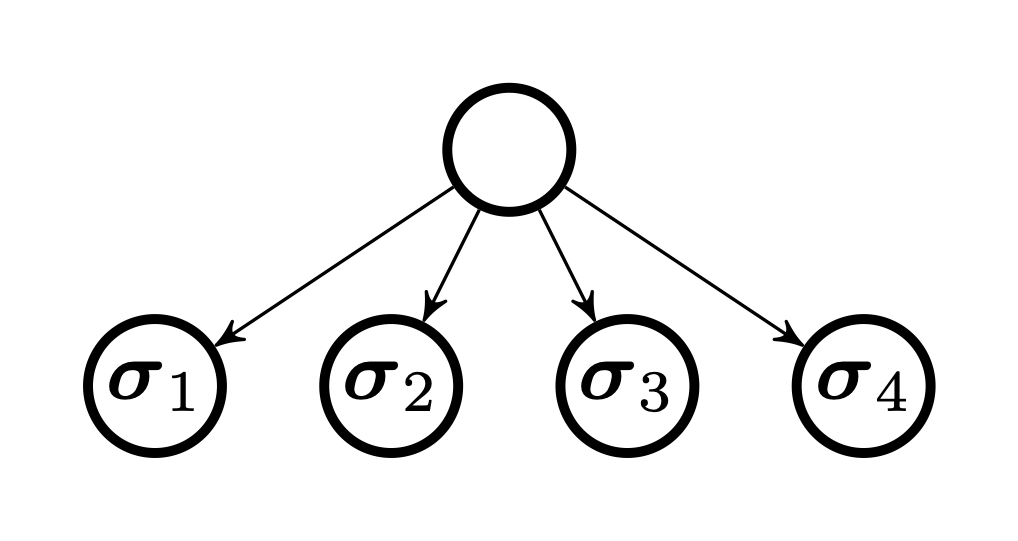}
\caption{\textbf{Star OGP}: many solutions, medium overlaps.}
\end{subfigure}
\hfill
\begin{subfigure}[b]{.4\textwidth}
\centering
\includegraphics[width=.8\linewidth]{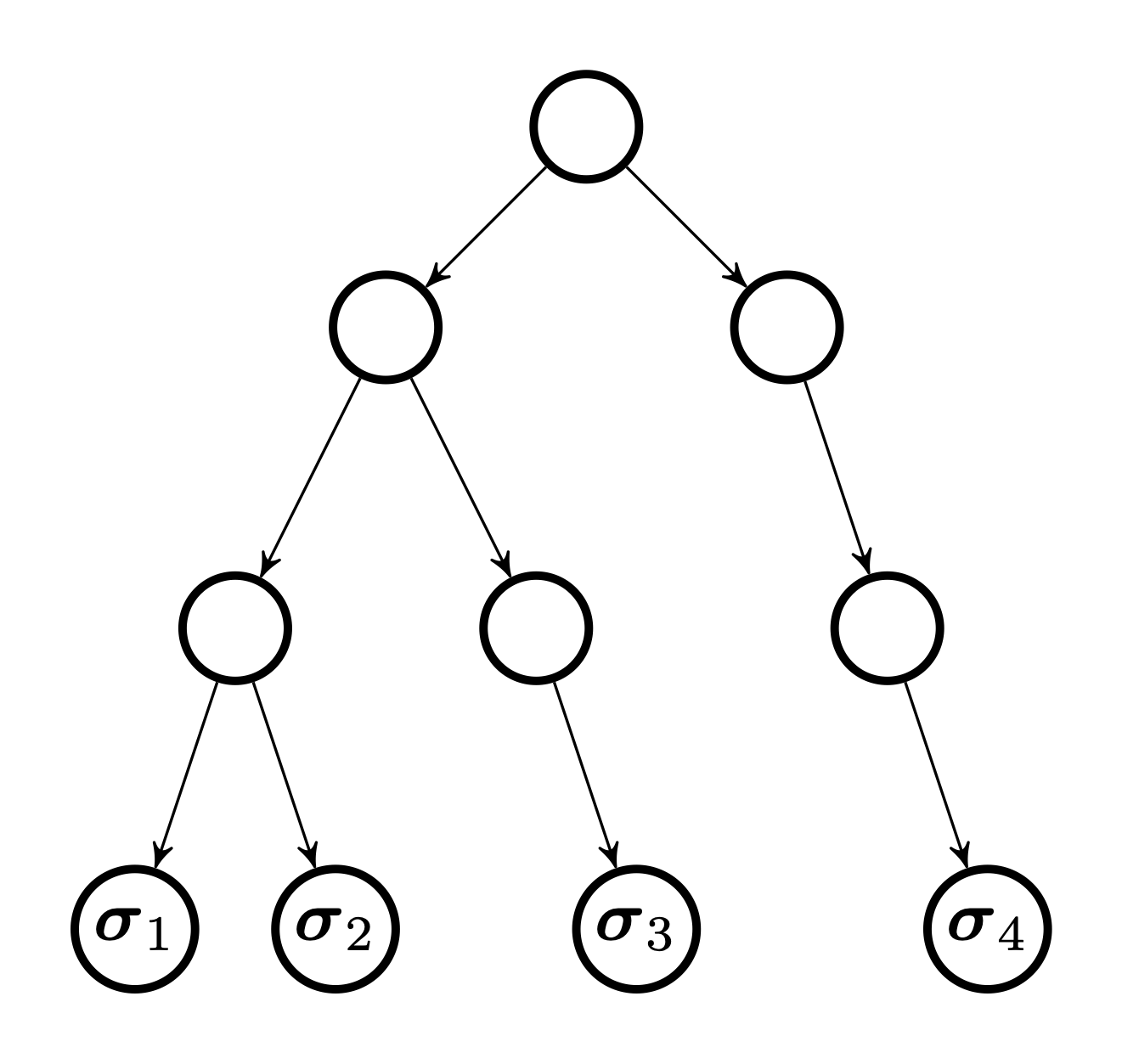}
\caption{\textbf{Ladder OGP}: medium ``multi-overlaps'' between $\bsig_i$ and $\{\bsig_1,\dots,\bsig_{i-1}\}$.}
\end{subfigure}
\begin{subfigure}[b]{.6\textwidth}
\centering

\includegraphics[width=\linewidth]{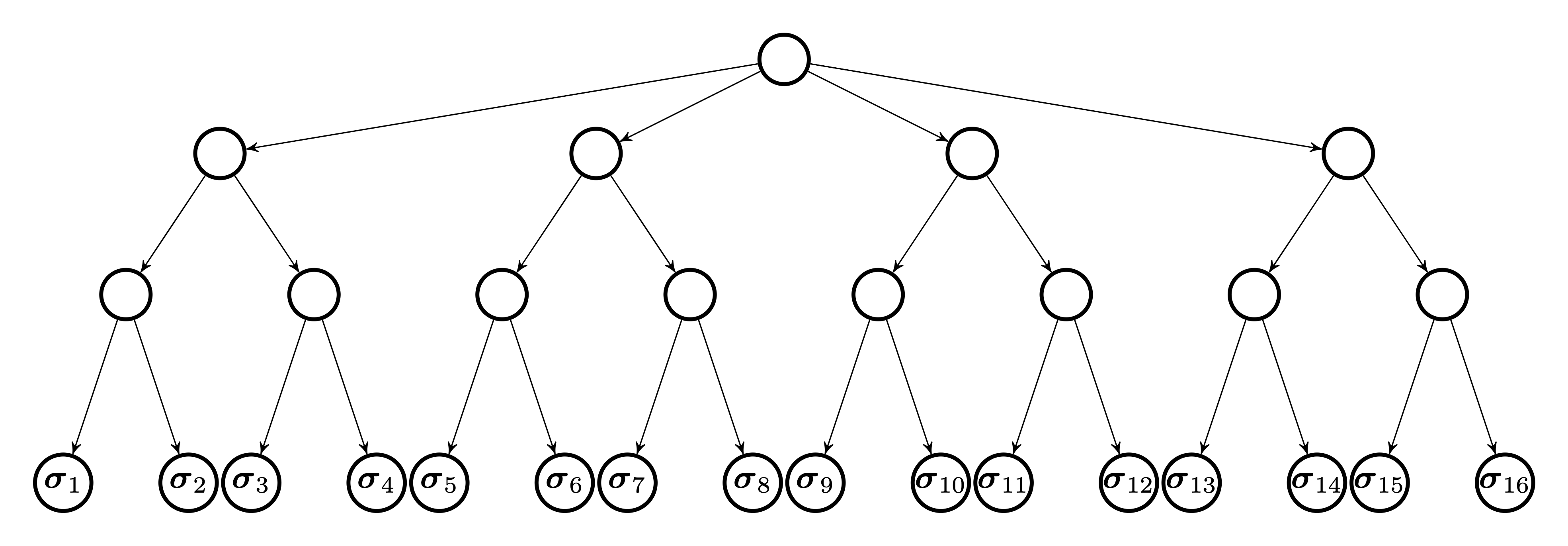}
\vspace{0.1cm}
\caption{\vspace{0.45cm}\textbf{Branching OGP}: many solutions in an ultrametric tree.}
\end{subfigure}

\caption{Schematics of forbidden structures in OGP arguments.}
\label{fig:ogps}
\end{figure}

In contrast, the constellation in our multi-OGP is an arbitrarily complicated ultrametric tree of solutions.
We call this the \emph{Branching OGP}.
Informally, the Branching OGP is the condition that for any fixed $\eps > 0$ and $E = \ALG + \eps$, the set $S_E(H_N)$ does not contain a sufficiently large ultrametric constellation, and this remains true if each point in the constellation is input to the corresponding member of a family of ``ultrametrically correlated" Hamiltonians. 


We establish this branching OGP as follows. 
Using a version of the Guerra-Talagrand interpolation, which we take to zero temperature, we derive an upper bound for the maximum average energy of configurations arranged into the desired structure. 
This upper bound is a multi-dimensional analogue of the Parisi formula, and depends on an essentially arbitrary increasing function $\zeta:[0,1]\to\bbR^+$ (which we are free to minimize over). 
We show that for a symmetric branching tree, the resulting estimate can be upper bounded by $\Par(\kappa\zeta)$. 
Here $\Par$ is the Parisi functional $\Par^{\Is}$ or its spherical analogue $\Par^{\Sp}$, and $\kappa$ is a decreasing piecewise-constant function that depends on the tree. 
By making the tree branch rapidly, the function $\kappa$ can be arranged to decrease as rapidly as desired. As a result, the functions $\kappa\zeta$ are dense in the space $\cuL$.
Thus, we may choose a tree and $\zeta$ such that $\Par(\kappa \zeta)$ is arbitrarily close to $\ALG$.

Roughly speaking, we show that an overlap concentrated $\cA$ allows the construction of an arbitrary ultrametric constellation of outputs. 
Consequently, if $\cA$ outputs points with energy at least $\ALG+\eps$, then $\cA$ run on the appropriate family of ultrametrically correlated Hamiltonians will output the forbidden structure above, a contradiction. 
Some additional complications are created by the fact that $\E[\cA(H_N)]$ may be arbitrary, and that $\cA(H_N)$ may be in the interior of $C_N$ (or in the spherical case, $B_N$). 
The former issue requires us to control the maximum average energy of ultrametric constellations of points that all have approximately a fixed overlap with $\E[\cA(H_N)]$. 
We deal with the latter issue by composing $\cA$ with an additional phase that grows each output of $\cA$ into its own ultrametric tree of points in $\Sigma_N$ (or $S_N$), so that the resulting set of points has the forbidden ultrametric structure.

We also show that the full strength of the branching OGP is necessary to establish Lipschitz hardness at all objectives above $\ALG$, in the sense that any less complex ultrametric structure fails to be forbidden at an energy bounded away from $\ALG$.
More precisely, consider a spherical model $\xi$ without external field; we restrict to this case for convenience. 
Consider a fixed ultrametric overlap structure of inputs, whose corresponding rooted tree (cf. Subsection~\ref{subsec:trees-ultrametrics}) does not contain a full depth-$D$ binary tree.
We prove that if $\ALG^{\Sp} < \OPT^{\Sp}$, with high probability there exists a constellation of inputs with this overlap structure where each input achieves energy at least $\ALG^{\Sp} + \eps_{\xi, D}$, for a constant $\eps_{\xi, D} > 0$ depending only on $\xi, D$.

\begin{remark}
    To our knowledge, this is the first hardness result in any random optimization problem that is tight in the strong sense of characterizing the exact point $\ALG$ where hardness occurs. 
    The aforementioned hardness results for maximum independent set on $G(N,d/N)$ are tight in the sense of matching the best algorithms within a $1+o_d(1)$ factor in the limit $d\to\infty$, while there is still a constant factor gap for random $k$-SAT.
    In fact, prior to this work, all outstanding \emph{predictions} for the algorithmic threshold in any random optimization problem have only matched the best algorithms within a $1+o_d(1)$ factor in the large-degree limit.
    Consequently we believe that the branching OGP elucidates the fundamental reason for algorithmic hardness and may provide a framework for exact algorithmic thresholds in other problems.
\end{remark}
\begin{remark}
The significance of ultrametricity in mean-field spin glasses began with \cite{parisi1979infinite} and has played an enormous role in guiding the mathematical understanding of the low temperature regime in works such as \cite{ruelle1987mathematical,panchenko2013parisi,jagannath2017approximate,chatterjee2019average}. Ultrametricity also appears naturally in the context of optimization algorithms. Indeed in \cite[Remark 6]{subag2018following}, \cite[Section 3.4]{alaoui2020algorithmic} and \cite[Theorem 4]{sellke2021optimizing} it was realized that the aforementioned algorithms achieving asymptotic energy $\ALG$ are capable of more. Namely, they can construct arbitrary ultrametric constellations of solutions (subject to a suitable diameter upper bound), each with energy $\ALG$. Our proof via branching OGP establishes a sharp converse --- the existence of essentially arbitrary ultrametric configurations at a given energy level is \emph{equivalent} to achievability by Lipschitz $\cA$.

The aforementioned results on ultrametricity in \cite{ruelle1987mathematical,panchenko2013parisi,jagannath2017approximate,chatterjee2019average} state that the Gibbs measure $e^{\beta H_N(\bsig)}/Z~\de \bsig$ is, very roughly speaking, supported on an ultrametric subset $S$ of $S_N$ or $\Sigma_N$. For large $\beta$, this Gibbs measure describes the typical near maxima of $H_N$. However, the pairwise overlaps in $S$ may not cover the entire interval $[0,1]$, which means that $S$ is highly disconnected. By contrast, the ultrametric structures we link with algorithms are forced to branch continuously, which implies that the pairwise overlaps are dense in $[0,1]$. The condition that the Gibbs measure is supported on a continuously branching tree is a strong form of \emph{full replica symmetry breaking}. It was under such a condition that the works \cite{mon18,subag2018following,ams20} gave algorithms achieving the value $\OPT$.
\end{remark}

\begin{remark}
\label{rem:subag-bopg-1}
Since the algorithm of Subag in \cite{subag2018following} uses the top eigenvector of the Hessian $\nabla^2 H_N(\bx)$ for various $\bx\in B_N$, it is not Lipschitz in $H_N$ in the sense we require. However a different branching OGP argument shows that a stylized class of algorithms which includes a natural variant of Subag's approach is also incapable of achieving energy $\ALG+\eps$. This argument uses only a single Hamiltonian, constructing a branching tree structure using the internal randomness of the algorithm. In this sense, it bears resemblance to the original OGP analysis of \cite{gamarnik2014limits}. An outline is given in Subsection~\ref{subsec:subag-bopg-2}.
\end{remark}

\subsection{On Algorithmic Signatures of Hardness}
\label{subsec:alg-signatures}

While it has long been believed that algorithmic hardness in random optimization problems is caused by a transition in the solution geometry, the precise geometric phenomenon giving rise to hardness has been the subject of much debate.
A popular belief has been that hardness is caused by a clustering transition.
Indeed, the influential work \cite{achlioptas2008phasetransitions} shows that in random $k$-SAT and $q$-coloring, the maximal constraint density where algorithms succeed coincides (up to leading order in the $k,q\to\infty$ limit) with a \emph{shattering} phase transition.
The intuition justifying this belief was that in the shattered regime, the solution geometry becomes rugged, meaning that local search and potentially other algorithms fail.
A related conjecture was put forward in \cite{krzakala2007gibbs}, that in random CSPs local Markov chains fail above a different \emph{clustering} (or \emph{dynamic RSB}) threshold.

However, for random CSPs with bounded typical degree, it is known that algorithms succeed at constraint densities beyond the clustering transition \cite{achlioptas2003almost, zdeborova2007phase}. 
Moreover, it was later observed that in random perceptron models, neither clustering nor shattering coincides with hardness!
Indeed, \cite{baldassi2015subdominant} empirically demonstrated an algorithm that finds solutions even when (according to physics heuristics) the overall solution space is dominated by well-separated isolated solutions, i.e. clusters of size one; they conjecture that algorithms find rare connected clusters of solutions.
As rigorous evidence for this perspective, for the symmetric binary perceptron \cite{perkins2021frozen, abbe2022proof} proved the isolated solutions phenomenon and \cite{abbe2022binary} gives an algorithm to construct a cluster of solutions with macroscopic diameter when the clause density is small. Still, it was not clear even heuristically what type of solution cluster should correspond to computational tractability.

For the spin glass models we consider, our results confirm that the the signature of algorithmic hardness is not clustering properties of typical solutions, but the existence of special dense clusters.
Furthermore, we show that \emph{the relevant dense clusters are precisely ``everywhere-branching" ultrametric trees}. 
We expect that this characterization generalizes to other random optimization problems.

\subsection{Related Work}

Several previous works have studied the computational complexity of optimizing spin glass Hamiltonians. First, in the worst case over the disorder $\Gp{p}$, achieving any constant approximation ratio to the true maximum value is known to be quasi-NP hard even for degree $2$ polynomials \cite{arora2005non,barak2012hypercontractivity}. For the Sherrington-Kirkpatrick model with $\xi(t)=t^2/2$ on the cube, it was recently shown to be NP-hard on average to compute the \emph{exact} value of the partition function \cite{gamarnik2021computing}. Of course, these computational hardness results demand much stronger guarantees than the approximate optimization with high probability that we consider.

Another important line of work, alluded to above, has studied the \emph{landscape complexity} of $H_N$ on the sphere, defined as the exponential growth rate for the number of local optima and saddle points of finite-index at a given energy level. These are understood to serve as barriers to efficient optimization, and a non-rigorous study was undertaken in \cite{crisanti2003complexity,crisanti2005complexity,parisi2006computing} followed by a great deal of recent progress in \cite{auffinger2013random,auffinger2013complexity,subag2017complexity,mckenna2021complexity,kivimae2021ground,subag2021concentration}. Notably because the true maximum value of $H_N$ is nothing but its largest critical value, the first moment results of \cite{auffinger2013random} combined with the second moment results of \cite{subag2017complexity} gave an alternate self-contained proof of the Parisi formula for the ground state in pure spherical models. In a related spirit, \cite{chatterjee2009disorder,ding2015multiple,chen2017parisi,chen2018energy} have shown that mixed even $p$-spin Hamiltonians typically contain exponentially many well-separated near-global maxima.

Other works such as \cite{cugliandolo1994out,bouchaud1998out,arous2006cugliandolo,arous2020bounding,celentano2021high} have studied natural algorithms such as Langevin and Glauber dynamics on short (independent of $N$) time scales. These approaches yield (often non-rigorous) predictions for the energy achieved after a fixed amount of time. However these predictions involve complicated systems of differential equations, and to the best of our knowledge it is not known how to cleanly describe the long-time limiting energy achieved. Let us also mention the recent results of \cite{eldan2021spectral,anari2021entropic} showing that the Glauber dynamics for the Sherrington-Kirkpatrick model mix rapidly at high temperature. By contrast the problem of optimization considered in this work is related to the \emph{low} temperature behavior of the model.


From a geometric point of view, our requirement that $\cA$ be Lipschitz resembles the setting of Lipschitz selection \cite{shvartsman1984lipshitz,przeslawski1995lipschitz,shvartsman2002lipschitz,fefferman2018sharp}. Here one is given a metrized family $\cS$ of subsets inside a metric space $X$. The goal is to find a function $f:\cS\to X$ with the \emph{selector} property that $f(S)\in S$ for all $S\in \cS$, and such that $f$ has a small Lipschitz constant. Indeed a Lipschitz function $\cA:\sH_N\to C_N$ achieving energy $E$ is almost the same as a Lipschitz selector for the super-level sets $S_E(H_N)$
metrized by the norm on $\sH_N$ defined above (and leaving aside the fact that $S_E(H_N)$ may not determine $H_N$).
Of course we can only hope for $\cA(H_N)\in S_E(H_N)$ to hold with high probability, since $S_E(H_N)$ is empty with small but positive probability for each $N$.

Finally, in the large degree limit, the maxima of random constraint satisfaction problems such as random max-$k$-SAT and MaxCut are known to be described by mean-field spin glasses \cite{dembo2017extremal, panchenko2018k}, see also \cite{alaoui2021local} for an algorithmic analogue.

%% file: tex/2-results.tex
\section{The Optimal Energy of Overlap Concentrated Algorithms}
\label{sec:results}

\subsection{Overlap Concentrated Algorithms}

For any $p\in [0,1]$, we may construct two correlated copies $\HNp{1}, \HNp{2}$ of $H_N$ as follows.
Construct three i.i.d. Hamiltonians $\wtH_N^{[0]}, \wtH_N^{[1]}, \wtH_N^{[2]}$ with mixture $\xi$, as in \eqref{eq:def-hamiltonian-no-field}.
For $i=1,2$, let 
\begin{align*}
    \HNp{i}(\bsig) &= \la \bh, \bsig \ra + \wtHNp{i}(\bsig), \quad \text{where}\\
    \wtHNp{i}(\bsig) &= \sqrt{p} \wtH_N^{[0]} + \sqrt{1-p} \wtH_N^{[i]}.
\end{align*}
We say the pair of Hamiltonians $\HNp{1}, \HNp{2}$ is $p$-correlated.
Note that pairs of corresponding entries in $\bgp{1} = \bg(\HNp{1})$ and $\bgp{2} = \bg(\HNp{2})$ are Gaussian with covariance $\lt[\begin{smallmatrix}1 & p \\ p & 1\end{smallmatrix}\rt]$.

We will determine the maximum energy attained by algorithms $\cA_N : \sH_N \to B_N$ or $\cA_N : \sH_N \to C_N$ (always assumed to be measurable) obeying the following overlap concentration property.

\begin{definition}
    Let $\lambda, \nu > 0$. 
    An algorithm $\cA = \cA_N$ is $(\lambda,\nu)$ overlap concentrated if for any $p\in [0,1]$ and $p$-correlated Hamiltonians $\HNp{1}, \HNp{2}$,
    \begin{equation}
        \label{eq:overlap-concentrated}
        \P
        \lt[
            \lt| 
                R \lt( \cA(\HNp{1}), \cA(\HNp{2}) \rt) -
                \E R \lt( \cA(\HNp{1}), \cA(\HNp{2}) \rt)
            \rt|\geq \lambda
        \rt]
        \le \nu.
    \end{equation}
\end{definition}

\subsection{The Spherical Zero-Temperature Parisi Functional}

We introduce a Parisi functional $\Par^{\Sp}$ for the spherical setting, analogous to the Parisi functional $\Par^{\Is}$ for the Ising setting introduced in \eqref{eq:def-parisi-functional-is}.
Similarly to Theorem~\ref{thm:opt-is}, Auffinger and Chen \cite{auffinger2017energy}, see also \cite{chen2017parisi}, characterize the ground state energy of the spherical spin glass by a variational formula in terms of this Parisi functional.
Recall the set $\cuU$ defined in \eqref{eq:def-cuU}.
Let
\[
    \cuV(\xi) = \lt\{
        (B, \zeta) \in \bbR^+ \times \cuU :
        B > \int_0^1 \xi''(t)\zeta(t) \diff{t}
    \rt\}.
\]
Define the spherical Parisi functional $\Par^{\Sp} = \Par^{\Sp}_{\xi,h} : \cuV(\xi) \to \bbR$ by
\begin{equation}
    \label{eq:def-parisi-functional-sp}
    \Par^{\Sp}(B, \zeta) =
    \fr12 \lt[
        \fr{h^2}{B_\zeta(0)} + \int_0^1 \lt(\fr{\xi''(t)}{B_\zeta(t)} + B_\zeta(t)\rt) \diff{t}
    \rt],
\end{equation}
where for $t\in [0,1]$
\begin{equation}
    \label{eq:def-B}
    B_\zeta(t) = B - \int_t^1 \xi''(q)\zeta(q) \diff{q}.
\end{equation}

\begin{theorem}[{\cite[Theorem 10]{auffinger2017energy}}]
    The following identity holds.
    \begin{equation}
        \label{eq:opt-variational-formula-sp}
        \OPT^{\Sp} =
        \inf_{(B, \zeta) \in \cuV(\xi)}
        \Par^{\Sp} (B, \zeta).
    \end{equation}
    The infimum is attained at a unique $(B_*, \zeta_*) \in \cuV(\xi)$.
\end{theorem}

\subsection{Main Results}

We defined $\ALG^{\Is}$ in \eqref{eq:def-alg-is} by a non-monotone extension of the variational formula in \eqref{eq:opt-variational-formula-is}.
We can similarly define $\ALG^{\Sp}$ by a non-monotone extension of \eqref{eq:opt-variational-formula-sp}.
Recall the set $\cuL$ defined in \eqref{eq:def-cuL}. 
Let $\cuK(\xi) \supseteq \cuV(\xi)$ denote the set
\[
    \cuK(\xi) = \lt\{
        (B, \zeta) \in \bbR^+ \times \cuL :
        B > \int_0^1 \xi''(t)\zeta(t) \diff{t}
    \rt\}.
\]
The Parisi functional $\Par^{\Sp}$ can clearly be defined on $\cuK(\xi)$.
We define $\ALG^{\Sp} = \ALG^{\Sp}_{\xi, h}$ by
\begin{equation}
    \label{eq:def-alg-sp}
    \ALG^{\Sp} =
    \inf_{(B, \zeta) \in \cuK(\xi)}
    \Par^{\Sp} (B, \zeta).
\end{equation}
Note that $\ALG^{\Sp} \le \OPT^{\Sp}$ trivially.

We are now ready to state the main result of this work. 
We will show that for any mixed even spherical or Ising spin glass, no overlap concentrated algorithm can attain an energy level above the algorithmic thresholds $\ALG^{\Sp}$ and $\ALG^{\Is}$ with nontrivial probability.

\begin{theorem}[Main Result]
    \label{thm:main}
    Consider a mixed even Hamiltonian $H_N$ with model $(\xi,h)$. 
    Let $\ALG = \ALG^{\Sp}$ (resp. $\ALG^{\Is}$). 
    For any $\eps>0$ there are $\lambda,c,N_0>0$ depending only on $\xi,h, \eps$ such that the following holds for any $N\geq N_0$ and any $\nu\in [0,1]$. 
    For any $(\lambda,\nu)$ overlap concentrated $\cA = \cA_N : \sH_N \to B_N$ (resp. $C_N$),
    \[
        \P \lt[
            \fr{1}{N} 
            H_N\lt(\cA(H_N)\rt)
            \ge
            \ALG + \eps
        \rt]
        \le
        \exp(-cN) + 3(\nu/\lambda)^c.
    \]
\end{theorem}

\begin{remark}
    If $\cA$ is $\tau$-Lipschitz, $(\lambda,\nu)$ overlap concentration holds with $\nu = \exp(-c_{\lambda,\tau}N)$ by concentration of measure on Gaussian space, see Proposition~\ref{prop:lp-overlap-conc}. 
    Hence in this case the probability on the right-hand side above is exponentially small in $N$.
    The same property holds when $\cA$ is $\tau$-Lipschitz on a set of inputs with $1-\exp(-\Omega(N))$ probability, see Proposition~\ref{prop:approx-lp-overlap-conc}.
\end{remark}

In tandem with Theorem~\ref{thm:ams20} and its spherical analogue Theorem~\ref{thm:ams20sphere} below, Theorem~\ref{thm:main} exactly characterizes the maximum energy attained by overlap concentrated algorithms (again with the caveat on the algorithmic side in the Ising case that a minimizer $\gamma_*\in\cuL$ exists in Theorem~\ref{thm:ams20}).
We will see in Section~\ref{sec:overlap-conc-of-algs} that the algorithms in these two theorems are overlap concentrated.

\begin{theorem}[\cite{ams20,sellke2021optimizing}]
    \label{thm:ams20sphere}
    For any $\eps>0$, there exists an efficient and $O_{\eps}(1)$-Lipschitz AMP algorithm $\cA:\sH_N\to B_N$ such that
    \[
        \P[H_N(\cA(H_N))/N\geq \ALG^{\Sp}-\eps]\geq 1-\exp(-cN),\quad c=c(\eps)>0.
    \] 
\end{theorem}


In the case of the spherical spin glass, the value of $\ALG^{\Sp}$ is explicit, and is given by the following proposition.
We will prove this proposition in Appendix~\ref{sec:alg-sp-value}.
\begin{proposition}
    \label{prop:alg-sp-value}
    If $h^2 + \xi'(1) \ge \xi''(1)$, then
    \[
        \ALG^{\Sp}
        =
        (h^2 + \xi'(1))^{1/2},
    \]
    and the infimum in \eqref{eq:def-alg-sp} is uniquely attained by $B = (h^2 + \xi'(1))^{1/2}$, $\zeta = 0$.
    Otherwise, 
    \[
        \ALG^{\Sp}
        = 
        \hq \xi''(\hq)^{1/2} + 
        \int_{\hq}^1 \xi''(q)^{1/2} \diff{q}
    \]
    where $\hq\in [0,1)$ is the unique number satisfying $h^2 + \xi'(\hq) = \hq\xi''(\hq)$. 
    If $h > 0$, the infimum in \eqref{eq:def-alg-sp} is uniquely attained by $B = \xi''(1)^{1/2}$ and
    \begin{equation}
        \label{eq:best-zeta}
        \zeta(q) 
        = 
        \ind\{q \ge \hq\} \fr{\xi'''(q)}{2\xi''(q)^{3/2}}
        =
        -\ind\{q \ge \hq\} \fr{\diff{}}{\diff{q}} \xi''(q)^{-1/2}.
    \end{equation}
    If $h = 0$, the infimum is not attained.
    It is achieved by $B = \xi''(1)^{1/2}$ and $\zeta$ given by \eqref{eq:best-zeta} in the limit as $\hq \to 0^+$.\footnote{When $h=0$, we cannot take $\hq = 0$ in \eqref{eq:best-zeta} because then $B = \int_0^1 \xi''(q)\zeta(q)\diff{q}$, so $(B, \zeta) \not\in \cuK(\xi)$.}
\end{proposition}

Note that $\ALG^{\Sp} = \OPT^{\Sp}$ if and only if the infimum in \eqref{eq:def-alg-sp} is attained at a pair $(B, \zeta) \in \cuV(\xi)$.
Thus, Proposition~\ref{prop:alg-sp-value} implies that $\ALG^{\Sp} = \OPT^{\Sp}$ if and only if $h^2 + \xi'(1) \ge \xi''(1)$ or $\xi''(q)^{-1/2}$ is concave on $[\hq, 1]$. 
In the former case, the model is replica symmetric at zero temperature; in the latter case it is full replica symmetry breaking on $[\hq, 1]$ at zero temperature. 
Interestingly, in the case $h^2 + \xi'(1) > \xi''(1)$, \cite{fyodorov2013high,belius2021triviality} showed that $H_N$ has ``trivial complexity'': no critical points on $S_N$ with high probability except for the unique global maximizer and minimizer.

In the important case of the pure $p$-spin model, with $h=0$ and $\xi(x) = x^p$ for $p\ge 4$ even,
\[
    \ALG^{\Sp}
    =
    \int_0^1 \xi''(q)^{1/2} \diff{q}
    = 
    2 \sqrt{\fr{p-1}{p}}.
\]
This coincides with the threshold $E_\infty(p)$ identified in \cite{auffinger2013random}.
As conjectured in \cite{auffinger2013random} and proved in \cite{subag2017complexity}, with high probability an overwhelming majority of local maxima of $H_N$ on $S_N$ have energy value $E_\infty(p) \pm o(1)$.
This suggests that it may be computationally intractable to achieve energy at least $E_\infty(p) + \eps$ for any $\eps > 0$; our results confirm this hypothesis for overlap concentrated algorithms.

\begin{remark}
    Our results generalize with no changes in the proofs to arbitrary external fields $\bh=(h_1,\ldots,h_N)$ independent of $\wt{H}_N$ -- one only needs to replace $h^2$ by $\frac{\norm{\bh}^2}{N}$ in \eqref{eq:def-parisi-functional-sp} and replace $\Phi(0,h)$ by $\frac{1}{N}\sum_{i=1}^N \Phi(0,h_i)$ in \eqref{eq:def-parisi-functional-is}. This includes for instance the natural case of Gaussian external field $\bh\sim\cN(0,I_N)$. Here $\cA$ can depend arbitrarily on $\bh$ as long as overlap concentration holds conditionally on $\bh$.
\end{remark}

\subsection{Notation and Preliminaries}

We generally use ordinary lower-case letters $(x,y,\ldots)$ for scalars and bold lower-case $(\bx,\by,\ldots)$ for vectors.
For $\bx, \by \in \bbR^N$, we denote the ordinary inner product by $\la \bx,\by \ra = \sum_{i=1}^N x_iy_i$ and the normalized inner product by $R(\bx, \by) = \fr{1}{N} \la \bx,\by\ra$.
We associate with these inner products the norms $\norm{\bx}_2^2 = \la \bx, \bx\ra$ and $\norm{\bx}_N^2 = R(\bx,\bx)$.
There is no confusion between the $\norm{\cdot}_N$ norm and the $\ell_p$ norm, which will not appear in this paper. We use the standard notations $O(\cdot), \Omega(\cdot), o(\cdot)$ to indicate asymptotic behavior in $N$.

Ensembles of scalars over an index set $\bbL$ are denoted with an arrow $(\vx,\vy,\ldots)$, and the entry of $\vx$ indexed by $u\in \bbL$ is denoted $x(u)$.
Similarly, ensembles of vectors are written in bold and with an arrow $(\vbx,\vby,\ldots)$, and the entry of $\vbx$ indexed by $u\in \bbL$ are denoted $\bx(u)$.
Sequences of scalars parametrizing these ensembles are also denoted with an arrow, for example $\vk = (k_1,\ldots,k_D)$.

We reiterate that $S_N = \{\bx \in \bbR^N : \norm{\bx}_2^2 = N\}$ and $\Sigma_N = \{-1,1\}^N$, and that $B_N = \{\bx\in \bbR^N : \norm{\bx}_2^2 \le N\}$ and $C_N=[-1,1]^N$ are their convex hulls.
The space of Hamiltonians $H_N$ is denoted $\sH_N$.
We identify each Hamiltonian $H_N$ with its disorder coefficients $(\Gp{p})_{p\in 2\bbN}$, which we concatenate into a vector $\bg = \bg(H_N)$.

For any tensor $A_p \in (\bbR^N)^{\otimes p}$, where $p\ge 1$, we define the operator norm 
\[
    \norm{A_p}_{\op}
    =
    \fr{1}{N} 
    \max_{\bsig^1, \ldots, \bsig^p \in S_N} 
    \lt|\la 
        A_p, \bsig^1 \otimes \cdots \otimes \bsig^p
    \ra\rt|.
\]
Note that when $p=1$, $\norm{A_p}_{\op} = \norm{A_p}_N$.
The following proposition shows that with exponentially high probability, the operator norms of all constant-order gradients of $H_N$ are bounded and $O(1)$-Lipschitz.
We will prove this proposition in Appendix~\ref{sec:basic-estimates}.
\begin{proposition}
\label{prop:gradients-bounded}
    For fixed model $(\xi,h)$ and $r\in [1, \sqrt{2})$, there exists a constant $c>0$, sequence $(K_N)_{N\geq 1}$ of sets $K_N\subseteq \sH_N$, and sequence of constants $(C_{k})_{k\geq 1}$ independent of $N$, such that the following properties hold.
    \begin{enumerate}
        \item $\P[H_N\in K_N]\geq 1-e^{-cN}$;
        \item If $H_N\in K_N$ and $\bx, \by\in \bbR^N$ satisfy $\norm{\bx}_N, \norm{\by}_N \le r$, then
        \begin{align}
            \label{eq:gradient-bounded}
            \norm{\nabla^k H_N(\bx)}_{\op}
            &\le 
            C_{k}, \\
            \label{eq:gradient-lipschitz}
            \norm{\nabla^k H_N(\bx) - \nabla^k H_N(\by)}_{\op}
            &\le 
            C_{k+1} \norm{\bx - \by}_N.
        \end{align}
    \end{enumerate}
\end{proposition}

\paragraph{Organization.}
The rest of the paper is structured as follows.
In Section~\ref{sec:main-proof}, we formulate Proposition~\ref{prop:uniform-multi-opt}, which establishes our main branching OGP, and prove Theorem~\ref{thm:main} assuming this proposition. 
Sections~\ref{sec:interpolation} through \ref{sec:grand-hamiltonian-is} prove Proposition~\ref{prop:uniform-multi-opt} using a many-replica version of the Guerra-Talagrand interpolation.
Section~\ref{sec:necessity-fully-branching-trees} shows that (for spherical models with $h=0$) the full strength of our branching OGP is necessary to show tight algorithmic hardness.
Section~\ref{sec:overlap-conc-of-algs} shows that approximately Lipschitz algorithms are overlap concentrated, and that natural optimization algorithms including gradient descent, AMP, and Langevin dynamics are approximately Lipschitz.

%% file: tex/3-main-proof.tex
\section{Proof of Main Impossibility Result}
\label{sec:main-proof}

In this section, we prove Theorem~\ref{thm:main} assuming Proposition~\ref{prop:uniform-multi-opt}, which establishes the main OGP.
Throughout, we fix a model $(\xi,h)$ and $\eps > 0$.
Let $H_N$ be a Hamiltonian \eqref{eq:def-hamiltonian} with model $(\xi,h)$.
Let $\lambda > 0$ be a constant we will set later, and let $\cA : \sH_N \to B_N$ (resp. $C_N$) be $(\lambda, \nu)$ overlap concentrated.

\subsection{The Correlation Function}

We define the correlation function $\chi:[0,1] \to \bbR$ by
\begin{equation}
    \label{eq:def-correlation-fn}
    \chi(p) 
    = 
    \E R \lt(
        \cA(\HNp{1}), 
        \cA(\HNp{2})
    \rt),
\end{equation}
where $\HNp{1}, \HNp{2}$ are $p$-correlated copies of $H_N$.
The following proposition establishes several properties of correlation functions, which we will later exploit.

\begin{proposition}
    \label{prop:corelation-fn-properties}
    The correlation function $\chi$ has the following properties. 
    \begin{enumerate}[label=(\roman*), ref=\roman*] 
        \item \label{itm:cor-fn-props-zeroone} For all $p\in [0,1]$, $\chi(p) \in [0,1]$. 
        \item \label{itm:cor-fn-props-increasing} $\chi$ is either strictly increasing or constant on $[0,1]$.
        \item \label{itm:cor-fn-props-sublinear} For all $p\in [0,1]$, $\chi(p) \le (1-p)\chi(0) + p\chi(1)$.
    \end{enumerate}
\end{proposition}
We call any $\chi:[0,1]\to\bbR$ satisfying the conclusions of Proposition~\ref{prop:corelation-fn-properties} a \emph{correlation function}.
\begin{proof}
    In this proof, we will write $\cA(\bg)$ to mean $\cA(H_N)$ for the Hamiltonian $H_N$ with disorder coefficients $\bg = \bg(H_N)$. 
    We introduce the Fourier expansion of $\cA$.
    For each nonnegative integer $j$, let $\He_j$ denote the $j$-th univariate Hermite polynomial.
    These are defined by $\He_0(x) = 1$ and for $n\ge 0$, 
    \[
        \He_{n+1}(x)
        =
        x\He_{n}(x)-\He_{n}'(x).
    \]
    Recall that the renormalized Hermite polynomials $\tHe_n = \fr{1}{\sqrt{n!}} \He_n$ form an orthonormal basis of $L^2(\bbR)$ with the standard Gaussian measure, i.e. they form a complete basis and satisfy
    \[
        \E_{g\sim \cN(0,1)} \tHe_n(g) \tHe_m(g) = \ind[n=m].
    \]
    For each multi-index $\alpha = (\alpha_1, \alpha_2, \ldots,)$ of nonnegative integers that are eventually zero, define the multivariate Hermite polynomial
    \[
        \tHe_\alpha(\bg) = \prod_{i} \tHe_{\alpha_i}(\bg_i),
    \]
    These polynomials form an orthonormal basis of $L^2(\bbR^{\bbN})$ with the standard Gaussian measure, see e.g. \cite[Theorem 8.1.7]{lunardi2015infinite}.
    Hence for each $1\le i\le N$, we can write
    \[
        \cA_i(\bg)
        =
        \sum_{\alpha}
        \hcA_i(\alpha) \tHe_\alpha(\bg)
        \qquad
        \text{where}
        \qquad
        \hcA_i(\alpha) = \E \lt[\cA_i(\bg) \tHe_\alpha(\bg)\rt].
    \]
    For each multi-index $\alpha$, let $|\alpha| = \sum_{i\ge 1} \alpha_i$.
    For each nonnegative integer $j$, introduce the Fourier weight
    \[
        W_j = \fr{1}{N} \sum_{i=1}^N \sum_{|\alpha| = j} \hcA_i(\alpha)^2 \ge 0.
    \]
    For $i=1,2$, let $\bgp{i} = \bg(\HNp{i})$.
    Let $T_p$ denote the Ornstein-Uhlenbeck operator.
    We compute that
    \begin{align*}
        \chi(p)
        &=
        \fr{1}{N}
        \E
        \lt\la
            \cA(\bgp{1}),
            \cA(\bgp{2})
        \rt\ra
        =
        \fr{1}{N}
        \E
        \lt\la
            \cA(\bg),
            T_{p}
            \cA(\bg)
        \rt\ra
        =
        \fr{1}{N}
        \E
        \norm{
            T_{\sqrt{p}}
            \cA(\bg)
        }_2^2 \\
        &=
        \fr{1}{N}
        \sum_{i=1}^N
        \norm{
            T_{\sqrt{p}}
            \cA_i(\bg)
        }_2^2
        =
        \fr{1}{N}
        \sum_{i=1}^N
        \sum_{\alpha}
        p^{|\alpha|}
        \hcA_i(\alpha)^2
        =
        \sum_{j\ge 0} p^j W_j.
    \end{align*}
    It is now clear that $0\le \chi(p) \le \chi(1)$.
    Since $\chi(1) = \E \norm{\cA(H_N)}_N^2 \le 1$, this proves part (\ref{itm:cor-fn-props-zeroone}).
    Part (\ref{itm:cor-fn-props-increasing}) follows because $\chi(p)$ is strictly increasing unless $W_j=0$ for all $j\ge 1$, in which case $\chi(p)$ is constant.
    Finally part (\ref{itm:cor-fn-props-sublinear}) follows since $\chi$ is manifestly convex.
\end{proof}

\subsection{Hierarchically Correlated Hamiltonians}
\label{subsec:grand-hamiltonian}

Here we define the hierarchically organized ensemble of correlated Hamiltonians that will play a central role in our proofs of impossibility.
Let $D$ be a nonnegative integer and $\vk = (k_1,\ldots,k_D)$ for positive integers $k_1,\ldots,k_D$. 
For each $0\le d \le D$, let $V_d = [k_1]\times \cdots \times [k_d]$ denote the set of length $d$ sequences with $j$-th element in $[k_j]$. 
The set $V_0$ consists of the empty tuple, which we denote $\emptyset$.
Let $\bbT(\vk)$ denote the depth $D$ tree rooted at $\emptyset$ with depth $d$ vertex set $V_d$, where $u\in V_d$ is the parent of $v\in V_{d+1}$ if $u$ is an initial substring of $v$.
For nodes $u^1,u^2\in \bbT(\vk)$, let 
\[
    u^1 \wedge u^2
    =
    \max \lt\{
        d \in \bbZ_{\ge 0}: 
        \text{$u^1_{d'} = u^2_{d'}$ for all $1\le d' \le d$}
    \rt\},
\]
where the set on the right-hand side always contains $0$ vacuously.
This is the depth of the least common ancestor of $u^1$ and $u^2$.
Let $\bbL(\vk) = V_D$ denote the set of leaves of $\bbT(\vk)$.
When $\vk$ is clear from context, we denote $\bbT(\vk)$ and $\bbL(\vk)$ by $\bbT$ and $\bbL$.
Finally, let $K = |\bbL| = \prod_{d=1}^D k_d$.

Let sequences $\vp = (p_0, p_1, \ldots, p_D)$ and $\vq = (q_0, q_1, \ldots, q_D)$ satisfy
\begin{align*}
    0 = p_0 \le p_1 \le \cdots \le p_D &= 1, \\
    0 \le q_0 < q_1 < \cdots < q_D &= 1.
\end{align*}
The sequence $\vp$ controls the correlation structure of our ensemble of Hamiltonians, while the sequence $\vq$ controls the overlap structure that we will require the inputs to these Hamiltonians to have.

We now construct an ensemble of Hamiltonians $(\HNp{u})_{u\in \bbL}$, such that each $\HNp{u}$ is marginally distributed as $H_N$ and each pair of Hamiltonians $\HNp{u^1}, \HNp{u^2}$ is $p_{u^1\wedge u^2}$-correlated.
For each $u\in \bbT$, including non-leaf nodes, let $\wtH_N^{[u]}$ be an independent copy of $\wtH_N$, generated by \eqref{eq:def-hamiltonian-no-field}.
For each $u\in \bbL$, we construct
\begin{align}
    \notag
    \HNp{u}(\bsig) &= 
    \la \bh, \bsig \ra + 
    \wtHNp{u}(\bsig), \quad \text{where} \\
    \label{eq:def-correlated-hamiltonian}
    \wtHNp{u} &= 
    \sum_{d=1}^D 
    \sqrt{p_d - p_{d-1}}\cdot
    \wtH_N^{[(u_1,\ldots,u_d)]}.
\end{align}
It is clear that this ensemble has the stated properties.
Consider a state space of $K$-tuples 
\[
    \vbsig = (\bsig(u))_{u\in \bbL} \in (\bbR^N)^K.
\]
We define a grand Hamiltonian on this state space by
\[
    \cH_N^{\vk,\vp}(\vbsig)
    \equiv
    \sum_{u\in \bbL} 
    \HNp{u}(\bsig(u)).
\]
We will denote this by $\cH_N$ when $\vk, \vp$ are clear from context.
For states $\vbsig^1, \vbsig^2 \in (\bbR^N)^K$, define the overlap matrix $R = R(\vbsig^1, \vbsig^2) \in \bbR^{K\times K}$ by
\[
    R_{u^1,u^2} = R(\bsig^1(u^1), \bsig^2(u^2))
\]
for all $u^1,u^2\in \bbL$.
We now define an overlap matrix $Q = Q^{\vk, \vq} \in \bbR^{K\times K}$; we will control the maximum energy of $\cH_N$ over inputs $\vbsig$ with approximately this self-overlap.
Let $Q$ have rows and columns indexed by $u^1,u^2\in \bbL$ and entries
\[
    Q_{u^1,u^2} = q_{u^1\wedge u^2}.
\]
Fix a point $\bm\in \bbR^N$ such that $\norm{\bm}_N^2=q_{0}$, which we will later take to be $\bm=\E[\cA(H_N)]$.
For a tolerance $\eta \in (0,1)$, define the band 
\[
    B(\bm, \eta) = 
    \lt\{
        \bsig \in \bbR^N : 
        \lt|R(\bsig, \bm) - q_0\rt| \le \eta
    \rt\}.
\]
Define the sets of points in $S_N^K$ and $\Sigma_N^K$ with self-overlap approximately $Q$ and overlap with $\bm$ approximately $q_0$ by
\begin{align*}
    \cQ^{\Sp}(Q, \bm, \eta) 
    &= 
    \lt\{
        \vbsig \in (S_N\cap B(\bm, \eta))^K : 
        \norm{R(\vbsig, \vbsig) - Q}_\infty \le \eta
    \rt\}, \\
    \cQ^{\Is}(Q, \bm, \eta) 
    &= 
    \lt\{
        \vbsig \in (\Sigma_N\cap B(\bm, \eta))^K : 
        \norm{R(\vbsig, \vbsig) - Q}_\infty \le \eta
    \rt\}.
\end{align*}

Let $\chi$ be a correlation function (recall Proposition~\ref{prop:corelation-fn-properties}).
We say $\vp=(p_0,\ldots,p_D)$ and $\vq=(q_0,\ldots,q_D)$ are \emph{$\chi$-aligned} if the following properties hold for all $0\le d\le D$. 
\begin{itemize}
    \item If $q_d \le \chi(1)$, then $\chi(p_d) = q_d$. 
    \item If $q_d > \chi(1)$, then $p_d = 1$.
\end{itemize}

The following proposition controls the expected maximum energy of the grand Hamiltonian constrained on the sets $\cQ^{\Sp}(Q, \bm, \eta)$ and $\cQ^{\Is}(Q, \bm, \eta)$, and is the main ingredient in our proof of impossibility.
We defer the proof of this proposition to Sections~\ref{sec:interpolation} through \ref{sec:grand-hamiltonian-is}.

\begin{proposition}
    \label{prop:uniform-multi-opt}
    For any mixed even model $(\xi,h)$ and $\eps > 0$, there exists a small constant $\eta_0 \in (0,1)$ and large constants $N_0, K_0 > 0$, dependent only on $\xi, h, \eps$, such that for all $N\ge N_0$ the following holds.
    
    Let $\ALG = \ALG^{\Sp}$ (resp. $\ALG^{\Is}$).
    For any correlation function $\chi$ and vector $\bm\in \bbR^N$ with $\norm{\bm}^2_N = \chi(0)$, there exist $D, \vk, \vp, \vq, \eta$ as above such that $\vp$ and $\vq$ are $\chi$-aligned, $\eta \ge \eta_0$, $K \le K_0$, and
    \[
        \fr{1}{N} 
        \E 
        \max_{\vbsig \in \cQ(\eta)}
        \cH_N(\vbsig)
        \le 
        K (\ALG + \eps),
    \]
    where $\cQ(\eta) = \cQ^{\Sp}(Q, \bm, \eta)$ (resp. $\cQ^{\Is}(Q, \bm, \eta)$).
\end{proposition}

\subsection{Extending a Branching Tree to $S_N$ and $\Sigma_N$}

To account for the possibility that $\cA$ outputs solutions in $B_N$ (resp. $C_N$) not in $S_N$ (resp. $\Sigma_N$), we will show that a branching tree of solutions in $B_N$ (resp. $C_N$) output by $\cA$ can always be extended into a branching tree of solutions in $S_N$ (resp. $\Sigma_N$), with only a small cost to the energies attained.

Consider $\chi$-aligned $\vp, \vq$ as above.
Let $\uD \le D$ be the smallest integer such that $p_{\uD} = 1$.
Define $\uvk = (k_1,\ldots,k_{\uD})$, $\uvp = (p_0, \ldots, p_{\uD})$, and $\uvq = (q_0, \ldots, q_{\uD})$.
Let $\ubbL = V_{\uD}$ denote the nodes of $\bbT$ at depth $\uD$, and let $\uK = |\ubbL| = \prod_{d=1}^{\uD} k_d$.

Consider an analogous state space of $\uK$-tuples 
\[
    \uvbsig = (\ubsig(\uu))_{\uu \in \ubbL} \in (\bbR^N)^{\uK}.
\]
Define $\uQ = \uQ^{\uvk, \uvq} \in \bbR^{\uK \times \uK}$ analogously as the matrix indexed by $\uu^1,\uu^2\in \ubbL$, where
\[
    \uQ_{\uu^1,\uu^2} = q_{\uu^1 \wedge \uu^2} \wedge \chi(1).
\]
Note that because $\vp, \vq$ are $\chi$-aligned, $q_{\uD-1} < \chi(1) \le q_{\uD}$. 
So, the right-hand side is $\chi(1)$ if $\uu^1 \wedge \uu^2 = \uD$ (i.e. $\uu^1=\uu^2$) and $q_{\uu^1\wedge \uu^2}$ otherwise.
The following sets capture the overlap structure of outputs of $\cA$.
\begin{align*}
    \ucQ^{\Sp}(\uQ, \bm, \eta) 
    &= 
    \lt\{
        \uvbsig \in (B_N\cap B(\bm, \eta))^{\uK} : 
        \norm{R(\uvbsig, \uvbsig) - \uQ}_\infty \le \eta
    \rt\}, \\
    \ucQ^{\Is}(\uQ, \bm, \eta) 
    &= 
    \lt\{
        \uvbsig \in (C_N\cap B(\bm, \eta))^{\uK}  : 
        \norm{R(\uvbsig, \uvbsig) - \uQ}_\infty \le \eta
    \rt\}.
\end{align*}
By the construction \eqref{eq:def-correlated-hamiltonian}, for each $\uu \in \ubbL$ the Hamiltonians 
\[
    \lt\{
        \HNp{u} : \text{$u\in \bbL$ is a descendant of $\uu$ in $\bbT$}
    \rt\}
\]
are equal almost surely.
Let $\HNp{\uu}$ denote any representative from this set.

We next define the condition $\Seigen$ which guarantees existence of a suitable ``extension'' $\vbsig$ of $\uvbsig = \lt(\cA(\HNp{\uu})\rt)_{\uu\in\ubbL}$. 
First, given a subset $S\subseteq [N]$, denote by $W_S$ the $|S|$ dimensional subspace spanned by the elementary basis vectors $\{ e_s: s\in S\}$. 
Below, $\lambda_{j}$ denotes the $j$-th largest eigenvalue and $(\cdot)|_{W_S}$ denotes restriction to the subspace $W_S$ as a bilinear form, or equivalently $A|_{W_S}=P_{W_S}AP_{W_S}$, where $P_{W_S}$ is the projection onto $W_S$.

\begin{definition}
    For constants $\delta$ and $K$, let $\Seigen(\delta,K)$ denote the event that both of the below hold for all $\uu \in \ubbL$.
    \begin{enumerate}
        \item $\lambda_{2K+1}\lt(\nabla^2 \HNp{\uu}(\bx)|_{W_S}\rt)\ge 0$ for all $S\subseteq [N]$ of size $|S|\geq \delta N$.
        \item $\HNp{\uu} \in K_N$, for the $K_N$ given by Proposition~\ref{prop:gradients-bounded}.
    \end{enumerate}
\end{definition}

We will use the following lemma, whose proof is deferred to Subsection~\ref{subsec:proof-of-extra-tree}.

\begin{lemma}\label{lem:extra-tree}
    Fix a model $\xi, h$, constants $\eps,\eta >0$, and $\vk,\vq$ as above. 
    Let $\delta$ be sufficiently small depending on $\xi,h,\eta,\eps$, and assume that $\Seigen(\delta,K)$ holds. 
    For any $\uvbsig\in \ucQ(\eta/2)$, there exists $\vbsig\in \cQ(\eta)$ such that
    \[
        \HNp{\uu}(\bsig(u))
        \ge
        \HNp{\uu}(\ubsig(\uu)) - N\eps
    \]
    whenever $\uu\in\ubbL$ is an ancestor of $u\in\bbL$.
\end{lemma} 

\subsection{Completion of the Proof}

We will now finish the proof of Theorem~\ref{thm:main}.
Below we give the proof in the spherical setting; the Ising case follows verbatim up to replacing $B_N$ by $C_N$ and $\ALG^{\Sp}$ by $\ALG^{\Is}$ (since $C_N\subseteq B_N$).

Let $\ALG = \ALG^{\Sp}$.
Let $\chi$ be the correlation function of $\cA$ defined in \eqref{eq:def-correlation-fn} and set $\bm = \E[\cA(H_N)]$. 
Note that $\norm{\bm}_N^2=\chi(0)$ by definition. For small $\eps/2>0$ there exist $N_0, K_0, \eta_0$ and $D, \vk, \vp, \vq, \eta, K$ as in Proposition~\ref{prop:uniform-multi-opt} such that
\begin{equation}
    \label{eq:multi-opt-ev-guarantee}
    \fr{1}{N} 
    \E 
    \max_{\vbsig \in \cQ(\eta)}
    \cH_N(\vbsig) 
    \le 
    K (\ALG + \eps/2).
\end{equation}
For $N \ge N_0$ let
\[
    \alpha_N 
    = 
    \P \lt[
        \fr{1}{N}
        H_N(\cA(H_N))
        \ge 
        \ALG + \eps
    \rt].
\]
For each $\uu\in \ubbL$, let $\ubsig(\uu) = \cA(\HNp{\uu})$, and let $\uvbsig = (\ubsig(\uu))_{\uu\in \ubbL}$. 
We define the following events, where $\delta>0$ is chosen so that Lemma~\ref{lem:extra-tree} holds with parameters $\eps/4,\eta,\vk,\vq$.
In the statement of Theorem~\ref{thm:main}, we take $\lambda = \eta_0/4 \le \eta/4$.

Define the following events.
\begin{align*}
    \Ssolve &= \lt\{
        \fr{1}{N} 
        \HNp{\uu}(\ubsig(\uu)) 
        \ge 
        \ALG + \eps
        \text{~for all~}
        \uu\in \ubbL(\uvk)
    \rt\}, \\
    \Soverlap &= \lt\{
        \uvbsig \in \ucQ(\eta/2)
    \rt\}, \\
    \Seigen &= \lt\{
        \Seigen(\delta,K)
    \rt\}, \\
    \Sogp &= \lt\{
        \fr{1}{N}
        \max_{\vbsig \in \cQ(\eta)} 
        \cH_N(\vbsig) 
        <
        K(\ALG + 3\eps/4)
    \rt\}.
\end{align*}

\begin{proposition}
    \label{prop:s-intersection-empty}
    With parameters as above,
    \[
        \Ssolve \cap \Soverlap\cap \Seigen \cap \Sogp = \emptyset.
    \]
\end{proposition}
\begin{proof}
    Suppose that the first three events hold. 
    Then $\cA$ outputs $\uvbsig\in\ucQ(\eta/2)$ such that for all $\uu \in \ubbL$,
    \[
        \HNp{\uu}(\ubsig(\uu))
        \ge 
        \ALG+\eps.
    \]
    Lemma~\ref{lem:extra-tree} now implies the existence of $\vbsig\in\cQ(\eta)$ such that for all $\uu \in \ubbL$,
    \[
        \HNp{u}(\bsig(u))
        \ge 
        \ALG+ 3\eps/4.
    \]
    This contradicts $\Sogp$.
\end{proof}

\begin{proposition}
    \label{prop:ogp-prob-bds}
    The following inequalities hold.
    \begin{enumerate}[label=(\alph*), ref=\alph*]
        \item \label{itm:ogp-prob-bd-ssolve} $\P(\Ssolve) \ge \alpha_N^K$. 
        \item \label{itm:ogp-prob-bd-soverlap} $\P(\Soverlap) \ge 1 -K^2 \nu - \fr{2K\nu}{\lambda}$.
        \item \label{itm:ogp-prob-bd-seigen} $\P(\Seigen) \ge 1 - \exp(-cN)$ for $c>0$ depending only on $\xi,h,\eps$.
        \item \label{itm:ogp-prob-bd-sogp} $\P(\Sogp) \ge 1 - 2\exp\lt(-\fr{\eps^2}{32\xi(1)}N\rt)$. 
    \end{enumerate}
\end{proposition}
We defer the proof of this proposition to after the proof of Theorem~\ref{thm:main}.

\begin{proof}[Proof of Theorem~\ref{thm:main}]
    Lemma~\ref{prop:s-intersection-empty} implies that $\P(\Ssolve) + \P(\Soverlap) +\P(\Seigen)+\P(\Sogp) \le 3$.
    Because $(K^2+2K)^{1/K}\leq 3$ for any positive integer $K$ and $\lambda<1$,
    \begin{align*}
        \alpha_N 
        &\le \lt(K^2 \nu + \fr{2K\nu}{\lambda}\rt)^{1/K} + 2\exp\lt(-\fr{\eps^2}{32K\xi(1)}N\rt)+e^{-cN/K} \\
        &\le 3\lt(\fr{\nu}{\lambda}\rt)^{1/K} + 2\exp\lt(-\fr{\eps^2}{32K\xi(1)}N\rt)+e^{-cN/K}.
    \end{align*}
    Recall that $K \le K_0$ and $K_0$ is a constant depending only on $\xi,h,\eps$.
    The proof is complete up to choosing an appropriate $c$ in Theorem~\ref{thm:main}.
\end{proof}

\subsection{Proofs of Probability Lower Bounds}

In this section, we will prove Proposition~\ref{prop:ogp-prob-bds}.
As preparation we first give two useful concentration lemmas. 
The first shows that $R(\cA(H_N),\bm)$ concentrates around $\norm{\bm}_N^2$ for overlap concentrated algorithms with $\E[\cA(H_N)]=\bm$.

\begin{lemma}
    \label{lem:overlap-with-m}
    If $\cA = \cA_N$ is $(\lambda, \nu)$ overlap concentrated and $\E [\cA(H_N)] = \bm$, then
    \begin{equation}
        \label{eq:overlap-with-m}
        \P \lt[
            \lt|R(\cA(H_N),\bm) - \norm{\bm}_N^2\rt| > 2\lambda
        \rt]
        \le \fr{2\nu}{\lambda}.
    \end{equation}
\end{lemma}

\begin{proof}
    Define the convex function $\psi(t)=(|t-\norm{\bm}_N^2|-\lambda)_+$. 
    Then by Jensen's inequality, for independent Hamiltonians $H_N$ and $H_N'$, 
    \[
        \E\lt[
            \psi\lt(R(\cA(H_N),\bm)\rt)
        \rt]
        \le
        \E\lt[
            \psi\lt(R(\cA(H_N),\cA(H_N'))\rt)
        \rt].
    \]
    Because $\cA$ is $(\lambda, \nu)$ overlap concentrated, $\psi\lt(R(\cA(H_N),\cA(H_N'))\rt) = 0$ with probability at least $1-\nu$.
    Moreover, $\psi\lt(R(\cA(H_N),\cA(H_N'))\rt) \le 2$ pointwise.
    So, 
    \[ 
        \E \lt[ 
            \psi\lt(R(\cA(H_N),\bm)\rt) 
        \rt] 
        \le 
        2\nu.
    \]
    By Markov's inequality,
    \[
        \P \lt[
            \lt|R(\cA(H_N),\bm) - \norm{\bm}_N^2\rt| > 2\lambda
        \rt]
        =
        \P\lt[
            \psi\lt(R(\cA(H_N),\bm)\rt) > \lambda
        \rt]
        \le 
        \fr{2\nu}{\lambda}.
    \]
\end{proof}

The next lemma shows subgaussian concentration for $\fr{1}{N} \max_{\vbsig \in \cQ(\eta)} \cH_N(\vbsig)$. 

\begin{proposition}
    \label{prop:multi-opt-tail}
    The random variable
    \[
        Y = 
        \fr{1}{N}
        \max_{\vbsig \in \cQ(\eta)} 
        \cH_N(\vbsig)
    \]
    satisfies for all $t\ge 0$
    \[
        \P[|Y - \E Y| \ge t]
        \le 
        2\exp\lt( 
            - \fr{Nt^2}{2K^2 \xi(1)} 
        \rt).
    \]
\end{proposition}
\begin{proof}
    For any $\vbsig \in S_N^K$, by Cauchy-Schwarz the variance of $\cH_N(\vbsig)$ is at most
    \begin{align*}
        \E \lt[
            \lt(
                \cH_N(\vbsig)-
                \E \cH_N(\vbsig)
            \rt)^2
        \rt]
        &=
        \sum_{u^1,u^2\in\bbL} \E\wtHNp{u^1}(\bsig(u^1))\E\wtHNp{u^2}(\bsig(u^2))\\
        &\le
        K\sum_{u\in \bbL} 
        \E \wtHNp{u}(\bsig(u))^2\\
        &= 
        NK^2\xi(1).
    \end{align*}
    The result now follows from the Borell-TIS inequality
    (\cite{borell1975brunn,cirel1976norms}, or see \cite[Theorem 2]{zeitouni2015gaussian}). Note that both the statement and proof of Borell-TIS hold for noncentered Gaussian processes with no modification.
\end{proof}

We now prove each part of Proposition~\ref{prop:ogp-prob-bds} in turn.

\begin{proof}[Proof of Proposition~\ref{prop:ogp-prob-bds}(\ref{itm:ogp-prob-bd-ssolve})]
    For $0\le d\le \uD$, let $1^d \in \bbT$ denote the node $(1,\ldots,1)$ with $d$ entries (so $1^0=\emptyset$ is the root of $\bbT$), and let $S_d$ be the event that $\HNp{\uu}(\ubsig(\uu)) \ge \ALG + \eps$ for all $\uu\in \ubbL$ descended from the node $1^d$.
    Let $P_d = \P[S_d]$.
    Note that $P_{\uD} = \alpha_N$. 
    We will show $P_0 \ge \alpha_N^{\uK}\ge \alpha_N^{K}$ by showing that for all $1\le d\le \uD$,
    \[
        P_{d-1} \ge P_d^{k_d}.
    \]
    The result will then follow by induction.

    Recall the construction \eqref{eq:def-correlated-hamiltonian} of the Hamiltonians $\HNp{u}$ in terms of i.i.d. Hamiltonians $(\wtH_N^{[u]})_{u\in \bbT}$.
    Conditioned on the Hamiltonians $\Omega_{d-1} = (\wtH_N^{[1^{d'}]})_{0\le d'\le d-1}$, 
    let $f_d(\Omega_{d-1})$ denote the conditional probability of $S_d$.
    Note that
    \[
        P_d = \E f_d(\Omega_{d-1}).
    \]
    By symmetry of the $k_d$ descendant subtrees of the node $1^{d-1}$,
    \[
        P_{d-1}=\E f_d(\Omega_{d-1})^{k_d}.
    \]
    Thus $P_{d-1} \ge P_d^{k_d}$ by Jensen's inequality.
\end{proof}

\begin{proof}[Proof of Proposition~\ref{prop:ogp-prob-bds}(\ref{itm:ogp-prob-bd-soverlap})]
    By definition of $\chi$, $\E R(\ubsig(\uu^1), \ubsig(\uu^2)) = \chi(p_{\uu^1\wedge \uu^2})$. 
    If $\uu^1\wedge \uu^2 < \uD$, then $p_{\uu^1\wedge \uu^2} < 1$. 
    Because $\vp, \vq$ are $\chi$-aligned, we have $\chi(p_{\uu^1\wedge \uu^2}) = q_{\uu^1\wedge \uu^2}$.
    If $\uu^1\wedge \uu^2 = \uD$, then $p_{\uu^1\wedge \uu^2} = 1$, so clearly $\chi(p_{\uu^1\wedge \uu^2}) = \chi(1)$. 
    So, in all cases, $\E R(\ubsig(\uu^1), \ubsig(\uu^2)) = \uQ_{\uu^1, \uu^2}$.
    
    Using \eqref{eq:overlap-concentrated} and a union bound over $\uu^1,\uu^2\in \ubbL$, we have
    \[
        \norm{R(\uvbsig, \uvbsig) - \uQ}_\infty
        \le \lambda
    \]
    with probability at least $1 - K^2 \nu$.
    By Lemma~\ref{lem:overlap-with-m} and a union bound, we have 
    \[
        \lt|R(\cA(\HNp{\uu}), \bm) - \norm{\bm}_N^2\rt| \le 2\lambda
    \]
    for all $\uu \in \ubbL$ with probability at least $1 - \fr{2K\nu}{\lambda}$.
    Recall that $\lambda = \eta_0/4 \le \eta/4$.
    By a final union bound, 
    \[
        \P[\uvbsig \in \ucQ(\eta/2)]
        \ge
        1-K^2 \nu - \fr{2K\nu}{\lambda}.
    \]
\end{proof}

\begin{proof}[Proof of Proposition~\ref{prop:ogp-prob-bds}(\ref{itm:ogp-prob-bd-seigen})]

We focus on a fixed $\uu\in\ubbL$. The requirements $H_N^{(u)}\in K_N$ follow from Proposition~\ref{prop:gradients-bounded}. The uniform eigenvalue lower bound follows by union bounding over subspaces $S$ and a net of points $\bx$. In fact it follows from exactly the same proof as \cite[Lemma 2.6]{sellke2020approximate} up to replacing each appearance of an eigenvalue $\lambda_i$ to $\lambda_{K+i}$.

\end{proof}

\begin{proof}[Proof of Proposition~\ref{prop:ogp-prob-bds}(\ref{itm:ogp-prob-bd-sogp})]
    By \eqref{eq:multi-opt-ev-guarantee} and Proposition~\ref{prop:multi-opt-tail} with $t = K\eps/4$,
    \begin{align*}
        \P \lt[
            \fr{1}{N}
            \max_{\vbsig \in \cQ(\eta)} 
            \cH_N(\vbsig)
            \ge K (\ALG + 3\eps/4)
        \rt]
        &\le
        \P \lt[
            \fr{1}{N}
            \max_{\vbsig \in \cQ(\eta)} 
            \cH_N(\vbsig)
            - 
            \fr{1}{N}
            \E 
            \max_{\vbsig \in \cQ(\eta)} 
            \cH_N(\vbsig)
            \ge \fr{K\eps}{4}
        \rt] \\
        &\le 
        2\exp\lt( 
            - \fr{\eps^2}{32 \xi(1)} N
        \rt).
    \end{align*}
\end{proof}

\subsection{Proof of Lemma~\ref{lem:extra-tree}}
\label{subsec:proof-of-extra-tree}

The spherical case of Lemma~\ref{lem:extra-tree} follows from \cite[Remark 6]{subag2018following} and does not require any of the axis-aligned subspace conditions. We therefore focus on the Ising case, which is a slight extension of the main result of \cite{sellke2020approximate}.

\begin{lemma}

\label{lem:increment}

Suppose $\Seigen(\delta,K)$ holds. Then for any $\bx\in [-1,1]^N$ with $||\bx||_N^2\leq 1-\delta$, any $u\in\bbL$ and any subspace $W\subseteq\R^N$ of dimension $\dim(W) \geq N-K-1$, there are mutually orthogonal vectors $\by^1,\dots,\by^K\in W\cap\bx^{\perp}$ such that for each $i\in[K]$ the following hold where $C_3$ is as in Proposition~\ref{prop:gradients-bounded}.
\begin{enumerate}
    \item $\bx+\by^i\in [-1,1]^N$.
    \label{it:in-cube}
    \item If $x_j\in \{-1,1\}$ then $y^i_j=0$.
    \label{it:orthogonal}
    \item $H_N^{(u)}(\bx+\by^i)-H_N^{(u)}(\bx)\geq 
        -\delta \norm{\by^i}_2^2.$
    \label{it:energy-gain}
    \item $\norm{\by^i}_N\leq \frac{\delta}{10C_3}.$
    \label{it:short}
    \item If $\norm{\bx}_N^2<q_d$ for some $1\leq d\leq D$, then $\norm{\bx+\by^i}_N^2\leq q_d$.
    \label{it:no-cross}
    \item At least one of the following three events holds.
    \label{it:3-events}
    \begin{enumerate}
        \item $\norm{\by^i}_N=\frac{\delta}{10C_3}$.
        \item $\bx+\by^i$ has strictly more $\pm 1$-valued coordinates than $\bx$.
        \item $||\bx||_N^2<q_d$ and $||\bx+\by^i||_N^2= q_d$ for some $1\leq d\leq D$.
    \end{enumerate}
       
\end{enumerate} 
\end{lemma}

\begin{proof}

By the Markov inequality, $\bx$ has a set $S$ of at least $(1-||\bx||_N^2) N$ coordinates not in $\{-1,1\}$. 
$\Seigen(\delta, K)$ and the Cauchy interlacing inequality imply
\[
    \lambda_K\lt(\nabla^2H_N^{(u)}(\bx)|_{W_S\cap W}\rt)\geq \lambda_{2K+1}\lt(\nabla^2H_N^{(u)}(\bx)|_{W_S}\rt)\geq 0.
\]

Let $\by^1,\dots,\by^K\in W_S(\bx)\cap W$ be a corresponding choice of orthogonal eigenvectors, each satisfying
\[
    \lt\langle \by^i,\nabla^2 H_N^{(u)}(\bx)\by^i\rt\rangle\geq 0.
\] 
Since $\by^i$ and $-\by^i$ play symmetric roles we may assume without loss of generality that $\langle\nabla H_N^{(u)}(\by),\by^i\rangle\geq 0$. Replacing $\by^i$ by $t\by^i$ for suitable $t\in[0,1]$ if needed, we may ensure that Items~\ref{it:in-cube}, \ref{it:orthogonal}, \ref{it:short}, \ref{it:no-cross}, and \ref{it:3-events} above hold.

Since $\Seigen(\delta,K)$ implies that $\norm{\nabla^3 H_N^{(u)}}_{\op}$ is uniformly bounded by $C_3$, it follows that along the line segment $\bx+[0,1] \by^i$ the Hessian of $H_N^{(u)}$ varies in operator norm by at most $\frac{\delta}{5}$. This combined with $\langle\nabla H_N^{(u)}(\bx),\by^i\rangle\geq 0$ implies
\[
    H_N^{(u)}(\bx+\by^i)\geq H_N^{(u)}(\bx)-\delta \norm{\by^i}_2^2. 
\]
This completes the proof. 
\end{proof}

\begin{proof}[Proof of Lemma~\ref{lem:extra-tree}]

Take 
\[
    \delta< \fr{\min(\eps,\eta,1-q_{D-1})^2}{16(C_1+C_3+1)}
\]
sufficiently small, where $C_1, C_3$ are given by Proposition~\ref{prop:gradients-bounded}. 
Enumerate $\uu^1,\dots,\uu^{\uK}\in\ubbL$. Assume the points $\bsig(u)$ for descendants $u\in\bbL$ of $\uu^1,\dots,\uu^{j-1}$ have already been chosen and satisfy the conclusions of Lemma~\ref{lem:extra-tree}. We show how to define the points $\bsig(u)$ for $u$ a descendant of $\uu^j$. 


From the starting point $\bx^{0,\uu^j}=\ubsig(\uu^j)$, we produce iterates $\bx^{i,v}$ for $i\in \bbN$ and $v\in \bbT$ a descendant of $\uu^j$, similarly to \cite{subag2018following} and \cite[Proof of Theorem 1]{sellke2020approximate}.
First let $d_0=d_0(\uu^j)\in [D]$ be such that $||\bx^{0,\uu^j}||_N^2\in [q_{d_0-1},q_{d_0})$, and set $\bx^{0,v}=\bx^{0,\uu^j}$ for all depth $d_0$ descendants $v$ of $\uu^j$ if $d_0>\uD$. 

Given a point $\bx^{m,v}$ with $v$ a descendant of $\uu^j$, suppose that $||\bx^{m,v}||_N^2\in (q_{|v|-1},q_{|v|}\wedge (1-\delta))$. Then take the subspace $W^{\perp}$ (which changes from iteration to iteration) to be the span of $\bm$ as well as all currently defined leaves of the exploration tree (including $\bx^{m,v}$ itself). Hence $\dim(W^{\perp})\leq K+1$ and so $\dim(W)\geq N-K-1$. (The resulting exploration tree can be constructed in arbitrary order; at any time it will have at most $K$ leaves.) 

Then there exists $\by^{m,v}$ satisfying the properties of Lemma~\ref{lem:increment} with subspace $W$ and Hamiltonian $H^{(\uu^j)}_N$. 
We update 
\[
  \bx^{m+1,v}=\bx^{m,v}+\by^{m,v},\quad v\in\bbT.
\]
However if $||\bx^{m,v}||_N^2=q_{|v|}$, then we let $v^1,\dots,v^{k_{d+1}}$ be the children of $v$ in $\bbT$ and generate $\by^{m,v^1},\dots\by^{m,v^{k_{d+1}}}$ again using Lemma~\ref{lem:increment}. We then define
\[
  \bx^{m+1,v^j}=\bx^{m,v}+\by^{m,v^j},\quad j\in [k_{d+1}].
\]
Continuing in this way, we eventually reach points $\bx^{m+1,u}$ with $||\bx^{m+1,u}||_N^2\geq (1-\delta)$ for each $u\in \bbL$; indeed the last condition of Lemma~\ref{lem:increment} ensures that this eventually occurs for each $u\in\bbL$. We set $\bx^u=\bx^{m+1,u}$. Observe that by orthogonality of $\bx^{m,v}$ and $\by^{m,v}$,
\begin{align*}
    H_N^{(u)}(\bx^{m+1,v})&\geq H_N^{(u)}(\bx^{m,v})-N\delta \norm{\by^{m,v}}_N^2\\
    &\geq H_N^{(u)}(\bx^{m,v})-N\delta \cdot \lt(\norm{\bx^{m+1,v}}_N^2-\norm{\bx^{m,v}}_N^2\rt).
\end{align*}
It follows by telescoping that (recall $\uu^j\in\ubbL$ is an ancestor of $u\in\bbL$),
\[
    H_N^{(u)}(\bx^u)\geq H_N^{(u)}(\bx^{\uu^j})-N\delta\geq H_N^{(u)}(\bx^{\uu^j})-N\eps/2.
\]
Since every update above is made orthogonally to all contemporaneous iterates, it is not difficult to see that the final iterates $(\bx^u)_{u\in \bbL}$ satisfy the following.
\begin{itemize}
    \item $R(\bx^{u},\bx^{u})\geq 1-\delta \geq q_{u\wedge u}-\fr{\eta}{2}$.
    \item If $u^1\neq u^2$ are both descendants of $\uu^j\in\ubbL$ and $u^1\wedge u^2 < d_0(\uu^j)$, then
    \[
        R(\bx^{u^1},\bx^{u^2})= R(\bx^{\uu^j},\bx^{\uu^j})\leq q_{\uD}+\fr{\eta}{2}\leq q_{u^1\wedge u^2}+\fr{\eta}{2}.
    \]
    and
    \[
        R(\bx^{u^1},\bx^{u^2})\geq q_{d_0-1}\geq q_{u^1\wedge u^2},
    \]
    hence $\lt|R(\bx^{u^1},\bx^{u^2})-q_{u^1\wedge u^2}\rt|\leq \eta/2$.
    \item Otherwise, $R(\bx^{u^1},\bx^{u^2})=q_{u^1\wedge u^2}$.
\end{itemize}
Moreover all updates were also orthogonal to $\bm$, so $|R(\bm,\bx^u)|\leq \eta/2$ for all $u\in\bbL$.

Finally, to produce outputs in $\Sigma_N$, for each $u\in\bbL$ and $i\in [N]$ we independently round the coordinate $x^u_i$ at random to $\sigma(u)_i\in \{-1,1\}$ so that $\E[\bsig(u)]=\bx^u$. It is not difficult to see that 
\[
    \P[|R(\bx^{u^1},\bx^{u^2})-R(\bsig(u^1),\bsig(u^2))|\geq \delta]\leq e^{-c(\delta)N}
\]
for each $u^1,u^2\in \bbL$, and similarly for inner products with $\bm$. We conclude that $\vbsig\in \cQ(\eta)$ holds with probability $1-e^{-c(\delta)N}$ (since $\delta\leq \eta/2$). Similarly $\norm{\bsig(u)-\bx^u}_2^2$ is an independent sum of $N$ terms each at most $1$ and has expectation at most $\delta$. It follows that
\[
    \P[\norm{\bsig(u)-\bx^u}_N\geq 2\delta^{1/2}]\leq e^{-c(\delta)N}.
\]
Now using $\Seigen$, for every $(\uu,u)\in \ubbL\times \bbL$ with $\uu$ an ancestor of $u$,
\begin{align*}
    H_N^{(u)}(\bsig(u))
    &\geq H_N^{(u)}(\bx^u)-2C_1\delta^{1/2}N\\
    & \geq H_N^{(u)}(\bx^u)-N\eps/2\\
    &\geq H_N^{(u)}(\bx^{\uu})-N\eps
\end{align*}
holds with probability $1-e^{-c(\delta)N}$. In particular, the above events hold simultaneously over all $(\uu,u)$ with probability at least $\fr12$ over the random rounding step. Hence there exists some $\vbsig$ satisfying all desired conditions. This concludes the proof.
\end{proof}

\subsection{A Different Class of Algorithms Capturing The Approach of Subag}
\label{subsec:subag-bopg-2}

The optimization algorithm of \cite{subag2018following} in the spherical setting can be summarized as follows. Starting from any $\bx^1\in B_N$ with $\norm{\bx^1}_N^2=\delta$, repeatedly compute the maximum-eigenvalue unit eigenvector $\bv^i\in \bbR^N$ of $P_{(\bx^i)^{\perp}}\nabla^2 H_N(\bx^i)P_{(\bx^i)^{\perp}}$ (the Hessian of $H_N$ at $\bx^i$ restricted to the orthogonal complement of $\bx^i$). Then, set
\begin{equation}\label{eq:subag-update}
    \bx^{i+1}=\bx^i + \bv^i\sqrt{\delta N}
\end{equation}
where the sign of $\bv^i$ is chosen depending on the gradient $\nabla H_N(\bx^i)$. By construction, $\norm{\bx^i}_N^2=i\delta$, so if $\delta^{-1}=m\in\bbN$ then $\bx^m\in S_N$. By uniformly lower bounding the maximum eigenvalue of the Hessians, \cite{subag2018following} showed that this algorithm obtains energy at least $(\ALG^{\Sp}+o_{\delta}(1))N$ as $\delta\to 0$. Because the maximum eigenvalue is a discontinuous operation, our results do not apply to Subag's algorithm.

We consider the following variant. At each $\bx^i$, let the subspace $W(\bx^i)$ be the span of the top $\lfloor \delta N\rfloor$ eigenvectors of $P_{(\bx^i)^{\perp}}\nabla^2 H_N(\bx^i)P_{(\bx^i)^{\perp}}$. Next, choose $\bv^i$ \emph{uniformly at random} from the unit sphere of $W(\bx^i)$ and update using \eqref{eq:subag-update}. This modified algorithm obeys the same guarantees as that of \cite{subag2018following} by exactly the same proof.

More generally, we define the class of $\delta$-\emph{subspace random walk} algorithms for $\delta>0$ with $\delta^{-1}=m\in\bbN$, only in the spherical setting for convenience, as follows. Given $H_N$, let $W(\bx^i)\subseteq \bbR^N$ be an arbitrary (measurable in $(H_N,\bx)$) subspace of dimension $\lfloor \delta N\rfloor$. Starting from arbitrary $\bx^1\in B_N$ with $\norm{\bx_1}_N^2=\delta$, repeatedly choose a uniformly random unit vector $\bv^i\in W(\bx^i)$ and define $\bx^{i+1}$ via \eqref{eq:subag-update}, leading to the output $\bsig=\bx^m$. Note that in contrast to elsewhere in the paper, here the output $\bx^{i+1}$ is random even given $H_N$, i.e. $\bx^{i+1}=\cA(H_N,\omega)$ for some independent random variable $\omega$. As we now outline, for $\delta\leq \delta_0(\eps)$ sufficiently small depending on $\eps$, no $\delta$-subspace random walk algorithm can achieve energy than $\ALG^{\Sp}+\eps$ with non-negligible probability.

Fixing $H_N$ and $\bx^1$, for any $j\leq m$ we may generate coupled outputs $\bsig^1,\bsig^2$ as follows. First use shared iterates $\bx^{i,1}=\bx^{i,2}=\bx^i$ for $i\leq j$ and then proceed via 
\[
    \bx^{i+1,\ell}=\bx^{i,\ell}+\bv^{i,\ell}\sqrt{\delta N},\quad \ell\in \{1,2\}
\]
for independent update sequences $(\bv^{j,1},\dots,\bv^{m-1,1})$ and $(\bv^{j,2},\dots,\bv^{m-1,2})$. Finally output $\bsig^{\ell}=\bx^{m,\ell}$. It is not difficult to see that for $N$ sufficiently large,
\[
    \mathbb P\lt[\lt|R(\bsig^1,\bsig^2)-j\delta\rt|>\eta/2 \rt]\leq e^{-cN}
\]
for some $c=c(\delta,\eta)$ thanks to the random directions of the updates $\bv^{i,\ell}$. With $\bbL$ as in the earlier part of this section, we can now construct a branching tree of outputs $\bsig(u)$ for $u\in\bbL$. As $\delta\to 0$, for appropriate $j_d=\lfloor q_d\delta^{-1}\rfloor$, the solution configuration $\vbsig$ hence constructed satisfies
\[
    \mathbb P[\vbsig\in \cQ(\eta)]\leq e^{-cN}
\]
with $\bm$ the zero vector.
Because we consider a single Hamiltonian $H_N$, we use Proposition~\ref{prop:uniform-multi-opt} with $\chi(p)\to 0$ for all $p<1$. Since the statement is uniform in $\chi$, this does not present any difficulties (we are essentially ``defining'' $\vp=1^D$ to be $\chi$-aligned with arbitrary $\vq$). Mimicking the proofs earlier in this section (including the argument in the proof of Proposition~\ref{prop:ogp-prob-bds}(\ref{itm:ogp-prob-bd-ssolve}) which now uses Jensen's inequality on the randomness of $\cA$), we obtain the following result.

\begin{theorem}
    Consider a mixed even Hamiltonian $H_N$ with model $(\xi,h)$. 
    For any $\eps>0$ there are $\delta_0,c,N_0>0$ depending only on $\xi,h, \eps$ such that the following holds for any $N\geq N_0$ and $\delta<\delta_0$.
    For any $\delta$-subspace random walk algorithm $\cA$,
    \[
        \P \lt[
            \fr{1}{N} 
            H_N\lt(\cA(H_N,\omega)\rt)
            \ge
            \ALG^{\Sp} + \eps
        \rt]
        \le
        \exp(-cN).
    \]
\end{theorem}

%% file: tex/4-interpolation.tex
\section{Guerra's Interpolation}\label{sec:interpolation}



In this section, we begin the proof of Proposition~\ref{prop:uniform-multi-opt}.
We take either $\cQ(\eta) = \cQ^{\Sp}(Q,\bm,\eta)$ or $\cQ(\eta) = \cQ^{\Is}(Q,\bm,\eta)$ (recall $Q = Q^{\vk,\vq}$); the proofs in this section apply uniformly to both cases.
The goal of this section is to use Guerra's interpolation to upper bound the constrained free energy
\[
    F_N(\cQ(\eta))
    = 
    \fr{1}{N}
    \E 
    \log 
    \int_{\cQ(\eta)}
    \exp
    \cH_N(\vbsig) 
    \diff{\mu^K}(\vbsig),
\]
where $\mu$ is a (for now) arbitrary measure on $S_N$.
In the sequel, we will take $\mu$ to be the uniform measure on $S_N$ for spherical spin glasses, and the counting measure on $\Sigma_N$ for Ising spin glasses.
We develop a bound on $F_N(\cQ(\eta))$ that holds for all $D, \vk, \vp, \vq, \eta$, and will set these variables in the sequel to prove Proposition~\ref{prop:uniform-multi-opt}.

We will control this free energy by controlling the following related free energy.
Let $\lambda \in \bbR$ be a constant we will set later.
For all $\bsig \in \bbR^N$, let $\pi(\bsig) = \bsig - \bm$.
We define the following modified grand Hamiltonian, where we add an external field $\lambda \bm$ centered at $\bm$:
\begin{align*}
    \cH_{N,\lambda}(\vbsig)
    &=
    \cH_N(\vbsig)
    + 
    \sum_{u\in \bbL} 
    \la \lambda \bm, \pi(\bsig(u)) \ra \\
    &= 
    K \la \bh, \bm\ra
    +
    \sum_{u\in \bbL} \lt[
        \la \bh + \lambda \bm, \pi(\bsig(u))\ra + 
        \wtHNp{u}(\bsig(u))
    \rt].
\end{align*}
We define the free energy
\[
    F_{N, \lambda}(\cQ(\eta))
    = 
    \fr{1}{N}
    \log 
    \E 
    \int_{\cQ(\eta)}
    \exp
    \cH_{N,\lambda}(\vbsig) 
    \diff{\mu^K}(\vbsig).
\]
Since $\cQ(\eta) \subseteq B(\bm, \eta)^K$, we have $|\cH_N(\vbsig) - \cH_{N,\lambda}(\vbsig)| \le NK |\lambda| \eta$ for all $\vbsig \in \cQ(\eta)$, and so 
\begin{equation}
    \label{eq:feta-to-fetaw}
    |F_N(\cQ(\eta)) - F_{N, \lambda}(\cQ(\eta))|
    \le 
    K|\lambda| \eta.
\end{equation}
Define the matrices $M^{\vk,\vp,1}, \ldots, M^{\vk,\vp,D} \in \bbR^{K\times K}$, whose rows and columns are indexed by $\bbL$, by
\[
    M^{\vk,\vp,d}_{u^1,u^2} 
    = 
    \ind\{u^1\wedge u^2\ge d\} 
    p_{u^1\wedge u^2}.
\]
Further, define $M^{\vk,\vp,\vq} : [q_0, 1) \to \bbR^{K\times K}$ as the piecewise constant matrix-valued function such that for $q\in [q_{d-1}, q_d)$, $M^{\vk,\vp,\vq}(q) = M^{\vk,\vp,d}$.
Define $\kappa^{\vk,\vp,\vq} : [q_0,1) \to \bbR$ by 
\[
    \kappa^{\vk,\vp,\vq}(q) 
    = 
    \fr{1}{K} \Sum(M^{\vk,\vp,\vq}(q)),
\]
where $\Sum$ denotes the sum of entries of a matrix.
Explicitly, for $q\in [q_{d-1}, q_d)$,
\begin{equation}
    \label{eq:kappa-formula}
    \kappa^{\vk,\vp,\vq}(q) 
    = 
    \sum_{j = d}^{D-1}
    \lt[
        (k_{j+1}-1) \prod_{\ell = j+2}^D k_\ell
    \rt]
    p_j
    +
    p_D.
\end{equation}
When $\vk,\vp,\vq$ are clear, we will write $M^d = M^{\vk,\vp,d}$, $M(q) = M^{\vk,\vp,\vq}(q)$ and $\kappa(q) = \kappa^{\vk,\vp,\vq}(q)$.
Consider a sequence
\[
    0 = \zeta_{-1} < \zeta_{0} < \cdots < \zeta_{D} = 1,
\]
which we identify with the piecewise constant CDF $\zeta : [q_0,1) \to [0,1]$, where for $x\in [q_d, q_{d+1})$, 
\begin{equation}
    \label{eq:discrete-zeta}
    \zeta(x) = \zeta_d,
\end{equation}
corresponding to the discrete distribution $\zeta(\{q_d\}) = \zeta_d - \zeta_{d-1}$. 
We denote by $\cM_{\vq}$ the set of such CDFs $\zeta$ for a given $\vq$.

Let $\bbT_D = \bbN^0 \cup \bbN^1 \cup \cdots \cup \bbN^D$ and for $\omega \in \bbT_D$, let $|\omega|$ denote the length of $\omega$.
Let $\emptyset$ denote the empty tuple. 
We think of $\bbT_D$ as a tree rooted at $\emptyset$, where the parent of any $\omega\neq \emptyset$ is the initial substring of $\omega$ with length $|\omega|-1$.
For $\alpha \in \bbN^D$, let $p(\alpha) = \lt((\alpha_1), (\alpha_1,\alpha_2), \ldots, (\alpha_1,\ldots,\alpha_D)\rt)$ denote the path of vertices from the root to $\alpha$, not including the root.
For $\alpha^1, \alpha^2 \in \bbN^D$, let $\alpha^1\wedge \alpha^2$ denote the depth of the least common ancestor of $\alpha^1$ and $\alpha^2$.
Recall the Ruelle cascades $(\nu_{\alpha})_{\alpha\in\bbN^D}$ corresponding to $(\zeta_0, \zeta_1, \ldots, \zeta_{D-1})$ which were introduced in \cite{ruelle1987mathematical}, see also \cite[Section 2.3]{panchenko2013sherrington}.

For each increasing $\psi : [q_0,1] \to \bbR_{\ge 0}$, we define a Gaussian process $g_\psi^{(u)}(\alpha)$ indexed by $(u, \alpha) \in \bbL \times \bbN^D$ as follows.
Generate $\veta_\emptyset \in \bbR^K$ by
\[
    \veta_\emptyset 
    = 
    (\eta_\emptyset(u))_{u\in \bbL} 
    \sim 
    \cN(0, M^1).
\]
Furthermore, for each non-root $\omega \in \bbT_D$, independently generate $\veta_\omega \in \bbR^K$ by
\[
    \veta_\omega 
    =
    (\eta_\omega(u))_{u\in \bbL} 
    \sim \cN(0, M^{|\omega|}).
\]
Then, for each $u\in \bbL$, set
\[
    g_\psi^{(u)}(\alpha) 
    = 
    \eta_\emptyset(u)
    \psi(q_0)^{1/2}
    +
    \sum_{\omega \in p(\alpha)}
    \eta_\omega(u)
    (\psi(q_{|\omega|}) - \psi(q_{|\omega|-1}) )^{1/2}.
\]
This is the centered Gaussian process with covariance 
\[
    \E 
    g_\psi^{(u^1)}(\alpha^1) 
    g_\psi^{(u^2)}(\alpha^2) 
    = 
    p_{u^1\wedge u^2} 
    \psi(
        q_{\alpha^1 \wedge \alpha^2} 
        \wedge 
        q_{u^1\wedge u^2}
    ),
\]
where for $x,y\in \bbR$, $x\wedge y = \min(x,y)$.
Generate $N$ i.i.d. copies of the process $g_{\xi'}^{(u)}(\alpha)$, which we denote $g_{\xi',i}^{(u)}(\alpha)$ for $i=1,\ldots,N$. 
Similarly, for the function
\[
    \theta(q) = (q-q_0) \xi'(q) - \xi(q) + \xi(q_0),
\]
we generate $N$ i.i.d. processes $g_{\theta,i}^{(u)}(\alpha)$ for $i=1,\ldots,N$. 
Note that for $q\in [q_0,1)$,
\[
    \theta(q) = 
    \int_{q_0}^q (\xi'(q) - \xi'(q')) 
    \diff{q'} \ge 0
    \qquad
    \text{and}
    \qquad
    \theta'(q) = (q-q_0) \xi''(q) \ge 0,
\]
so $\theta$ is nonnegative and increasing, as required.
For $t\in [0,1]$, define the interpolating Hamiltonian
\begin{align}
    \notag
    \cH_{N,\lambda,t}(\vbsig, \alpha)
    &=
    \sum_{u\in \bbL}
    \lt[
        \sqrt{t} 
        \wtHNp{u} (\bsig(u)) 
        + 
        \sqrt{1-t}
        \sum_{i=1}^N 
        g_{\xi', i}^{(u)}(\alpha) 
        \pi(\bsig(u))_i
        +
        \sqrt{t}
        \sum_{i=1}^N 
        g_{\theta, i}^{(u)}(\alpha)
    \rt] \\
    \label{eq:interpolating-hamiltonian}
    &\qquad + 
    K \la \bh, \bm\ra +
    \la \bh + \lambda \bm, \pi(\bsig(u))\ra
\end{align}
and the interpolating free energy
\[
    \varphi(t) 
    = 
    \fr{1}{N}
    \E 
    \log
    \sum_{\alpha \in \bbN^D}
    \nu_\alpha
    \int_{\cQ(\eta)}
    \exp
    \cH_{N,\lambda,t}(\vbsig,\alpha)
    \diff{\mu^K}(\vbsig).
\]
The following bound on $F_N(\cQ(\eta))$ is the main result of this section.
\begin{proposition}
    \label{prop:interpolation-common}
    The free energy $F_N(\cQ(\eta))$ is upper bounded by
    \[
        F_N(\cQ(\eta)) 
        \le 
        \varphi(0)
        - 
        \fr{K}{2}
        \int_{q_0}^1
        (q-q_0)
        \xi''(q)
        \kappa(q)
        \zeta(q) 
        \diff{q}
        + 3K^2 \xi''(1) \eta 
        + K|\lambda| \eta.
    \]
    where $\zeta : [q_0,1) \to [0,1]$ is defined in \eqref{eq:discrete-zeta}.
\end{proposition}

\begin{lemma}[Guerra's interpolation bound]
    \label{lem:gt-monotonicity}
    For all $t\in [0,1]$ and $\eta \in (0,1)$,
    \[
        \varphi'(t)
        \le
        3K^2 \xi''(1) \eta.
    \]
\end{lemma}
\begin{proof}
    Let $\la \cdot \ra_t$ denote the average with respect to the Gibbs measure on $\cQ(\eta) \times \bbN^D$ given by
    \[
        G(\vbsig, \alpha) 
        \propto 
        \nu_\alpha 
        \exp 
        \cH_{N,\lambda,t}(\vbsig, \alpha).
    \]
    By Gaussian integration by parts \cite[Lemma 1.4]{panchenko2013sherrington},
    \begin{align}
        \notag
        \varphi'(t) 
        &= 
        \fr{1}{N}
        \E 
        \lt\la
            \fr{\partial \cH_{N,\lambda,t}}{\partial t}
            (\vbsig, \alpha)
        \rt\ra_t \\
        \label{eq:varphi}
        &=
        \fr{1}{N}
        \E 
        \lt\la
            \E 
            \fr{\partial \cH_{N,\lambda,t}}{\partial t} (\vbsig^1, \alpha^1) 
            \cH_{N,\lambda,t}(\vbsig^1, \alpha^1) -
            \E 
            \fr{\partial \cH_{N,\lambda,t}}{\partial t} (\vbsig^1, \alpha^1) 
            \cH_{N,\lambda,t}(\vbsig^2, \alpha^2) 
        \rt\ra_t,
    \end{align}
    where $(\vbsig^1, \alpha^1)$ and $(\vbsig^2, \alpha^2)$ are independent samples from the Gibbs measure.
    Recall \eqref{eq:interpolating-hamiltonian}.
    For any realizations $(\vbsig^1, \alpha^1)$ and $(\vbsig^2, \alpha^2)$,
    \begin{align*}
        &\fr{2}{N} 
        \E 
        \fr{\partial \cH_{N,\lambda,t}}{\partial t}
        (\vbsig^1, \alpha^1)
        \cH_{N,t}(\vbsig^2, \alpha^2) \\
        &= 
        \sum_{u^1,u^2\in \bbL}
        p_{u^1\wedge u^2}
        \lt[
            \xi(R(\bsig^1(u^1), \bsig^2(u^2))) 
            - 
            R(\pi(\bsig^1(u^1)), \pi(\bsig^2(u^2))) 
            \xi'(q_{\alpha^1 \wedge \alpha^2} \wedge q_{u^1\wedge u^2})
            +
            \theta(q_{\alpha^1 \wedge \alpha^2} \wedge q_{u^1\wedge u^2})
        \rt] \\
        &= 
        \sum_{u^1,u^2\in \bbL}
        p_{u^1\wedge u^2}
        \big[
            \xi(R(\bsig^1(u^1), \bsig^2(u^2))) \\
            &\qquad -
            \lt(
                R(\bsig^1(u^1), \bsig^2(u^2))
                -
                R(\bsig^1(u^1), \bm)
                -
                R(\bsig^2(u^2), \bm)
                +
                R(\bm,\bm)
            \rt)
            \xi'(q_{\alpha^1 \wedge \alpha^2} \wedge q_{u^1\wedge u^2}) \\
            &\qquad +
            \theta(q_{\alpha^1 \wedge \alpha^2} \wedge q_{u^1\wedge u^2})
        \big] \\
        &= 
        \sum_{u^1,u^2\in \bbL}
        p_{u^1\wedge u^2}
        \big[
            C\lt(R(\bsig^1(u^1), \bsig^2(u^2)), q_{\alpha^1 \wedge \alpha^2} \wedge q_{u^1\wedge u^2}\rt) \\
            &\qquad +
            \lt(
                R(\bsig^1(u^1), \bm)
                + R(\bsig^2(u^2), \bm)
                - 2q_0
            \rt)
            \xi'(q_{\alpha^1 \wedge \alpha^2} \wedge q_{u^1\wedge u^2}) 
            +
            \xi(q_0)
        \big],
    \end{align*}
    where
    \begin{equation}
        \label{eq:def-C}
        C(x,y) 
        = 
        \xi(x) - \xi(y) - (x-y) \xi'(y) 
        = 
        \int_y^x \int_y^z \xi''(w) \diff{w}\diff{z}.
    \end{equation}
    Because $\bsig^1(u^1), \bsig^2(u^2) \in B(\bm, \eta)$,
    \[
        \lt|
            \lt(
                R(\bsig^1(u^1), \bm) 
                + R(\bsig^2(u^2), \bm)
                - 2q_0
            \rt)
            \xi'(q_{\alpha^1 \wedge \alpha^2} \wedge q_{u^1\wedge u^2}) 
        \rt|
        \le 
        2 \xi'(1) \eta.
    \]
    Hence using \eqref{eq:varphi} and noting that $q_{\alpha^1\wedge\alpha^1}=1$, we obtain
    \begin{align*}
        \varphi'(t)
        &\leq
        \frac{1}{2}
        \sup_{\substack{
            \vbsig^1,\vbsig^2\in\cQ(\eta) \\
            \alpha^1,\alpha^2\in \bbN^D
        }}
        \sum_{u^1,u^2\in\bbL}
        \bigg[
            C\lt(
                R(\bsig^1(u^1), \bsig^1(u^2)), 
                q_{u^1\wedge u^2}
            \rt)
            -
            C\lt(
                R(\bsig^1(u^1), \bsig^2(u^2)), 
                q_{\alpha^1 \wedge \alpha^2} \wedge q_{u^1\wedge u^2}
            \rt)
        \bigg] \\
        &\qquad
        +2K^2 \xi'(1)\eta.
    \end{align*}
    By \eqref{eq:def-C}, $0\le C(x,y) \le |x-y|^2 \xi''(1)$.
    Since $|R(\bsig^1(u^1),\bsig^1(u^2)) - q_{u^1\wedge u^2}| \le \eta$ for $\vbsig^1 \in \cQ(\eta)$,
    \[
        C\lt(
            R(\bsig^1(u^1), \bsig^1(u^2)), 
            q_{u^1\wedge u^2}
        \rt) 
        \le 
        \xi''(1) \eta^2.
    \]
    Moreover,
    \[
        C\lt(
            R(\bsig^1(u^1), \bsig^2(u^2)), 
            q_{\alpha^1 \wedge \alpha^2} \wedge q_{u^1\wedge u^2}
        \rt) 
        \ge 0.
    \]
    So, 
    \[
        \varphi'(t) 
        \le 
        \fr12 K^2 \xi''(1) \eta^2  + 2K^2 \xi'(1)\eta
        \le 
        3K^2 \xi''(1)\eta.
    \]
\end{proof}

We will now evaluate $\varphi(1)$ to complete the proof of Proposition~\ref{prop:interpolation-common}.

\begin{lemma}
    \label{lem:varphi-1}
    The following identity holds.
    \[
        \varphi(1) = 
        F_{N, \lambda}(\cQ(\eta)) + 
        \fr{K}{2} 
        \sum_{d=0}^{D-1} 
        \kappa(q_d) \zeta_d (\theta(q_{d+1}) - \theta(q_d)).
    \]
\end{lemma}
\begin{proof}
    It is clear that
    \[
        \varphi(1)
        =
        F_{N, \lambda}(\cQ(\eta)) + 
        \fr{1}{N}
        \log 
        \E 
        \sum_{\alpha \in \bbN^D}
        \nu_\alpha
        \exp
        \sum_{u\in \bbL}
        \sum_{i=1}^N
        g^{(u)}_{\theta,i}(\alpha).
    \]
    We will evaluate the last term by the recursive evaluation of Ruelle cascades.
    For $1\le d\le D$, independently generate $\vbfeta_d = (\bfeta_d(u))_{u\in \bbL} \in (\bbR^N)^K$ by generating, independently for each $1\le i\le N$,
    \[
        (\vbfeta_d)_i = (\bfeta_d(u)_i)_{u\in \bbL} \sim \cN(0, M^{d}).
    \]
    (Because $\theta(q_0) = 0$, we will not need $\vbfeta_0$, corresponding to the root $\emptyset$ of $\bbT_D$.)
    Let
    \[
        X_{D} = 
        \sum_{u\in \bbL}
        \sum_{i=1}^N
        \sum_{d=1}^D
        \bfeta_d(u)_i
        \lt(
            \theta(q_d) - \theta(q_{d-1})
        \rt)^{1/2},
    \]
    and for $0\le d\le D-1$ let
    \begin{equation}
        \label{eq:ruelle-recursion}
        X_{d} = 
        \fr{1}{\zeta_d} 
        \log 
        \bbE_d
        \exp 
        \zeta_d
        X_{d+1},
    \end{equation}
    where $\bbE_d$ denotes expectation with respect to $\vbfeta_{d+1}$.
    By properties of Ruelle cascades \cite[Theorem 2.9]{panchenko2013sherrington},
    \[
        \fr{1}{N}
        \log 
        \E 
        \sum_{\alpha \in \bbN^D}
        \nu_\alpha
        \exp
        \sum_{u\in \bbL}
        \sum_{i=1}^N
        g^{(u)}_{\theta,i}(\alpha)
        =
        \fr{1}{N}
        X_0.
    \]
    Here we use that the depth-zero term $\eta_\emptyset(u) \theta(q_0)^{1/2}$ of $g_\theta^{(u)}(\alpha)$ is zero because $\theta(q_0)=0$.
    We now evaluate $X_0$ by \eqref{eq:ruelle-recursion}.
    For each $1\le d\le D$, $\sum_{u\in \bbL}\sum_{i=1}^N\bfeta_d(u)_i$ has variance
    \[
        \E \lt(\sum_{u\in \bbL} \sum_{i=1}^N\bfeta_d(u)_i\rt)^2
        = 
        N
        \Sum(M^{d}) 
        =
        NK\kappa(q_{d-1}).
    \]
    So, 
    \begin{align*}
        \fr{1}{\zeta_d} 
        \log 
        \bbE_d
        \exp 
        \zeta_d
        \lt(\sum_{u\in \bbL} \sum_{i=1}^N \bfeta_{d+1}(u)_i\rt)
        (\theta(q_{d+1}) - \theta(q_d))^{1/2}
        &=
        \fr{1}{\zeta_d} 
        \log 
        \exp 
        \lt(
            \fr{NK}{2}
            \kappa(q_d)
            \zeta_d^2
            (\theta(q_{d+1}) - \theta(q_d))
        \rt) \\
        &= 
        \fr{NK}{2}
        \kappa(q_d)
        \zeta_d
        (\theta(q_{d+1}) - \theta(q_d)).
    \end{align*}
    A straightforward induction argument using this computation gives
    \[
        \fr{1}{N} X_0 = 
        \fr{K}{2} 
        \sum_{d=0}^{D-1} 
        \kappa(q_d) \zeta_d (\theta(q_{d+1}) - \theta(q_d)),
    \]
    completing the proof.
\end{proof}
\begin{corollary}
    \label{cor:varphi-1}
    For the distribution function $\zeta : [q_0,1) \to [0,1]$ defined in \eqref{eq:discrete-zeta}, 
    \[
        \varphi(1) = 
        F_{N, \lambda}(\cQ(\eta)) + 
        \fr{K}{2} 
        \int_{q_0}^1 (q-q_0) \xi''(q) \kappa(q) \zeta(q) \diff{q}
    \]
\end{corollary}
\begin{proof}
    On each interval $[q_d, q_{d+1})$, the functions $\kappa(q)$ and $\zeta(q)$ are constant.
    Moreover, recall that $\theta'(q) = (q-q_0) \xi''(q)$.
    The result follows from Lemma~\ref{lem:varphi-1}.
\end{proof}

\begin{proof}[Proof of Proposition~\ref{prop:interpolation-common}]
    By Lemma~\ref{lem:gt-monotonicity} and Corollary~\ref{cor:varphi-1},
    \[
        F_{N, \lambda}(\cQ(\eta))
        \le 
        \varphi(0) 
        - 
        \fr{K}{2} 
        \int_{q_0}^1 (q-q_0) \xi''(q) \kappa(q) \zeta(q) \diff{q}
        +
        3K^2 \xi''(1) \eta.
    \]
    The result follows from \eqref{eq:feta-to-fetaw}.
\end{proof}

In the following two sections, we will use Proposition~\ref{prop:interpolation-common} to upper bound $F_N(\cQ(\eta))$ in the spherical and Ising settings by estimating
\begin{equation}
    \label{eq:varphi-0}
    \varphi(0) 
    = 
    K R(\bh, \bm) + 
    \fr{1}{N} 
    \log 
    \E 
    \sum_{\alpha \in \bbN^D}
    \nu_\alpha
    \int_{\cQ(\eta)}
    \exp
    \sum_{u\in \bbL}
    \lt[
        \la \bh + \lambda \bm, \pi(\bsig(u)) \ra +
        \sum_{i=1}^N 
        g^{(u)}_{\xi,i}(\alpha)
        \pi(\bsig(u))_i
    \rt]
    \diff{\mu^K}(\vbsig).
\end{equation}
In the spherical and Ising settings, $\mu$ is respectively the uniform measure on $S_N$ and the counting measure on $\Sigma_N$. 
We denote $\varphi(0)$ in these settings by $\varphi^{\Sp}(0)$ and $\varphi^{\Is}(0)$.
We will also denote $F_N$ in these settings by $F_N^{\Sp}$ and $F_N^{\Is}$.

%% file: tex/5-grand-hamiltonian-ub-spherical.tex
\section{Overlap-Constrained Upper Bound on the Spherical Grand Hamiltonian}
\label{sec:grand-hamiltonian-sp}

In this section, we complete the proof of Proposition~\ref{prop:uniform-multi-opt} in the spherical setting.
Denote the expected overlap-constrained maximum energy of the grand Hamiltonian by
\[
    \GS_N^{\Sp}(\cQ(\eta)) = 
    \fr{1}{N} 
    \E 
    \max_{\vbsig \in \cQ(\eta)}
    \cH_N(\vbsig).
\]
Let $\ucuL$ and $\ocuL$ denote the subsets of $\cuL$ supported on $[0, q_0)$ and $[q_0, 1)$, respectively.
The function $\kappa$ defined in \eqref{eq:kappa-formula} is an element of $\ocuL$.
Moreover (recall \eqref{eq:discrete-zeta}) $\cM_{\vq}\subseteq \ocuL$.
For $\beta > 0$ and $\zeta \in \cM_{\vq}$, let $\beta \kappa \zeta \in \ocuL$ denote the pointwise product $\beta\kappa\zeta(q) = \beta \kappa(q)\zeta(q)$.
For any $\uzeta \in \ucuL$, let $\uzeta + \beta \kappa \zeta \in \cuL$ be the function 
\[
    (\uzeta + \beta \kappa \zeta)(q) = 
    \begin{cases}
        \uzeta(q) & q < q_0, \\
        \beta \kappa \zeta(q) & q \ge q_0.
    \end{cases}
\]
We will develop the following bound on $\GS_N^{\Sp}(\cQ(\eta))$ for all $D, \vk, \vp, \vq, \eta, \beta$.
\begin{proposition}
    \label{prop:spherical-multi-opt-ub}
    Let $\zeta \in \cM_{\vq}$ and $\uzeta \in \ucuL$ be arbitrary.
    Let $\beta > 0$ and suppose that $(B, \uzeta + \beta \kappa \zeta) \in \cuK(\xi)$, $B \ge \beta^{-1}$.
    There exists a constant $C$, depending only on $\xi, h$, such that for $N\ge C \log \max(K,2)$,
    \[
        \GS_N^{\Sp}(\cQ(\eta))
        \le 
        K \Par^{\Sp}(B, \uzeta + \beta \kappa \zeta) +
        CK^2 \lt(
            \beta \eta 
            + B\eta
            + \fr{\log \fr{1}{\eta}}{\beta}
            + \fr{1}{\sqrt{N}}
        \rt).
    \]
\end{proposition}
Crucially, in the input of the Parisi functional, the increasing function $\zeta$ is pointwise multiplied by $\kappa$, which (by selecting appropriate parameters $\vk, \vp, \vq$) can be arranged to decrease as rapidly as desired.
This multiplication by $\kappa$ allows us to pass from increasing functions $\zeta \in \cM_{\vq}$ to arbitrary bounded variation functions, in the sense that $\beta \kappa \zeta$ can approximate any element of $\ocuL$.
Consequently, $\uzeta + \beta \kappa \zeta$ can approximate any element of $\cuL$, and $\Par^{\Sp}(B, \uzeta + \beta \kappa \zeta)$ can be made arbitrarily close to $\ALG^{\Sp}$.
We will prove Proposition~\ref{prop:uniform-multi-opt} by setting the parameters in Proposition~\ref{prop:spherical-multi-opt-ub} such that $(B, \uzeta + \beta \kappa \zeta)$ approximates the minimizer of $\Par^{\Sp}$ and the error term is small.

Our proof of Proposition~\ref{prop:uniform-multi-opt} proceeds in three steps. 
In Subsection~\ref{subsec:spherical-fe} we use the machinery of the previous section to prove Proposition~\ref{prop:gt-spherical-ub}, an upper bound on the free energy $F^{\Sp}_N(\cQ(\eta))$.
In Subsection~\ref{subsec:fe-to-gse}, we take this bound to low temperature to prove Proposition~\ref{prop:spherical-multi-opt-ub}.
In Subsection~\ref{subsec:grand-hamiltonian-spherical-completion}, we complete the proof of Proposition~\ref{prop:uniform-multi-opt} by setting appropriate parameters in Proposition~\ref{prop:spherical-multi-opt-ub}.

\subsection{The Free Energy Upper Bound}
\label{subsec:spherical-fe}

In this subsection, we will use Proposition~\ref{prop:interpolation-common} to upper bound $F_N^{\Sp}(\cQ(\eta))$.
We take $\mu$ to be the uniform measure on $S_N$.
The main result of this subsection is the following upper bound on $F_N^{\Sp}(\cQ(\eta))$, which holds for all $D, \vk, \vp, \vq, \eta$.

\begin{proposition}
    \label{prop:gt-spherical-ub}
    Let $\zeta \in \cM_{\vq}$ and $\uzeta \in \ucuL$ be arbitrary.
    Suppose $(B, \uzeta + \kappa \zeta) \in \cuK(\xi)$, $B \ge 1$, and $N\ge 2$.
    Then, 
    \[
        F_N^{\Sp}(\cQ(\eta))
        \le 
        K \Par^{\Sp}(B, \uzeta + \kappa \zeta) 
        + 3K^2 \xi''(1) \eta
        + KB\eta.
    \]
\end{proposition}

The crux of this argument is to upper bound $\varphi^{\Sp}(0)$ so that we may apply Proposition~\ref{prop:interpolation-common}.
We equip the state space $(\bbR^N)^K$ with the natural inner product
\[
    \la \vby^1, \vby^2 \ra
    = 
    \sum_{u\in \bbL}
    \la \by^1(u), \by^2(u) \ra
\]
and norm $\norm{\vby}^2 = \la \vby, \vby \ra$.
Generate $\vbfeta_0 = (\bfeta_0(u)) \in (\bbR^N)^K$ by generating, independently for each $1\le i\le N$, 
\begin{equation}\label{eq:vbfeta0}
    (\vbfeta_0)_i 
    = 
    (\bfeta_0(u)_i)_{u\in \bbL} 
    \sim 
    \cN(0, M^{1}).
\end{equation}
Similarly, for $1\le d\le D$, independently generate $\vbfeta_d = (\bfeta_d(u))_{u\in \bbL} \in (\bbR^N)^K$ by generating, independently for each $1\le i\le N$,
\begin{equation}\label{eq:vbfetad}
    (\vbfeta_d)_i 
    = 
    (\bfeta_d(u)_i)_{u\in \bbL} 
    \sim 
    \cN(0, M^{d}).
\end{equation}
Let $\vbm = (\bm(u))_{u\in \bbL} \in (\bbR^N)^K$ and $\vbh = (\bh(u))_{u\in \bbL} \in (\bbR^N)^K$ satisfy $\bm(u) = \bm$ and $\bh(u) = \bh$ for all $u\in \bbL$.
For $\vbsig \in (\bbR^N)^K$, define $\pi(\vbsig) = \vbsig - \vbm$.
We define the following functions on $(\bbR^N)^K$. 
Let
\begin{align*}
    G_D(\vby) 
    &= 
    \log 
    \int_{\cQ(\eta)}
    \exp
    \la \vby, \pi(\vbsig)\ra
    \diff{\mu^K}(\vbsig) \\
    &= 
    - \la \vby, \vbm \ra
    + \log 
    \int_{\cQ(\eta)}
    \exp
    \la \vby, \vbsig\ra
    \diff{\mu^K}(\vbsig).
\end{align*}
and for $0\le d\le D-1$, let
\[
    G_d(\vby) = 
    \fr{1}{\zeta_d}
    \log
    \E
    \exp
    \zeta_d
    G_{d+1}\lt(
        \vby + 
        \vbfeta_{d+1} 
        (\xi'(q_{d+1}) - \xi'(q_d))^{1/2}
    \rt).
\]
By properties of Ruelle cascades, 
\[
    \varphi^{\Sp}(0) 
    = 
    \fr{1}{N}
    \E 
    G_0((\vbh + \lambda \vbm) + \vbfeta_0 \xi'(q_0)^{1/2})
    + KR(\bh, \bm).
\]
We will estimate the spherical integral $G_D$, and through it the functions $G_d$ for $0\le d\le D-1$, by comparison with a Gaussian integral. 
This step relies on the following lemma, which is a straightforward extension of \cite[Lemma 3.1]{talagrand2006spherical}; we defer the proof to the end of this section.
For $B \ge 1$, let $\nu_B$ denote the measure of $\cN(0, \fr{1}{B})$.
Let $\chi^2(d)$ denote a $\chi^2$ random variable with $d$ degrees of freedom.
\begin{lemma}
    \label{lem:compare-spherical-gaussian}
    For all $\vby \in (\bbR^N)^K$, 
    \[
        \exp G_D(\vby) 
        \le 
        \P\lt(\chi^2(N) \ge BN\rt)^{-K}
        \exp\lt(- \la \vby, \vbm \ra \rt)
        \int 
        \exp 
        \la \vby, \vbrho \ra
        \diff{\nu_B^{N\times K}(\vbrho)}.
    \]
\end{lemma}
The probability term in this lemma can be controlled by the following standard bound, whose proof we also defer.
\begin{lemma}
    \label{lem:chisq-prob-lb}
    If $B\ge 1$ and $N\ge 2$, then
    \[
        \P(\chi^2(N) \ge BN)
        \ge 
        \exp(-BN/2).
    \]
\end{lemma}
It remains to analyze the terms in Lemma~\ref{lem:compare-spherical-gaussian} involving $\vby$.
Define further 
\begin{align*}
    G'_D(\vby) 
    &= 
    -\la \vby, \vbm\ra + 
    \log \int 
    \exp
    \la \vby, \vbrho \ra
    \diff{\nu_b^K}(\vbrho)
    = 
    \fr{\norm{\vby}_2^2}{2B} 
    - 
    \la \vby, \vbm \ra, \\
    G'_d(\vby) 
    &= 
    \fr{1}{\zeta_d}
    \log
    \E
    \exp
    \zeta_d
    G_{d+1}\lt(
        \vby + 
        \vbfeta_{d+1} 
        (\xi'(q_{d+1}) - \xi'(q_d))^{1/2}
    \rt)
    \qquad
    \text{for $0\le d\le D-1$},
\end{align*}
Henceforth, suppose $N\ge 2$. 
Lemmas~\ref{lem:compare-spherical-gaussian} and \ref{lem:chisq-prob-lb} imply that
\begin{equation}
    \label{eq:varphi-0-ub}
    \varphi^{\Sp}(0)
    \le 
    \fr{1}{N}
    \E 
    G'_0((\vbh + \lambda \vbm) +\vbfeta_0 \xi'(q_0)^{1/2})
    + KR(\bh, \bm) 
    + \fr12 KB.
\end{equation}
Consider a new state space $\bbR^K$ with elements $\vy = (y(u))_{u\in \bbL}$ where $y(u)\in \bbR$, equipped with the natural inner product
\[
    \la \vy^1, \vy^2\ra 
    =
    \sum_{u\in \bbL} 
    \vy^1(u)\vy^2(u)
\]
and norm $\norm{\vy}_2^2 = \la \vy, \vy\ra$.
Generate the $\bbR^K$-valued Gaussians $\veta_0\sim \cN(0, M^{1})$ and, for $1\le d\le D$, $\veta_d\sim \cN(0, M^{d})$.
Recall that $\bh = (h,\ldots,h)$. 
Let $\bm = (m_1,\ldots,m_N)$, and let $\vone \in \bbR^K$ denote the all-1 vector.
For $1\le i\le N$, define the following functions on $\bbR^K$.
\begin{align*}
    \Gamma^i_D(\vy) 
    &= 
    \fr{\norm{\vy}^2}{2B} - m_i \la \vone, \vy\ra, \\
    \Gamma^i_d(\vy)
    &= 
    \fr{1}{\zeta_d}
    \log
    \E
    \exp
    \zeta_d
    \Gamma^i_{d+1}\lt(
        \vy + 
        \veta_{d+1} 
        (\xi'(q_{d+1}) - \xi'(q_d))^{1/2}
    \rt)
    \qquad
    \text{for $0\le d\le D-1$}.
\end{align*}
By independence of the $1\le i\le N$ coordinates in the $G'_d$, \eqref{eq:varphi-0-ub} implies
\begin{equation}
    \label{eq:varphi-0-ub-gamma}
    \varphi^{\Sp}(0)
    \le
    \fr{1}{N}
    \sum_{i=1}^N
    \E 
    \Gamma_0^i((h + \lambda m_i)\vone + \veta_0 \xi'(q_0)^{1/2})
    + KR(\bh, \bm) 
    + \fr12 KB.
\end{equation}
It remains to compute the Gaussian integrals $\Gamma^i_d$.
For this, we rely on the following lemma.
We defer the proof, which is a standard computation with Gaussian integrals. 
Let $\bbS_K$ denote the set of $K\times K$ positive definite matrices, and let $|\cdot|$ denote the matrix determinant.
\begin{lemma}
    \label{lem:gaussian-integral}
    Suppose $\zeta > 0$ and $\Lambda, \Sigma \in \bbS_K$ satisfy $\Lambda - \zeta \Sigma \in \bbS_K$. 
    If $\vv \in \bbR^K$ and $\veta \sim \cN(0, \Sigma)$, then
    \begin{align*}
        &\fr{1}{\zeta} 
        \log 
        \E 
        \exp 
        \fr12 \zeta 
        \lt[
            (\vy + \veta)^\top 
            \Lambda^{-1} 
            (\vy + \veta)
            - 
            2 \vv^\top (\vy + \veta)
        \rt] \\
        &= 
        \fr12 \lt[
            \vy^\top (\Lambda - \zeta \Sigma)^{-1} \vy
            - 
            2 \vv^\top \Lambda (\Lambda - \zeta \Sigma)^{-1} \vy
        \rt]
        +
        \fr{1}{2\zeta}
        \log \fr{|\Lambda|}{|\Lambda - \zeta \Sigma|}
        + 
        \fr{1}{2} 
        \vv^\top (\zeta \Sigma) (\Lambda - \zeta \Sigma)^{-1} \Lambda \vv.
    \end{align*}
\end{lemma}
We can compute the expectations in \eqref{eq:varphi-0-ub-gamma} by applying this lemma recursively.
Define 
\[
    \ocuK(\xi) = \lt\{
        (B, \zeta) \in \bbR^+ \times \ocuL : 
        B > \int_{q_0}^1 \xi''(q') \zeta(q') \diff{q'}
    \rt\}.
\]
\begin{proposition}
    \label{prop:gaussian-recursion}
    Let $\zeta \in \cM_{\vq}$, and suppose $(B, \kappa \zeta) \in \ocuK(\xi)$.
    Then, for $B_{\kappa \zeta}$ defined as in \eqref{eq:def-B},
    \[
        \E \Gamma^i_0((h + \lambda m_i)\vone + \veta_0 \xi'(q_0)^{1/2})
        \le 
        \fr{K}{2} \lt[
            \fr{(h + (\lambda - B) m_i)^2 + \xi'(q_0)}{B_{\kappa \zeta}(q_0)}
            +
            \int_{q_0}^1 \fr{\xi''(q)}{B_{\kappa \zeta}(q)} \diff{q}
             - Bm_i^2
        \rt].
    \]
\end{proposition}
\begin{proof}
    Let $\Lambda_D = BI_K$, and for $0\le d\le D-1$, let 
    \[
        \Lambda_d 
        = 
        \Lambda_{d+1} 
        - \zeta_d (\xi'(q_{d+1}) - \xi'(q_d)) 
        M^{d+1}.
    \]
    We will first show that $\Lambda_0,\ldots,\Lambda_D \in \bbS_K$, so that we can apply Lemma~\ref{lem:gaussian-integral}.
    For $q\in [q_0, 1]$, we define
    \[
        \Lambda(q) = BI_K - \int_{q}^1 \xi''(q') M(q')\zeta(q') \diff{q'}.
    \]
    Note that $\Lambda_d = \Lambda(q_d)$ for all $0\le d\le D$.
    Since $M(q) \preceq \kappa(q) I_K$ in the Loewner order,
    \begin{equation}
        \label{eq:lambda-r-lb}
        \Lambda(q) 
        \succeq 
        \lt(B - \int_q^1 \xi''(q') \kappa(q')\zeta(q') \diff{q'} \rt) I_K 
        = 
        B_{\kappa \zeta}(q) I_K.
    \end{equation}
    So, the hypothesis $(B, \kappa\zeta) \in \ocuK(\xi)$ implies $\Lambda(q) \in \bbS_K$ for all $q\in [q_0,1]$.
    In particular $\Lambda_0,\ldots,\Lambda_D \in \bbS_K$.
    
    Further, define $\vv_D = m_i \vone$, and for $0\le d\le D-1$, define $\vv_d = \Lambda_d^{-1} \Lambda_{d+1} \vv_{d+1}$.
    This implies that $\vv_d = Bm_i \Lambda_d^{-1} \vone$.
    We can write $\Gamma^i_D$ as 
    \[
        \Gamma^i_D(\vy) = 
        \fr12 \lt(
            \vy^\top \Lambda_D^{-1} \vy - 
            2\vv_D^\top \vy
        \rt).
    \]
    By a recursive computation with Lemma~\ref{lem:gaussian-integral} (which applies because $\Lambda_0,\ldots,\Lambda_D \in \bbS_K$), we have for all $0\le d\le D$ that
    \begin{align*}
        \Gamma_d^i(\vy) 
        &= 
        \fr12
        \lt[
            \vy^\top \Lambda_d^{-1} \vy -
            2\vv_d^\top \vy
            +
            \sum_{d' = d}^{D-1} 
            \fr{1}{\zeta_{d'}}
            \log \fr{|\Lambda_{d'+1}|}{|\Lambda_{d'}|}
            + 
            \sum_{d' = d}^{D-1} 
            \vv_{d'+1} 
            (\Lambda_{d'+1}-\Lambda_{d'}) 
            \Lambda_{d'}^{-1} 
            \Lambda_{d'+1} 
            \vv_{d'+1}
        \rt] \\
        &= 
        \fr12
        \lt[
            \vy^\top \Lambda_d^{-1} \vy -
            2Bm_i \vone^\top \Lambda_d^{-1} \vy
            +
            \sum_{d' = d}^{D-1} 
            \fr{1}{\zeta_{d'}}
            \log \fr{|\Lambda_{d'+1}|}{|\Lambda_{d'}|}
            + 
            B^2 m_i^2
            \sum_{d' = d}^{D-1} 
            \vone^\top 
            \Lambda_{d'+1}^{-1}
            (\Lambda_{d'+1}-\Lambda_{d'}) 
            \Lambda_{d'}^{-1} 
            \vone
        \rt].
    \end{align*}
    Note that
    \[
        \sum_{d' = d}^{D-1} 
        \vone^\top
        \Lambda_{d'+1}^{-1}
        (\Lambda_{d'+1}-\Lambda_{d'}) 
        \Lambda_{d'}^{-1} 
        \vone
        = 
        \sum_{d' = d}^{D-1} 
        \vone^\top
        (\Lambda_{d'}^{-1} - \Lambda_{d'+1}^{-1})
        \vone
        = 
        \vone^\top
        (\Lambda_{d}^{-1} - \Lambda_{D}^{-1})
        \vone
        = 
        \vone^\top \Lambda_d^{-1} \vone
        - \fr{K}{B}.
    \]
    So, 
    \begin{align*}
        \Gamma_0^i(\vy) 
        &= 
        \fr12
        \lt[
            \vy^\top \Lambda_0^{-1} \vy 
            - 2Bm_i \vone^\top \Lambda_0^{-1} \vy
            + B^2m_i^2 \vone^\top \Lambda_0^{-1} \vone
            +
            \sum_{d = 0}^{D-1} 
            \fr{1}{\zeta_{d}}
            \log \fr{|\Lambda_{d+1}|}{|\Lambda_{d}|}
            - KB m_i^2
        \rt] \\
        &= 
        \fr12
        \lt[
            (\vy - B m_i \vone)^{\top} \Lambda_0^{-1} (\vy - B m_i \vone)
            +
            \sum_{d = 0}^{D-1} 
            \fr{1}{\zeta_{d}}
            \log \fr{|\Lambda_{d+1}|}{|\Lambda_{d}|}
            - KB m_i^2
        \rt]
    \end{align*}
    Therefore,
    \begin{align*}
        &\E \Gamma^i_0((h + \lambda m_i)\vone + \veta_0 \xi'(q_0)^{1/2}) \\
        &=
        \fr12
        \lt[
            (h + (\lambda - B) m_i)^2 
            \Tr( \Lambda_0^{-1} \vone\vone^\top)
            +
            \xi'(q_0)
            \Tr(\Lambda_0^{-1} M^{1}) 
            + 
            \sum_{d = 0}^{D-1} 
            \fr{1}{\zeta_{d}}
            \log \fr{|\Lambda_{d+1}|}{|\Lambda_{d}|}
            - KB m_i^2
        \rt].
    \end{align*}
    By Jacobi's formula,
    \[
        \fr{\diff{}}{\diff{q}} 
        \log |\Lambda(q)|
        = 
        \xi''(q) \zeta(q) \Tr(\Lambda(q)^{-1} M(q)),
    \]
    so
    \[
        \fr{1}{\zeta_d}
        \log
        \fr{|\Lambda_{d+1}|}{|\Lambda_d|}
        =
        \int_{q_d}^{q_{d+1}} 
        \xi''(q) 
        \Tr(\Lambda(q)^{-1} M(q)) 
        \diff{q}.
    \]
    Therefore,
    \begin{align*}
        &\E \Gamma^i_0((h + \lambda m_i)\vone + \veta_0 \xi'(q_0)^{1/2}) \\
        &=
        \fr12
        \lt[
            (h + (\lambda - B) m_i)^2 
            \Tr(\Lambda(q_0)^{-1} \vone\vone^\top)
            +
            \xi'(q_0)
            \Tr(\Lambda(q_0)^{-1} M(q_0)) 
            + 
            \int_{q_0}^1 
            \Tr(\Lambda(q)^{-1} M(q)) 
            \diff{q}
            - KB m_i^2
        \rt].
    \end{align*}
    Finally, for each $q\in [q_0, 1)$, (\ref{eq:lambda-r-lb}) implies $\Lambda(q)^{-1} \preceq \fr{I_K}{B_{\kappa \zeta}(q)}$, so 
    \[
        \Tr(\Lambda(q)^{-1}M(q)) 
        \le 
        \Tr\lt(\fr{M(q)}{B_{\kappa \zeta}(q)}\rt) 
        = 
        \fr{K}{B_{\kappa \zeta}(q)},
    \]
    and similarly $\Tr(\Lambda(q_0)^{-1} \vone \vone^\top) \le \fr{K}{B_{\kappa \zeta}(q_0)}$.
    This implies the result.
\end{proof}
Proposition~\ref{prop:gaussian-recursion} and \eqref{eq:varphi-0-ub-gamma} readily imply the following bound on $F_N^{\Sp}(\cQ(\eta))$.
\begin{proposition}
    \label{prop:fe-ub}
    Let $B \ge 1$, $N \ge 2$, and $\lambda \in \bbR$.
    Let $\zeta \in \cM_{\vq}$, and suppose $(B, \kappa \zeta) \in \ocuK(\xi)$. 
    Then,
    \begin{align*}
        F_N^{\Sp}(\cQ(\eta)) 
        &\le 
        \fr{K}{2} \lt[
            \fr{\norm{\bh + (\lambda - B)\bm}_N^2 + \xi'(q_0)}{B_{\kappa \zeta}(q_0)}
            + 2R(\bh, \bm) 
            + \int_{q_0}^1 
            \lt(\fr{\xi''(q)}{B_{\kappa \zeta}(q)} + B_{\kappa \zeta}(q)\rt) 
            \diff{q}
        \rt] \\
        &\qquad + 3K^2 \xi''(1) \eta + K|\lambda| \eta.
    \end{align*}
\end{proposition}
\begin{proof}
    By averaging Proposition~\ref{prop:gaussian-recursion} over $1\le i\le N$, we get
    \[
        \fr{1}{N} 
        \sum_{i=1}^N 
        \E \Gamma^i_0((h + \lambda m_i)\vone + \veta_0 \xi'(q_0)^{1/2})
        \le 
        \fr{K}{2} \lt[
            \fr{\norm{\bh + (\lambda - B) \bm}_N^2 + \xi'(q_0)}{B_{\kappa \zeta}(q_0)}
            +
            \int_{q_0}^1 \fr{\xi''(q)}{B_{\kappa \zeta}(q)} \diff{q}
             - Bq_0
        \rt]
    \]
    where we used that $\norm{\bm}_N^2 = q_0$.
    Equation \eqref{eq:varphi-0-ub-gamma} implies that
    \[
        \varphi^{\Sp}(0)
        \le 
        \fr{K}{2} \lt[
            \fr{\norm{\bh + (\lambda - B) \bm}_N^2 + \xi'(q_0)}{B_{\kappa \zeta}(q_0)}
            +
            2R(\bh, \bm) 
            +
            \int_{q_0}^1 \fr{\xi''(q)}{B_{\kappa \zeta}(q)} \diff{q}
            + (1-q_0)B
        \rt].
    \]
    By Proposition~\ref{prop:interpolation-common}, this implies
    \begin{align*}
        F_N^{\Sp}(\cQ(\eta)) 
        &\le 
        \fr{K}{2} \bigg[
            \fr{\norm{\bh + (\lambda - B) \bm}_N^2 + \xi'(q_0)}{B_{\kappa \zeta}(q_0)}
            +
            2R(\bh, \bm) 
            +
            \int_{q_0}^1 \fr{\xi''(q)}{B_{\kappa \zeta}(q)} \diff{q}
            + (1-q_0)B \\
            &\qquad 
            - \int_{q_0}^1 
            (q-q_0)
            \xi''(q)
            \kappa(q)
            \zeta(q) \diff{q}
        \bigg]
        + 3K^2 \xi''(1) \eta + K|\lambda| \eta
    \end{align*}
    By integration by parts,
    \begin{align*}
        - \int_{q_0}^1 (q-q_0) \xi''(q) \kappa(q) \zeta(q) \diff{q}
        &= 
        (q-q_0) \int_q^1 \xi''(q') \kappa(q') \zeta(q') \diff{q'} \Big|_{q=q_0}^1
        - \int_{q_0}^1 \int_q^1 \xi''(q') \kappa(q') \zeta(q') \diff{q'} \diff{q} \\
        &= 
        \int_{q_0}^1 B_{\kappa \zeta}(q) \diff{q} -(1-q_0)B,
    \end{align*}
    which yields the result.
\end{proof}

The next lemma upper bounds our estimates for $F_N^{\Sp}(\cQ(\eta))$ in terms of the Parisi functional uniformly in $\bm$. 

\begin{lemma}
    \label{lem:spherical-off-center-worse-than-on-center}
    Let $q_0 \in [0,1]$.
    For $(B, \zeta) \in \cuK(\xi)$, $\bh = (h, \ldots, h)$, $\norm{\bm}_N^2 = q_0$, there exists $\lambda \in [0,B]$ such that
    \[
        \fr12 \lt[
            \fr{\norm{\bh + (\lambda - B)\bm}_N^2 + \xi'(q_0)}{B_{\zeta}(q_0)}
            + 2R(\bh, \bm) 
            + \int_{q_0}^1 
            \lt(\fr{\xi''(q)}{B_{\zeta}(q)} + B_{\zeta}(q)\rt) 
            \diff{q}
        \rt]
        \le 
        \Par^{\Sp}(\zeta).
    \]
\end{lemma}
\begin{proof}
    We take $\lambda = \int_0^1 \xi''(q)\zeta(q) \diff{q}$.
    The condition $(B,\zeta) \in \cuK(\xi)$ implies that $\lambda \in [0, B]$. 
    Note that $\lambda - B = -B_\zeta(0)$.
    It suffices to prove that
    \[
        \fr{\norm{\bh - B_\zeta(0) \bm}_N^2 + \xi'(q_0)}{B_\zeta(q_0)} 
        + 
        2R(\bh, \bm) 
        \le 
        \fr{\norm{\bh}_N^2}{B_{\zeta}(0)}
        +
        \int_0^{q_0}\lt(
            \fr{\xi''(q)}{B_{\zeta}(q)} + B_{\zeta}(q)
        \rt) \diff{q}.
    \]
    Note that
    \[
        \fr{\xi'(q_0)}{B_{\zeta}(q_0)}
        =
        \int_0^{q_0}
        \fr{\xi''(q)}{B_{\zeta}(q_0)}
        \diff{q}
        \le 
        \int_0^{q_0}
        \fr{\xi''(q)}{B_{\zeta}(q)}
        \diff{q}
        \qquad
        \text{and}
        \qquad
        q_0 B_{\zeta}(0)
        \le 
        \int_0^{q_0}
        B_{\zeta}(q)
        \diff{q}.
    \]
    So, it suffices to prove that
    \[
        \fr{\norm{\bh - B_\zeta(0) \bm}_N^2}{B_\zeta(q_0)}
        + 2R(\bh, \bm) 
        \le 
        \fr{\norm{\bh}_N^2}{B_{\zeta}(0)}
        + q_0 B_{\zeta}(0).
    \]
    This rearranges to (using that $\norm{\bm}_N^2 = q_0$)
    \[
        0 
        \le 
        \lt(
            \fr{1}{B_\zeta(0)} - \fr{1}{B_\zeta(q_0)}
        \rt)
        \lt(
            \norm{\bh}_N^2
            - 2B_\zeta(0) R(\bh, \bm)
            + B_\zeta(0)^2 \norm{\bm}_N^2
        \rt),
    \]
    which follows from Cauchy-Schwarz.
\end{proof}

We are now ready to prove Proposition~\ref{prop:gt-spherical-ub}.
\begin{proof}[Proof of Proposition~\ref{prop:gt-spherical-ub}]
    Recall that the restriction of $\uzeta + \kappa \zeta \in \cuL$ on $[q_0, 1)$ is $\kappa \zeta$.
    Because $(B, \uzeta + \kappa \zeta) \in \cuK(\xi)$, we have $(B,  \kappa \zeta) \in \ocuK(\xi)$, and so Proposition~\ref{prop:fe-ub} applies.
    Combining this with Lemma~\ref{lem:spherical-off-center-worse-than-on-center} applied on $(B, \uzeta + \kappa \zeta)$ gives the result. 
\end{proof}

\subsection{From Free Energy to Ground State Energy}
\label{subsec:fe-to-gse}

Next, we will prove Proposition~\ref{prop:spherical-multi-opt-ub} by taking Proposition~\ref{prop:gt-spherical-ub} to low temperature.
We introduce the following temperature-scaled free energy.
For $\beta > 0$ and $\eta \in (0, 1)$, let
\[
    F_N^{\Sp}(\beta, \cQ(\eta)) = 
    \fr{1}{N} 
    \log \E \int_{\cQ(\eta)}
    \exp \beta \cH_N(\vbsig) 
    \diff{\mu^K}(\vbsig).
\]
This free energy can be upper bounded by the following application of Proposition~\ref{prop:gt-spherical-ub}.

\begin{corollary}
    \label{cor:temp-fe-ub}
    Let $\zeta \in \cM_{\vq}$ and $\uzeta \in \ucuL$ be arbitrary.
    Let $\beta > 0$ and suppose $(B, \uzeta + \beta \kappa \zeta) \in \cuK(\xi)$, $B \ge \beta^{-1}$, and $N\ge 2$.
    Then,
    \[
        \fr{1}{\beta} F_N^{\Sp}(\beta, \cQ(\eta))
        \le 
        K \Par^{\Sp}(B, \uzeta + \beta \kappa \zeta) 
        + 3K^2 \xi''(1) \beta \eta
        + KB\eta.
    \]
\end{corollary}
\begin{proof}
    The hypothesis $(B, \uzeta + \beta \kappa \zeta) \in \cuK(\xi)$ implies $(\beta B, \beta^{-1} \uzeta + \kappa \zeta) \in \cuK(\beta^2 \xi)$.
    The hypothesis $B \ge \beta^{-1}$ implies $\beta B \ge 1$. 
    By Proposition~\ref{prop:gt-spherical-ub} with parameters $(\beta^2 \xi, \beta h)$ (corresponding to the Hamiltonian $\beta H_N$), $\zeta$, $\beta B$, and $\beta^{-1}\uzeta$,
    \begin{equation}
        \label{eq:temp-fe-ub}
        F_N^{\Sp}(\beta, \cQ(\eta)) 
        \le 
        K\Par^{\Sp}_{\beta^2 \xi, \beta h}(\beta B, \beta^{-1} \uzeta + \kappa \zeta)
        + 3K^2 \xi''(1) \beta^2\eta 
        + KB \beta \eta.
    \end{equation}
    We can verify that
    \[
        \Par^{\Sp}_{\beta^2 \xi, \beta h} (\beta B, \beta^{-1}\uzeta + \kappa \zeta)
        = 
        \beta \Par^{\Sp}_{\xi, h} (B, \uzeta + \beta \kappa \zeta).
    \]
    So, dividing \eqref{eq:temp-fe-ub} by $\beta$ gives the result.
\end{proof}

The following lemma relates the ground state energy $\GS_N^{\Sp}(\cQ(\eta))$ to this free energy at large inverse temperature $\beta$.
We defer the proof, which is a relatively standard approximation argument.

\begin{lemma}
    \label{lem:fn-to-opt}
    There exists a constant $C$ depending only on $\xi, h$ such that for all $\beta > 0$, $\eta \in (0, \fr12)$, and $N\ge C \log \max(K,2)$,
    \[
        \GS_N^{\Sp}(\cQ(\eta)) 
        \le 
        \fr{1}{\beta} F_N^{\Sp}(\beta, \cQ(2\eta))
        + 
        CK \lt(
            \eta 
            + \fr{\log \fr{1}{\eta}}{\beta}
            + \fr{1}{\sqrt{N}}
        \rt).
    \]
\end{lemma}

\begin{proof}[Proof of Proposition~\ref{prop:spherical-multi-opt-ub}]
    Let $C$ be large enough that Lemma~\ref{lem:fn-to-opt} is satisfied and $C \log 2 \ge 2$.
    For all $N\ge C \log \max(K, 2)$, Corollary~\ref{cor:temp-fe-ub} (with $2\eta$ in place of $\eta$) and Lemma~\ref{lem:fn-to-opt} imply that
    \[
        \GS_N^{\Sp}(\eta) 
        \le 
        K \Par^{\Sp}(B, \uzeta + \beta \kappa \zeta) 
        + 6K^2 \xi''(1) \beta \eta
        + 2KB\eta
        + CK \lt(
            \eta 
            + \fr{\log \fr{1}{\eta}}{\beta}
            + \fr{1}{\sqrt{N}}
        \rt).
    \]
    By applying the estimate $K\le K^2$ and absorbing constants depending on only $\xi, h$ into $C$, we deduce
    \[
        \GS_N^{\Sp}(\eta) 
        \le 
        K \Par^{\Sp}(B, \uzeta + \beta \kappa \zeta) 
        + CK^2 \lt(
            \beta \eta 
            + B\eta 
            + \eta 
            + \fr{\log \fr{1}{\eta}}{\beta}
            + \fr{1}{\sqrt{N}}
        \rt).
    \]
    Finally, because $B\ge \beta^{-1}$, we have $\beta + B \ge \beta + \beta^{-1} \ge 2$, so by increasing the constant $C$ we may drop the term $\eta$ from the sum.
\end{proof}

\subsection{Proof of the Main Upper Bound}
\label{subsec:grand-hamiltonian-spherical-completion}

We now complete the proof of Proposition~\ref{prop:uniform-multi-opt}.
We will set the parameters of Proposition~\ref{prop:spherical-multi-opt-ub} such that $(B, \uzeta + \beta \kappa \zeta)$ approximates the minimizer of $\Par^{\Sp}$ in $\cuL$ and the error term is small.

For $\zeta \in \cuL$ and $\delta, x \in [0,1)$, we define a perturbation $\zeta_{\delta,x} \in \cuL$ of $\zeta$ by
\[
    \zeta_{\delta,x}(q) = 
    \begin{cases}
        \zeta(x+\delta) & q\in [x, x+\delta), \\
        \zeta(q) & \text{otherwise}.
    \end{cases}
\]
Note that $\zeta_{0,x} = \zeta$.

We now set several constants depending only on $\xi, h, \eps$. 
Let $C$ be the constant given by Proposition~\ref{prop:spherical-multi-opt-ub}. 
By continuity of the Parisi functional $\Par^{\Sp}$ on $\cuK(\xi)$, we may pick $(B^*, \zeta^*) \in \cuK(\xi)$ and a small constant $\Delta \in (0,1)$ such that the following properties hold. 
\begin{enumerate}[label=(\alph*), ref=\alph*]
    \item $\zeta^*$ is positive-valued, right-continuous, and piecewise constant with finitely many jump discontinuities $0 < x_1 < \cdots < x_{r} < 1$. 
    \item For all $\delta \in [0,\Delta]$ and $x\in [0, 1)$, $(B^*, \zeta^*_{\delta, x}) \in \cuK(\xi)$ and
    \begin{equation}
        \label{eq:cp-guarantee}
        \Par^{\Sp}(B^*, \zeta^*_{\delta, x}) \le \ALG + \fr{\eps}{2}.
    \end{equation}
\end{enumerate}

The perturbations $\zeta^*_{\delta, x}$ will be used in the following way.
Given $q_0 \in [0,1]$, we will apply Proposition~\ref{prop:spherical-multi-opt-ub} with $\uzeta + \beta \kappa \zeta = \zeta^*_{(1-q_0)\Delta, q_0}$.
In particular, we will construct $\beta$, $\kappa = \kappa^{\vk,\vp,\vq}$ and $\zeta \in \cM_{\vq}$ such that $\beta \kappa \zeta = \zeta^*_{(1-q_0)\Delta, q_0}$ on $[q_0, 1)$.
Because $\zeta$ is increasing, we must construct a $\kappa$ that decreases rapidly enough to make this equality hold.
In the below proof, the fact that $\zeta^*_{(1-q_0)\Delta, q_0}$ does not have any discontinuities in $[q_0, q_0+(1-q_0)\Delta]$ implies that $q_1 > q_0 + (1-q_0)\Delta$, which implies that $p_1 > \Delta$ for any $\chi$-aligned $\vp, \vq$. 
This allows us to construct a suitable $\kappa$ while keeping $K = \prod_{d=1}^D k_d$ bounded by a constant.

\begin{proof}[Proof of Proposition~\ref{prop:uniform-multi-opt}, spherical case]
    We first set the constants $K_0, \eta_0, N_0$. 
    For $x\in (0,1]$, let $\zeta^*(x^-) = \lim_{y\to x^-}\zeta^*(y)$.
    Let
    \[
        K_0 = 
        \prod_{j=1}^r
        \lt(
            \lt\lfloor 
                \fr{\zeta^*(x_j)}{\Delta \zeta^*(x_j^-)}
            \rt\rfloor 
            + 1 
        \rt).
    \]
    This is well-defined because $\zeta^*$ is positive-valued.
    Let $\eta_0 \in (0,\fr{1}{2})$ satisfy the inequalities
    \begin{align}
        \label{eq:eta-bd}
        CK_0 \lt(
            B^*\eta_0 
            + \eta_0^{1/2} 
            + \eta_0^{1/2} \log \fr{1}{\eta_0}
        \rt)
        &\le \fr{\eps}{4}, \\
        \label{eq:bstbeta-ge-one}
        \eta_0 &\le (B^*)^2, \\
        \label{eq:beta-ge-gamma}
        \eta_0  &< \zeta^*(1^-)^{-2}.
    \end{align}
    Finally, let $N_0$ satisfy $N_0 \ge C \log \max (K_0, 2)$ and 
    \begin{equation}
        \label{eq:n0-bd}
        \fr{CK_0}{\sqrt{N_0}} \le \fr{\eps}{4}.
    \end{equation}
    We emphasize that $K_0, \eta_0, N_0$ depend only on $\xi, h, \eps$.
    
    In the below analysis, we always set $\eta = \eta_0$ (this clearly satisfies $\eta \ge \eta_0$) and $\beta = \eta_0^{-1/2}$.
    
    We are given a correlation function $\chi : [0,1] \to [0,1]$ and a point $\bm \in \bbR^N$ with $\norm{\bm}_N^2 = \chi(0)$.
    We set $q_0 = \chi(0)$; we will set the rest of $\vq$ below.
    We will construct $D, \vk, \vp, \vq, \zeta$ such that on $[q_0, 1)$, 
    \begin{equation}
        \label{eq:increasify}
        \beta \kappa^{\vk,\vp,\vq} \zeta
        =
        \zeta^*_{(1-q_0)\Delta, q_0}.
    \end{equation}
    Let 
    \[
        S = \{x_1,\ldots,x_r\} \cap (q_0 + (1-q_0)\Delta, 1).
    \]
    Set $D-1 = |S|$.
    Set $\vq$ such that $(q_1,\ldots, q_{D-1})$ is the set $S$ in increasing order and $q_D = 1$.
    
    By Proposition~\ref{prop:corelation-fn-properties}(\ref{itm:cor-fn-props-increasing}), $\chi$ is either strictly increasing or constant.
    If $\chi$ is strictly increasing, set $\vp = (p_0, \ldots, p_D)$ by $p_d = \chi^{-1}(q_d)$ for all $q_d \le \chi(1)$ and $p_d=1$ for all $q_d > \chi(1)$.
    If $\chi$ is constant, its unique value is $q_0 = \chi(0)$; set $p_0 = 0$ and $p_d = 1$ for all $1\le d\le D$.
    In either case, $\vp, \vq$ are clearly $\chi$-aligned.
    Moreover, we always have $p_1 > \Delta$: if $\chi$ is increasing, this follows from $q_1 > q_0 + (1-q_0)\Delta$ and  Proposition~\ref{prop:corelation-fn-properties}(\ref{itm:cor-fn-props-sublinear}), while if $\chi$ is constant this is obvious.
    
    Set $k_1 = 1$, and for $1\le d\le D-1$, set
    \[
        k_{d+1} =
        \lt\lfloor 
            \fr{\zeta^*(q_d^-)}{\Delta \zeta^*(q_d)}
        \rt\rfloor + 1.
    \]
    Because $q_1,\ldots,q_{D-1}$ are a subset of $x_1,\ldots,x_r$, we indeed have $K = \prod_{d=1}^D k_d \le K_0$.
    
    This constructs $D,\vk,\vp,\vq,\eta$, which defines $\cH^{\vk,\vp}_N$, $\cQ(\eta) = \cQ^{\Sp}(Q^{\vk,\vq}, \bm, \eta)$, and $\kappa^{\vk,\vp,\vq}$.
    Finally, we construct the sequence $(\zeta_{-1},\zeta_0,\ldots,\zeta_D)$ satisfying
    \begin{equation}
        \label{eq:zeta-incr}
        0 = \zeta_{-1} < \zeta_0 < \cdots < \zeta_D = 1
    \end{equation}
    such that the $\zeta \in \cM_{\vq}$ defined by \eqref{eq:discrete-zeta} satisfies \eqref{eq:increasify} on $[q_0, 1)$.
    In particular, we define $\zeta_d$ for $0\le d\le D-1$ by
    \[
        \zeta_d = \fr{\zeta^*_{(1-q_0)\Delta, q_0}(q_d)}{\beta \kappa^{\vk,\vp,\vq}(q_d)}.
    \]
    For this choice of $\zeta_d$, \eqref{eq:increasify} holds at $q_0,q_1,\ldots,q_{d-1}$ by inspection.
    Because $\zeta$, $\kappa^{\vk,\vp,\vq}$ and $\zeta^*_{(1-q_0)\Delta, q_0}$ are all piecewise constant and right-continuous on $[q_0, 1)$ with jump discontinuities only at $q_1,\ldots,q_{D-1}$, \eqref{eq:increasify} holds on $[q_0, 1)$.
    It remains to verify that this choice of $\zeta_d$ satisfies the increasing condition \eqref{eq:zeta-incr}.
    Because $\zeta^*_{(1-q_0)\Delta, q_0}$ is positive-valued, $\zeta_0 > \zeta_{-1} = 0$.
    At each $1\le d\le D-1$, we have
    \[
        \fr{\zeta_d}{\zeta_{d-1}}
        = 
        \fr{\zeta^*_{(1-q_0)\Delta, q_0}(q_d)}{\zeta^*_{(1-q_0)\Delta, q_0}(q_{d-1})}
        \cdot
        \fr{\kappa^{\vk,\vp,\vq}(q_{d-1})}{\kappa^{\vk,\vp,\vq}(q_d)}
    \]
    By \eqref{eq:kappa-formula},
    \[
        \kappa^{\vk,\vp,\vq}(q_d)
        \le 
        \sum_{j=d+1}^{D-1}\lt[(k_{j+1}-1) \prod_{\ell = j+2}^D k_\ell\rt] + 1 
        = 
        \prod_{\ell = d+2}^D k_\ell,
    \]
    where we upper bounded all the $p_d$ by $1$.
    So, 
    \[
        \fr{\kappa^{\vk,\vp,\vq}(q_{d-1})}{\kappa^{\vk,\vp,\vq}(q_d)}
        = 
        1 + \fr{(k_{d+1}-1) \prod_{\ell=d+2}^D k_\ell}{\kappa^{\vk,\vp,\vq}(q_d)} p_d
        \ge 
        1 + (k_{d+1}-1) p_d 
        \ge 
        k_{d+1}p_d
        \ge 
        k_{d+1} \Delta.
    \]
    Here we used that $p_d \ge p_1 \ge \Delta$.
    Further noting that $\zeta^*_{(1-q_0)\Delta, q_0}(q_{d-1}) = \zeta^*(q_d^-)$, we have
    \[
        \fr{\zeta_d}{\zeta_{d-1}}
        \ge 
        \fr{\Delta \zeta^*_{(1-q_0)\Delta, q_0}(q_d)}{\zeta^*_{(1-q_0)\Delta, q_0}(q_{d-1})}
        \cdot
        k_{d+1} 
        =
        \fr{\Delta \zeta^*(q_d)}{\zeta^*(q_{d}^-)}
        \cdot
        k_{d+1} 
        > 1
    \]
    by definition of $k_{d+1}$.
    Thus $\zeta_d > \zeta_{d-1}$ for $1\le d \le D-1$.
    Finally, because $\kappa^{\vk,\vp,\vq}(q_{D-1}) = 1$, 
    \[
        \zeta_{D-1} = \fr{\zeta^*_{(1-q_0)\Delta, q_0}(q_{D-1})}{\beta} = \eta_0^{1/2} \zeta^*(1^-) < 1 = \zeta_D,
    \]
    using \eqref{eq:beta-ge-gamma}.
    Thus the $\zeta$ we constructed satisfies \eqref{eq:increasify} and \eqref{eq:zeta-incr}.

    Define $\uzeta \in \ucuL$ by $\uzeta = \zeta^*$ on $[0, q_0)$.
    Thus, as elements of $\cuL$, 
    \[
        \uzeta + \beta \kappa^{\vk,\vp,\vq} \zeta = \zeta^*_{(1-q_0)\Delta, q_0}.
    \]
    By construction, $(B^*, \zeta^*_{(1-q_0)\Delta, q_0}) \in \cuK(\xi)$, and \eqref{eq:bstbeta-ge-one} implies $B^* \ge \beta^{-1}$.
    By Proposition~\ref{prop:spherical-multi-opt-ub},
    \[
        \fr{1}{N} \E \max_{\vbsig \in \cQ(\eta)} \cH_N(\vbsig)
        \le 
        K \Par^{\Sp}(B^*, \zeta^*_{(1-q_0)\Delta, q_0})
        + CK^2 \lt(
            B^* \eta + \eta^{1/2} + \eta^{1/2} \log \fr{1}{\eta} + \fr{1}{\sqrt{N}}
        \rt).
    \]
    By \eqref{eq:cp-guarantee},
    \[
        K\Par^{\Sp}(B^*, \zeta^*_{(1-q_0)\delta, q_0})
        \le 
        K \lt(\ALG + \fr{\eps}{2}\rt).
    \]
    By \eqref{eq:eta-bd},
    \[
        CK^2 \lt(
            B^* \eta 
            + \eta^{1/2} 
            + \eta^{1/2} \log \fr{1}{\eta} 
        \rt)
        \le 
        \fr{K\eps}{4}.
    \]
    Finally, by \eqref{eq:n0-bd},
    \[
        \fr{CK^2}{\sqrt{N}} \le \fr{K\eps}{4}.
    \]
    Combining the last four inequalities gives the result.
\end{proof}

\subsection{Deferred Proofs}

Here we give the proofs of Lemmas~\ref{lem:compare-spherical-gaussian}, \ref{lem:chisq-prob-lb}, \ref{lem:gaussian-integral}, and \ref{lem:fn-to-opt}, which are all relatively standard.
We recall the following lemma, due to Talagrand, from which Lemma~\ref{lem:compare-spherical-gaussian} readily follows.
\begin{lemma}[{\cite[Lemma 3.1]{talagrand2006spherical}}]
    \label{lem:compare-spherical-gaussian-aux}
    For all $\by \in \bbR^N$, the following inequality holds.
    \[
        \int_{S_N}
        \exp 
        \la \by, \bsig \ra
        \diff{\mu}(\bsig)
        \le 
        \P\lt(\chi^2(N) \ge BN\rt)^{-1}
        \int 
        \exp 
        \la \by, \brho \ra
        \diff{\nu_B^N(\brho)}.
    \]
\end{lemma}
\begin{proof}[Proof of Lemma~\ref{lem:compare-spherical-gaussian}]
    Using $\cQ(\eta) \subseteq S_N^K$ and Lemma~\ref{lem:compare-spherical-gaussian-aux}, we get
    \begin{align*}
        \exp G_D(\vby)
        &=
        \exp(-\la \vby, \vbm \ra)
        \int_{\cQ(\eta)} 
        \exp 
        \la \vby, \vbsig \ra
        \diff{\mu^K}(\vbsig) \\
        &\le
        \exp(-\la \vby, \vbm \ra)
        \int_{S_N^K} 
        \exp 
        \la \vby, \vbsig \ra
        \diff{\mu^K}(\vbsig) \\
        &=
        \exp(-\la \vby, \vbm \ra)
        \prod_{u\in \bbL}
        \int_{S_N} 
        \exp 
        \la \by(u), \bsig(u) \ra
        \diff{\mu}(\vbsig(u)) \\
        &\le 
        \P\lt(\chi^2(N) \ge BN\rt)^{-K}
        \exp(-\la \vby, \vbm \ra)
        \prod_{u\in \bbL}
        \int
        \exp 
        \la \by(u), \brho(u) \ra
        \diff{\nu_B^N}(\vbrho(u)) \\
        &=
        \P\lt(\chi^2(N) \ge BN\rt)^{-K}
        \exp(-\la \vby, \vbm \ra)
        \int 
        \exp
        \la \vby, \vbrho \ra
        \diff{\nu_B^{N\times K}(\vbrho)}.
    \end{align*}
\end{proof}

\begin{proof}[Proof of Lemma~\ref{lem:chisq-prob-lb}]
    Using the probability density of $\chi^2(N)$, we compute:
    \begin{align*}
        \P(\chi^2(N) \ge BN)
        &=
        \int_{BN}^\infty 
        \fr{x^{N/2-1} e^{-x/2}}{2^{N/2} \Gamma\lt(\fr{N}{2}\rt)}
        \diff{x} \\
        &= 
        \fr{(N/2)^{N/2}}{\Gamma\lt(\fr{N}{2}\rt)}
        \int_B^\infty 
        y^{N/2-1} e^{-Ny/2}
        \diff{y} \\
        &\ge 
        \fr{(N/2)^{N/2}}{\Gamma\lt(\fr{N}{2}\rt)}
        \int_B^\infty 
        e^{-Ny/2}
        \diff{y} \\
        &= 
        \fr{(N/2)^{N/2-1}}{\Gamma\lt(\fr{N}{2}\rt)}
        e^{-BN/2} \\
        &\ge e^{-BN/2},
    \end{align*}
    where the last step uses that $(N/2)^{N/2-1}  \ge  \Gamma\lt(\fr{N}{2}\rt)$ for $N\ge 2$.
\end{proof}

\begin{proof}[Proof of Lemma~\ref{lem:gaussian-integral}]
    By a straightforward computation,
    \begin{align*}
        &\E 
        \exp 
        \fr12 \zeta 
        \lt[
            (\vy + \veta)^\top 
            \Lambda^{-1} 
            (\vy + \veta) 
            - 2\vv^\top (\vy + \veta)
        \rt]\\
        &= 
        |\Sigma|^{-1/2} (2\pi)^{-K/2} 
        \int
        \exp \lt[ -\fr12 \lt(
            \vx^\top \Sigma^{-1} \vx -
            \zeta 
            (\vy + \vx)^\top 
            \Lambda^{-1} 
            (\vy + \vx)
            + 2\zeta \vv^\top (\vy + \vx)
        \rt) \rt]
        \diff{\vx} \\
        &= 
        |\Sigma|^{-1/2} (2\pi)^{-K/2} 
        \int
        \exp \lt[ -\fr12 \lt(
            \vx^\top \lt(\Sigma^{-1} - \zeta \Lambda^{-1}\rt) \vx -
            2 \zeta (\Lambda^{-1} \vy  - \vv)^\top \vx -
            \zeta \vy^\top \Lambda^{-1} \vy
            + 2\zeta \vv^\top \vy
        \rt) \rt]
        \diff{\vx} \\
        &=
        |\Sigma|^{-1/2} |\Sigma^{-1} - \zeta \Lambda^{-1}|^{-1/2}
        \exp \fr12 \lt(
            \zeta^2 
            (\Lambda^{-1} \vy  - \vv)^\top 
            \lt(\Sigma^{-1} - \zeta \Lambda^{-1}\rt)^{-1} 
            (\Lambda^{-1} \vy  - \vv) +
            \zeta \vy^\top \Lambda^{-1} \vy -
            2\zeta \vv^\top \vy
        \rt) \\
        &= 
        \fr{|\Lambda|^{1/2}}{|\Lambda - \zeta \Sigma|^{1/2}}
        \exp \fr{\zeta}{2} \lt(
            \vy^\top (\Lambda - \zeta \Sigma)^{-1} \vy
            - 2 \vv^\top \Lambda (\Lambda - \zeta \Sigma)^{-1} \vy
            + \vv^\top (\zeta \Sigma) (\Lambda - \zeta \Sigma)^{-1} \Lambda \vv
        \rt).
    \end{align*}
    Taking logarithms and dividing by $\zeta$ yields the result.
\end{proof}

\begin{proof}[Proof of Lemma~\ref{lem:fn-to-opt}]
    Define the random variable
    \[
        \vbsig^* 
        = 
        \argmax_{\vbsig \in \cQ(\eta)} 
        \cH_N(\vbsig),
    \]
    where we break ties arbitrarily.
    For $\delta > 0$, define
    \[
        \cB(\vbsig^*, \delta) = 
        \lt\{
            \vbsig \in S_N^K : 
            \norm{\bsig(u) - \bsig^*(u)}_N \le \delta
            \text{~for all~}
            u\in \bbL
        \rt\}.
    \]
    If $\vbsig \in \cB(\vbsig^*, \eta/3)$, then for each $u\in \bbL$ we can write $\bsig(u) = \bsig^*(u) + \delta(u) \brho(u)$, where $\brho(u) \in S_N$ and $0\le \delta(u) \le \eta/3$. 
    Then, for all $u\in \bbL$, 
    \[
        |R(\bsig(u), \bm) - q_0| 
        \le 
        |R(\bsig^*(u), \bm) - q_0| 
        + \delta(u) |R(\brho(u), \bm)| 
        \le 
        \eta + \eta/3 
        \le 
        2\eta,
    \]
    and for all $u,v\in \bbL$,
    \begin{align*}
        &|R(\bsig(u), \bsig(v)) - q_{u\wedge v}| \\
        &\le 
        |R(\bsig^*(u), \bsig^*(v)) - q_{u\wedge v}| + 
        \delta(u) |R(\bsig^*(u), \brho(v))| + 
        \delta(v) |R(\bsig^*(v), \brho(u))| + 
        \delta(u)\delta(v) |R(\brho(u), \brho(u))| \\
        &\le \eta + \eta/3 + \eta/3 + \eta/3 = 2\eta.
    \end{align*}
    So, $\cB(\vbsig^*, \eta/3) \subseteq \cQ(2\eta)$.
    
    Let constants $c, C_1$ be given by Proposition~\ref{prop:gradients-bounded}.
    By this proposition, the event
    \[
        S = \lt\{
            \sup_{u\in \bbL}
            \sup_{\bsig\in S_N} 
            \norm{\nabla \HNp{u}(\bsig)}_N 
            \le C_1
        \rt\}
    \]
    has probability $\P(S) \ge 1 - Ke^{-cN}$.
    Here we use the fact that for $\bv \in \bbR^N$, $\norm{\bv}_N = \norm{\bv}_{\op}$.
    On $S$, 
    \[
        \cH_N(\vbsig) 
        \ge 
        \cH_N(\vbsig^*) - \fr{C_1 NK\eta}{3}
    \]
    for all $\vbsig \in \cB(\vbsig^*, \eta/3)$.
    So,
    \begin{align*}
        F_N^{\Sp}(\beta, \cQ(2\eta))
        &= 
        \fr{1}{N} \log \E 
        \int_{\cQ(2\eta)}
        \exp \beta \cH_N(\vbsig) 
        \diff{\mu^K}(\vbsig) \\
        &\ge 
        \fr{1}{N} \log \E 
        \ind(S)
        \int_{\cB(\vbsig^*, \eta/3)}
        \exp \beta \cH_N(\vbsig) 
        \diff{\mu^K}(\vbsig) \\
        &\ge 
        \fr{1}{N} \log \E 
        \ind(S)
        \int_{\cB(\vbsig^*, \eta/3)}
        \exp \beta \lt(\cH_N(\vbsig^*) - \fr{C_1 NK\eta}{3} \rt)
        \diff{\mu^K}(\vbsig)  \\
        &\ge
        \fr{1}{N} \log \E 
        \ind(S)
        \exp \beta \cH_N(\vbsig^*) 
        -\fr{\beta C_1 K\eta}{3}
        +\fr{1}{N}\log \mu^K(\cB(\vbsig^*, \eta/3))
        \\
        &=
        \beta \GS_N^{\Sp}(\cQ(\eta)) 
        -\fr{\beta C_1 K\eta}{3}
        +\fr{1}{N}\log \mu^K(\cB(\vbsig^*, \eta/3)) \\
        &\qquad 
        +
        \fr{1}{N} \log \E 
        \ind(S)
        \exp \beta \lt(\cH_N(\vbsig^*) - \E \cH_N(\vbsig^*)\rt).
    \end{align*}
    The set $\cB(\vbsig^*, \eta/3)$ is the product of $K$ spherical caps in $S_N$.
    By elementary properties of the spherical measure, there exists a large $C$ such that $\mu^K(\cB(\vbsig^*, \eta/3)) \le \eta^{CNK}$, and so
    \[
        \fr{1}{N}\log \mu^K(\cB(\vbsig^*, \eta/3))
        \ge 
        - CK \log \fr{1}{\eta}.
    \]
    By Proposition~\ref{prop:multi-opt-tail}, 
    \[
        \P\lt(\cH_N(\vbsig^*) - \E \cH_N(\vbsig^*) 
        \le 
        -K\sqrt{4\log 2\cdot \xi(1) N}\rt) 
        \le 
        \fr12.
    \]
    By a union bound, the complement of this event and $S$ simultaneously hold with probability at least $\fr12 - Ke^{-cN}$.
    Thus, 
    \[
        \fr{1}{N} \log \E 
        \ind(S)
        \exp \beta \lt(\cH_N(\vbsig^*) - \E \cH_N(\vbsig^*)\rt)
        \ge 
        -\beta K\sqrt{\fr{4\log 2\cdot \xi(1)}{N}}
        + \fr{1}{N} \log \lt(\fr12 - Ke^{-cN}\rt).
    \]
    Putting this all together, we can choose a large $C$ dependent only on $\xi, h$ such that
    \[
        F_N^{\Sp}(\beta, \cQ(2\eta))
        \ge 
        \beta \GS_N^{\Sp}(\cQ(\eta)) 
        - CK\beta \eta 
        - CK \log \fr{1}{\eta}
        - \fr{CK\beta}{\sqrt{N}}
        - \fr{1}{N} \log \fr{1}{\fr12 - K e^{-cN}}.
    \]
    By choosing $C$ large enough, we can ensure that if $N\ge C\log \max(K, 2)$, then $K e^{-cN} \le \fr14$. 
    Then, we may absorb the last term into the term $CK \log \fr{1}{\eta}$.
    Rearranging yields the result.
\end{proof}

%% file: tex/6-grand-hamiltonian-ub-ising.tex
\section{Overlap-Constrained Upper Bound on the Ising Grand Hamiltonian}
\label{sec:grand-hamiltonian-is}

In this section we upper-bound $\varphi^{\Is}(0)$. We take the reference measure $\mu$ to be the counting measure so that integrals over $\cQ^{\Is}(\eta)$ become sums.

We define $(Z_0,\dots,Z_D)$ similarly to $G_d$ of the previous section, but as a sum over all of $(\bSigma_N)^K$ directly.
As before, define $\vbfeta_0,\dots\vbfeta_D$ to be independent Gaussians as in \eqref{eq:vbfeta0} and \eqref{eq:vbfetad}. 
For $\vby\in(\bbR^K)^N$, define
\begin{align*}
    Z_D(\vby)
    &=\log
    \sum_{\vbsig\in(\bSigma_N)^K}
    \exp
    \sum_{u\in\bbL}\lt\langle \bh+\lambda\bm+\by(u),\pi(\bsig(u))\rt\rangle\\
    &=\log \prod_{i=1}^N\prod_{u\in\bbL}\bigg(2\cosh\left(h+\lambda m_i+\by(u)_i\right)\exp\lt(-m_i (h+\lambda m_i+\by(u)_i)\rt)\bigg)\\
    &=\sum_{i=1}^N\sum_{u\in\bbL}\bigg(\log\left(2\cosh\left(h+\lambda m_i+\by(u)_i\right)\right)-m_i (h+\lambda m_i+\by(u)_i)\bigg).
\end{align*}
Given the sequence $0=\zeta_{-1}<\zeta_0<\zeta_1<\dots<\zeta_L=1$, recursively set 
\[
    Z_{d}(\vby)=\frac{1}{\zeta_{d}}\E \zeta_{d}Z_{d+1}\lt(\vby+\vbfeta_{d+1}(\xi'(q_{d+1}-\xi'(q_d))^{1/2}\rt).
\]
Then $Z_0\equiv Z_0(0)$ is a deterministic function of $\bm$ and $h$.

\begin{proposition}\label{prop:X}
For any $\bm\in[-1,1]^N$,
\[
  \varphi^{\Is}(0)\leq \frac{1}{N}Z_0 +KR(\bh,\bm).
\]

\end{proposition}

\begin{proof}
Recall from \eqref{eq:varphi-0} that
\[
    \varphi^{\Is}(0) 
    = 
    KR(\bh,\bm)+\fr{1}{N} 
    \log 
    \E 
    \sum_{\alpha \in \bbN^D}
    \nu_\alpha
    \sum_{\vbsig\in\cQ^{\Is}(\eta)}
    \exp
    \lt(
    \sum_{u\in\bbL}\langle \bh+\lambda\bm,\pi(\bsig(u))\rangle
    +
    \sum_{u\in \bbL}
    \sum_{i=1}^N 
    g^{(u)}_{\xi',i}(\alpha)
    \pi(\bsig(u))_i
    \rt).
\]
Summing over all of $(\bSigma_N)^K$ gives the upper bound
\[
    \varphi^{\Is}(0) 
    \leq
    K R(\bh,\bm)+
    \fr{1}{N} 
    \log 
    \E 
    \sum_{\alpha \in \bbN^D}
    \nu_\alpha
    \sum_{\vbsig\in(\bSigma_N)^K}
    \exp
    \lt(
    \sum_{u\in\bbL}\langle \bh+\lambda\bm,\pi(\bsig(u))\rangle
    +
    \sum_{u\in \bbL}
    \sum_{i=1}^N 
    g^{(u)}_{\xi',i}(\alpha)
    \pi(\bsig(u))_i
    \rt).
\]
Similarly to previous sections or as in \cite[Theorem 2.9]{panchenko2013sherrington}, properties of Ruelle cascades imply that the right hand side above equals
\[
  KR(\bh,\bm)+\frac{1}{N}Z_0
\]
because the coordinates $i\in [N]$ now decouple. 
\end{proof}

\subsection{Properties of Parisi PDEs}
\label{subsec:1d-parisi-review}

Here we review properties of Parisi PDEs. We begin with the $1$-dimensional case for general $\zeta\in\cuL$ and consider the PDE
\begin{align}
\label{eq:1dParisiPDEdefn}
    \partial_t \Phi_{\zeta}(t,x)+\frac{1}{2}\xi''(t)\left(\partial_{xx}\Phi_{\zeta}(t,x)+\zeta(t)(\partial_x \Phi_{\zeta}(t,x))^2\right)&=0\\
\nonumber\Phi_{\zeta}(1,x)&=f_0(x).
\end{align}
For $\beta>0$ we will consider the initial conditions $f_0(x)=\log(\cosh(\beta x)/\beta)-ax$ for $a=m_i\in [-1,1]$ which leads to solution $\Phi_{a,\zeta}^{\beta}$. 
When not specified, we take $\beta=1$ and $a=0$, so for instance $\Phi_{a,\zeta}=\Phi_{a,\zeta}^1$ and $\Phi_{\zeta}^{\beta}=\Phi_{0,\zeta}^{\beta}$. We also allow the $\beta=\infty$ case $\Phi_{a,\zeta}^{\infty}$ corresponding to $f_0(x)=|x|-ax$. 
Note that \eqref{eq:intro-1dParisiPDEdefn} corresponds to the case $(a,\beta)=(0,\infty)$. 
Regularity properties for solutions to \eqref{eq:1dParisiPDEdefn} were derived in several works such as \cite{jagannath2016dynamic,chen2017variational} for $\zeta\in\cuU$. We draw on the results\footnote{Technically the cited results from \cite{ams20} assume $f_0$ is $1$-Lipschitz and even. The evenness is not used in the proofs of the statement below. In our case $f_0$ is $2$-Lipschitz when $|a|\leq 1$, which is equivalent up to a rescaling as in \eqref{eq:betascale}.}
of \cite{ams20} for $\zeta\in\cuL$.

\begin{proposition}{\cite[Proposition 6.1(b) and Lemma 6.4]{ams20}}
\label{prop:philip}
For $\zeta\in\cuL$ and $(a,\beta)\in [-1,1]\times(0,\infty]$, the function $\Phi_{a,\zeta}^{\beta}$ is continuous on $[0,1]\times\mathbb R$ and $2$-Lipschitz in $x$.
Moreover both
\[
\partial_{xx}\Phi_{a,\zeta}^{\beta}(t,x)\quad\text{and}\quad\partial_t \Phi_{a,\zeta}^{\beta}(t,x)
\]
are uniformly bounded on $(t,x)\in [0,1-\varepsilon]\times \R$ for any $\varepsilon>0$. Finally $\Phi_{a,\zeta}^{\beta}(t,x)$ is convex in $x$.

\end{proposition}

\begin{proposition}{\cite[Lemma 6.5]{ams20}}
For $\zeta\in\cuL$ the SDE
\begin{equation}\label{eq:1dParisiSDE}
\de X_t=\xi''(t)\zeta(t)\partial_x \Phi_{a,\zeta}^{\beta}(t,X_t)\de t+\sqrt{\xi''(t)}\de B_t, \quad X_0=X_0
\end{equation}
has strong and pathwise unique solution.
\end{proposition}

\begin{proposition}\cite[Proposition 6.1(c)]{ams20}\label{prop:6.1}
For $\zeta_1,\zeta_2\in\cuL$, and $\beta\in (0,\infty]$,
\[
|\Phi_{\zeta_1}^{\beta}-\Phi_{\zeta_2}^{\beta}|\leq \int_0^1\xi''(t) |\zeta_1(t)-\zeta_2(t)|dt.
\]
\end{proposition}

\subsubsection{The Multi-Dimensional Parisi PDE}

Here we define the Parisi PDE on $\R^K$. For simplicity we restrict attention to finitely supported $\zeta\in\mathcal M_{\vq}$. We construct $\Phi^{\bbL}$ via the Hopf-Cole transformation and verify that it solves a version of \eqref{eq:1dParisiPDEdefn}. 

Recall the definition of $M^d=M^{\vk,\vp,d}\in \bbR^{K\times K}$ given by
\[
    M^{\vk,\vp,d}_{u^1,u^2} 
    = 
    \ind\{u^1\wedge u^2\ge d\} 
    p_{u^1\wedge u^2}.
\]
As before, $M(t)=M^d$ for $t\in [q_{d-1},q_d)$.

For an atomic measure $\zeta\in\mathcal M_{\vq}$ consider the function $\Phi^{\bbL}_{\zeta}(t,\vx):[0,1]\times\R^{K}\to\R$ defined as as follows. The $t=1$ boundary condition is 
\[
\Phi^{\bbL}_{a,\zeta}(1,\vx)=\sum_{u\in\bbL}\log\left(2\cosh  x(u)\right)-a x(u).
\]
For $t\in [q_0,1)$, $\Phi^{\bbL}_{a,\zeta}$ is defined recursively by
\[
    \Phi^{\bbL}_{a,\zeta}(t,x)=\frac{1}{\zeta(t)}\log\E \exp\left(\zeta(t)\Phi^{\bbL}_{a,\zeta}(q_{d+1},\vx+\veta_{d+1}\cdot (\xi'(q_{d+1})-\xi'(t))^{1/2})\right), \quad t\in [q_{d},q_{d+1})
\]
where $\veta_0\sim \cN(0,M^1)$ and $\veta_d\sim\cN(0,M^d)$ for $1\leq d\leq D$ are independent Gaussian vectors in $\bbR^K$.
For $t \in [0, q_0)$, we extend the definition of $\zeta$ so that $\zeta(t)=0$ and define
\[
    \Phi^{\bbL}_{a,\zeta}(t,x)=\E \Phi^{\bbL}_{a,\zeta}(q_{0},\vx+\veta_{0}\cdot (\xi'(q_{0})-\xi'(t))^{1/2}).
\]

\begin{proposition}\label{prop:XPhi}

For any $\zeta\in\mathcal M_{\vq}$, 
\[
  Z_0=\frac{1}{N}\sum_{i=1}^N\Phi^{\bbL}_{m_i,\zeta}(0,(h+\lambda m_i)\vone).
\]
\end{proposition}

\begin{proof}
This follows from Lemma~\ref{lem:recursiontoPDE} since the recursive definition of $\Phi_{a,\zeta}^{\bbL}(t,x)$ restricted to times $t\in\{q_d\}_{d\in [D]}$ is exactly that of $Z_0$ up to an spatial shift of $(h+\lambda m_i)\vone$.
\end{proof}

We defer the proof of the next lemma, which is a standard computation.

\begin{lemma}\label{lem:recursiontoPDE}

The function $\Phi^{\bbL}_{a,\zeta}$ is smooth on each time interval $[q_d,q_{d+1}]\times \R^{K}$. Moreover it is continuous and solves the $K$-dimensional Parisi PDE
\begin{equation}\label{eq:parisi-pde-Rk}
    \partial_t \Phi^{\bbL}_{a,\zeta}(t,\vx)=-\frac{\xi''(t)}{2}\left(\langle M(t),\nabla^2 \Phi^{\bbL}_{a,\zeta}\rangle + \zeta(t)\langle M(t),(\nabla \Phi^{\bbL}_{a,\zeta})^{\otimes 2}\rangle\right).
\end{equation}
Finally $|\partial_{x(u)} \Phi^{\bbL}_{a,\zeta}(t,\vx)|\leq 1+|a|$ holds for all $(t,\vx,u)\in [0,1]\times \R^{K}\times \bbL$.
\end{lemma}


\subsubsection{Auffinger-Chen Representation}

As shown by \cite{auffinger2015parisi} the Parisi PDE admits a stochastic control formulation. We now recall such representations in the cases of interest starting with the $1$-dimensional case. For $0\leq t_1\leq t_2\leq 1$ let $\mathcal D[t_1,t_2]$ be the space of processes $v\in C([t_1,t_2];\R)$ with $\sup_{t_1\leq r\leq t_2}|v_r|\leq 2$ which are progressively measurable with respect to filtration supporting a standard Brownian motion $B_t$. Define the functional
\[
\mathcal X_{a,\zeta}^{t_1,t_2}(x,v)=\E\left[\mathcal Y_{a,\zeta}^{t_1,t_2}(x,v)-\mathcal Z_{a,\zeta}^{t_1,t_2}(v)\right]
\]
where
\begin{align*}
    \mathcal Y_{a,\zeta}^{t_1,t_2}(x,v)&\equiv \Phi^{\beta}_{a,\zeta}\left(t_2,x+\int_{t_1}^{t_2} \zeta(r)\xi''(r)v_r\de r+\int_{t_1}^{t_2} \sqrt{\xi''(r)}\de B_r\right),\\
    \mathcal Z^{t_1,t_2}_{a,\zeta}(v)&\equiv\frac{1}{2}\int_{t_1}^{t_2} \zeta(r)\xi''(r)v_r^{2} \de r.
\end{align*}
Note that since $|v_r|\leq 2$ is uniformly bounded and $||\xi''\cdot \zeta||_{1}<\infty$ there are no continuity issues near $t=1$. The next proposition, whose standard proof we defer, relates $\Phi^{\beta}_{a,\zeta}$ to stochastic control.

\begin{proposition}\label{prop:1dGTIsing}

For any $\zeta\in\cuL$, $[t_1,t_2]\subseteq [0,1]$, $a\in [-1,1]$ and $\beta\in (0,\infty]$, the function $\Phi_{a,\zeta}^{\beta}$ satisfies
\begin{equation}\label{eq:1dGTvariational}
\Phi^{\beta}_{a,\zeta}(t_1,x)=\sup_{v\in\mathcal D[t_1,t_2]}\mathcal X_{a,\zeta}^{t_1,t_2}(x,v).
\end{equation}
Moreover the maximum in \eqref{eq:1dGTvariational} is achieved by 
\[
v_r=\partial_x\Phi^{\beta}_{a,\zeta}(r,X_r)
\] 
where $X_r$ solves the SDE \eqref{eq:1dParisiSDE} with initial condition $X_{t_1}=x$.

\end{proposition}

The corresponding stochastic control formulation in $\R^{K}$ is as follows. For $0\leq t_1\leq t_2\leq 1$ let $\mathcal D^{\bbL}[t_1,t_2]$ be the space of processes $\vv\in C([t_1,t_2];\R^{K})$ with $\sup_{t_1\leq r\leq t_2}|\vv_r|_{\infty}\leq 2$ which are progressively measurable with respect to a filtration supporting an $\R^K$ valued Brownian motion $\vB_r=(B_r^u)_{u\in\bbL}$. Define the functional
\[
\mathcal X_{a,\zeta}^{\bbL,t_1,t_2}(\vx,\vv)\equiv\E\left[\mathcal Y_{a,\zeta}^{\bbL,t_1,t_2}(\vx,\vv)-\mathcal Z_{a,\zeta}^{\bbL,t_1,t_2}(\vv)\right]
\]
where
\begin{align*}
    \mathcal Y_{a,\zeta}^{\bbL,t_1,t_2}(\vx,\vv)&\equiv \Phi_{a,\zeta}^{\bbL}\left(t_2,\vx+\int_{t_1}^{t_2} \zeta(r)\xi''(r)M(r)\vv_r\de r+\int_{t_1}^{t_2} \sqrt{\xi''(r) M(r)}\de\vB_r\right),\\
    \mathcal Z^{\bbL,t_1,t_2}_{a,\zeta}(\vv)&\equiv\frac{1}{2}\int_{t_1}^{t_2} \zeta(r)\xi''(r)\langle M(r),\vv_r^{\otimes 2}\rangle \de r.
\end{align*}

In the multi-dimensional case we restrict attention to finitely supported $\zeta\in\mathcal M_{\vq}$ to avoid the by-now routine process of extending regularity properties of $\Phi^{\bbL}_{\zeta}$ to general $\zeta$. The proof is again deferred.

\begin{proposition}\label{prop:GTIsing}

For any $\zeta\in\mathcal M_{\vq}$, $[t_1,t_2]\subseteq [0,1]$ and $a\in [-1,1]$, the function $\Phi^{\bbL}_{a,\zeta}$ satisfies
\begin{equation}\label{eq:GTvariational}
\Phi^{\bbL}_{a,\zeta}(t_1,\vx)=\sup_{\vv\in\mathcal D^{\bbL}[t_1,t_2]}\mathcal X_{a,\zeta}^{\bbL,t_1,t_2}(\vx,\vv).
\end{equation}
Moreover \eqref{eq:GTvariational} is maximized by $\vv_s=\nabla\Phi^{\bbL}_{a,\zeta}(s,\vX_s)$ where the $\R^{K}$-valued process $\vX_s$ solves
\[
\vX_s=\vx+\int_{t_1}^s \zeta(r)\xi''(r)M(r)\nabla \Phi^{\bbL}_{a,\zeta}(r,\vX_r)\de r+\int_{t_1}^s \sqrt{\xi''(r)M(r)} \de \vB_r,\qquad s\in [t_1,t_2].
\]

\end{proposition}

\subsection{Relations Among Parisi PDEs}

Following \cite[Section 8]{chen2018generalized} we relate $\Phi_{a,\zeta}$ to $\Phi_{\zeta}$. Note that we always consider times $t\in [0,1]$ with endpoint conditions at $t=1$, while \cite{chen2018generalized} defines the boundary condition for $\Phi_{a,\zeta}$ at time $t=1-q_0$, see e.g. Equation (3.25) therein.

\begin{proposition}\label{prop:Phia-to-Phi}
For any $a\in [-1,1]$ and $\zeta\in\cuL$, with $y=x-a\int_0^1 \xi''(t)\zeta(t)dt$,
\[
	\Phi_{\zeta}(0,y)-ay=\Phi_{a,\zeta}(0,x)+\frac{a^2}{2}\int_0^1 \xi''(t)\zeta(t)dt.
\]
\end{proposition}

\begin{proof}

By setting $y=x-a\int_{0}^1\xi''(s)\zeta(s)\de s$, it suffices to show that for all $t\in[0,1]$,
\[
	\Phi_{a,\zeta}(t,x)=\Phi_{\zeta}\left(t,x-a\int_{t}^1\xi''(s)\zeta(s)\de s\right)-ax+\fr{a^2}{2}\int_{t}^1\xi''(s)\zeta(s)\de s.
\]
(In particular the desired result is obtained by setting $t=0$.) It suffices to show this for $\zeta$ continuous. Set 
\[
	f(t,x)\equiv \Phi_{\zeta}\left(t,x-a\int_{t}^1\xi''(s)\zeta(s)\de s\right)-ax+\fr{a^2}{2}\int_{t}^1\xi''(s)\zeta(s)\de s
\]
and define
\[
	b(t,x)\equiv x-a\int_t^1 \xi''(s)\zeta(s)\de s.
\]
Then we compute
\[
	\partial_{t} f(t, x)=\partial_{t} \Phi_{\zeta}(t, b(t, x))+a \xi''(t) \zeta(t) \partial_{x} \Phi_{\zeta}(t, b(t, x))-\frac{a^{2}}{2} \xi''(t) \zeta(t)
\]
and
\begin{align*}
\partial_{x} f(t, x) &=\partial_{x} \Phi_{\zeta}(t, b(t, x))-a,\\
\partial_{x x} f(t, x) &=\partial_{x x} \Phi_{\zeta}(t, b(t, x)).
\end{align*}
It follows that
\[
	\partial_{t} f(t, x)=-\frac{\xi''(t)}{2}\left(\partial_{x x} f(t, x)+\zeta(t)\left(\partial_{x} f(t, x)\right)^{2}\right).
\]
Note that at time $1$, $f(1,x)=\log(2\cosh(x))-ax=\Phi_{a,\zeta}(1,x)$. Uniqueness of solutions to the Parisi PDE as in \cite[Lemma 13]{jagannath2016dynamic} completes the proof.

\end{proof}

\begin{lemma}\label{lem:zetamonotone}
For any $\zeta,\gamma\in\cuL$ and any $(t,x,\beta)\in [0,1]\times \R\times (0,\infty]$,
\[
\Phi^{\beta}_{a,\zeta}(t,x)\leq \Phi^{\beta}_{a,\zeta+\gamma}(t,x).
\]

\end{lemma}

\begin{proof}
We use the Auffinger-Chen representation \eqref{eq:1dGTvariational} for $\Phi^{\beta}_{a,\zeta}$ and $\Phi^{\beta}_{a,\zeta+\gamma}$. For any control $v$, consider the modified control 
\[
w_t\equiv \frac{\zeta(t)v_t}{\zeta(t)+\gamma(t)}.
\]
It is not difficult to see that
\[
\cY_{a,\zeta}^{t,1}(x,v)=\cY_{a,\zeta+\gamma}^{t,1}(x,w)
\]
since the resulting SDE is the same, while
\[
\cZ_{a,\zeta}^{t,1}(v)\geq\cZ_{a,\zeta+\gamma}^{t,1}(w).
\]
Therefore
\[
\cX_{a,\zeta}^{t,1}(x,v)\leq\cX_{a,\zeta+\gamma}^{t,1}(x,w)
\]
Since $v$ was arbitrary, we are done by Proposition~\ref{prop:1dGTIsing}.
\end{proof}

Define $\ozeta=\zeta|_{[q_0,1]}$ and $\uzeta=\zeta|_{[0,q_0]}$ when $\zeta\in\cuL$ and $q_0\in [0,1]$ are given. The next lemma is analogous to Lemma~\ref{lem:spherical-off-center-worse-than-on-center} and will be used to connect our estimates for $\varphi(0)$ to the Parisi functional uniformly in $\bm$. 


\begin{lemma}\label{lem:Ising-shift-bound}
For $\zeta\in\cuL$, with $\lambda=\int_{0}^{1}\xi''(t)\zeta(t)dt$,
\[
\frac{1}{N}\sum_{i=1}^N \Phi^{\infty}_{m_i,\ozeta}(0,h+\lambda m_i)-\frac{1}{2}\int_{q_0}^1 (t-q_0)\ozeta(t)\xi''(t)dt+R(\bh,\bm)\leq \Par_{\xi,h}^{\Is}(\zeta).
\]
\end{lemma}

\begin{proof}
Define the constants
\begin{align*}
    I&=\int_0^1 t\xi''(t)\zeta(t)dt, \quad J=\lambda=\int_0^1 \xi''(t)\zeta(t)dt,\\
    \oI&=\int_{q_0}^1 t\xi''(t)\ozeta(t)dt,\quad \oJ=\int_{q_0}^1 \xi''(t)\ozeta(t)dt,\\
    \uI&=\int_0^{q_0} t\xi''(t)\uzeta(t)dt,\quad  \uJ=\int_0^{q_0} \xi''(t)\uzeta(t)dt.
\end{align*}
Then $I=\oI+\uI$ and $J=\oJ+\uJ$ and $q_0\uJ\geq \uI$. Recalling that $\Par^{\Is}_{\xi,h}(\zeta)=\Phi^{\infty}_{\zeta}(0,h)-\frac{I}{2}$, we estimate

\[
\begin{WithArrows}
    \hspace{-5cm}\Par^{\Is}_{\xi,h}(\zeta)&=
    \Phi^{\infty}_{\zeta}(0,h)-\frac{I}{2}
    \Arrow{Prop~\ref{prop:philip}}\\
    &= \frac{1}{N}\sum_{i=1}^N\lt(\Phi^{\infty}_{m_i,\zeta}\left(0,h+\lambda m_i 
    \rt)+\frac{m_i^2 J}{2}\right)
    -\frac{I}{2}
    +R(\bh,\bm)
    \Arrow{$\norm{\bm}_N^2=q_0$}\\
    & = \frac{1}{N}\sum_{i=1}^N \Phi^{\infty}_{m_i,\zeta}\left(0,h+\lambda m_i\right) 
    + \frac{q_0 J}{2}-\frac{I}{2} 
    +R(\bh,\bm)
    \Arrow{$\zeta=\ozeta+\uzeta$}\\
    &= \frac{1}{N}\sum_{i=1}^N \Phi^{\infty}_{m_i,\zeta}\left(0,h+\lambda m_i\right) +\frac{q_0\oJ}{2}
    -\frac{\oI}{2} +\frac{q_0\uJ- \uI}{2}
    +R(\bh,\bm)
    \Arrow{$q_0\uJ\geq \uI$}\\
    &\geq \frac{1}{N}\sum_{i=1}^N \Phi^{\infty}_{m_i,\zeta}\left(0,h+\lambda m_i\right) +\frac{q_0\oJ}{2}-\frac{\oI}{2}
    +R(\bh,\bm)
    \Arrow{Lem~\ref{lem:zetamonotone}}\\
    &\geq \frac{1}{N}\sum_{i=1}^N \Phi^{\infty}_{m_i,{\color{red}\ozeta}}\left(0,h+\lambda m_i\right) +\frac{q_0\oJ}{2}-\frac{\oI}{2}
    +R(\bh,\bm).
\end{WithArrows}
\]
This is exactly what we wanted to show.

\end{proof}




The next crucial lemma upper-bounds $\Phi^{\bbL}_{a,\zeta}$ using the $1$-dimensional function $\Phi_{a,\kappa\zeta}$. 
As in the spherical case, multiplying by $\kappa$ will allow us to pass from increasing $\zeta \in \cM_{\vq}$ to arbitrary functions in $\cuL$.

\begin{lemma}\label{lem:ineqising}
For any $\zeta\in\mathcal M_{\vq}$, $\vx\in\R^{K}$, $a\in [-1,1]$ and $t\in[0,1]$,
\begin{equation}\label{ineq:ising}
\Phi^{\bbL}_{a,\zeta}(t,\vx)\leq \sum_{u\in\bbL}\Phi_{a,\kappa\zeta}(t,x(u)).
\end{equation}

\end{lemma}

\begin{proof}

Define
\[
    \widetilde{\mathcal Z}^{\bbL,t_1,t_2}_{a,\zeta}(\vv)\equiv\frac{1}{2}\int_{t_1}^{t_2} \zeta(r)\xi''(r)\kappa(r)^{-1}\langle M(r)^2,\vv_r^{\otimes 2}\rangle \de r.
\]
Since $M(r)\preceq \kappa(r)I_K$, in the Loewner order, it follows that 
\[
    \kappa(r)^{-1}M(r)^2\preceq M(r).
\]
Hence $\widetilde{\mathcal Z}^{\bbL,t_1,t_2}_{a,\zeta}(\vv)\leq \mathcal Z^{\bbL,t_1,t_2}_{a,\zeta}(\vv)$ for any $\vx$ and $\vv$. Setting
\[
\widetilde{\mathcal X}^{\bbL,t_1,t_2}_{a,\zeta}(\vx,\vv)\equiv\mathcal Y^{\bbL,t_1,t_2}_{a,\zeta}(\vx,\vv)-\widetilde{\mathcal Z}^{\bbL,t_1,t_2}_{a,\zeta}(\vv),
\]
it follows that 
\[
    \mathcal X^{t,1}_{a,\zeta}(\vx,\vv)\leq \widetilde{\mathcal X}^{t,1}_{a,\zeta}(\vx,\vv)
\]
always holds. Next for any $\vv\in\mathcal D^{\bbL}[t,1]$ and $r\in [t,1]$, define $\vV_r=\frac{M(r)\vv_r}{\kappa(r)}\in\R^K$. Then 
\[
\langle M(r)^2,\vv_r^{\otimes 2}\rangle=\norm{M(r)\vv_r}_2^2 = \kappa(r)^2\norm{\vV_r}^2
\]
and so (including the relevant Brownian motions as arguments in a slight abuse of notation),
\[
\widetilde{\mathcal Z}^{\bbL,t,1}_{a,\zeta}(\vv,\vB)=\sum_{u\in\bbL}\mathcal Z^{t,1}_{a,\kappa\zeta}(V(u),B(u)).
\]
Moreover since $\kappa(r)\vV_r(u)=M(r)\vv_r(u)$,
\[
\widetilde{\mathcal Y}^{\bbL,t,1}_{a,\zeta}(\vx,\vv,\vB)=\sum_{u\in\bbL}\mathcal Y^{t,1}_{a,\kappa\zeta}(x(u),V(u),B(u)).
\]
Since each coordinate $B_r(u)$ of $\vB_r$ has the marginal law of a $1$-dimensional Brownian motion,
\[
\widetilde{\mathcal X}^{\bbL,t,1}_{a,\zeta}(\vx,\vv)=\sum_{u\in\bbL}\mathcal X^{t,1}_{a,\kappa\zeta}(x(u),V(u)).
\]
Therefore we obtain
\begin{align*}
\mathcal X^{t,1}_{a,\zeta}(\vx,\vv)&\leq \widetilde{\mathcal X}^{t,1}_{a,\zeta}(\vx,\vv)\\
&=\sum_{u\in\bbL}\mathcal X^{t,1}_{a,\kappa\zeta}(x(u),V(u))\\
&\leq  \sum_{u\in\bbL}\Phi_{a,\kappa\zeta}(t,x(u)).
\end{align*}
Since $\vv\in\mathcal D^{\bbL,t,1}$ was arbitrary this concludes the proof.
\end{proof}

\subsection{Zero Temperature Limit}

We now apply the above results with $(\beta^2\xi,\beta\bh,\beta\lambda )$ in place of $(\xi,\bh, \lambda)$, which corresponds to scaling $\cH_N$ to $\beta\cH_N$. We accordingly define $\Phi_{\beta^2\xi,\zeta}$ and $\varphi_{\beta^2\xi}^{\Is}(0)$ by making this substitution in their definitions. It is not hard to derive the scaling relation 
\begin{equation}\label{eq:betascale}
  \Phi_{\beta^2\xi,\zeta}(t,\beta x)=\beta\cdot \Phi^{\beta}_{\beta\zeta}(t,x),\quad (t,x)\in [0,1]\times \bbR
\end{equation}
for any $\beta\in (0,\infty)$ and $\zeta\in\cuL$.

We will also use the following simple estimate to pass to the zero temperature limit.

\begin{proposition}\label{prop:cPbetalarge}
$\sup_{\zeta}\left|\Phi_{\zeta}^{\beta}(t,x)-\Phi_{\zeta}^{\infty}(t,x)\right|\leq \frac{\log 2}{\beta}.$
\end{proposition}

\begin{proof}
Recall that $\Phi_{\zeta}^{\beta}(1,x)$ is convex and $1$-Lipschitz while $\Phi_{\zeta}^{\infty}(1,x)=|x|\leq \Phi_{\zeta}^{\beta}(1,x)$. It follows that 
\begin{align*}
\sup_{x\in\R}\left|\Phi_{\zeta}^{\beta}(1,x)-\Phi_{\zeta}^{\infty}(1,x)\right|&=\left|\Phi_{\zeta}^{\beta}(1,0)-\Phi_{\zeta}^{\infty}(1,0)\right|\\
&=\frac{\log 2}{\beta}.
\end{align*}
Hence
\[
\left|\mathcal X_{\zeta,\beta}^{0,1}(v,x)-\mathcal X_{\zeta,\infty}^{0,1}(v,x)\right|\leq \frac{\log 2}{\beta}
\] 
holds for any control $v$, since the only difference is from the boundary value at time $t=1$ in $\mathcal Y$. Proposition~\ref{prop:1dGTIsing} now implies the desired result.
\end{proof}

Below, recall the definition $\ozeta=\zeta|_{[q_0,1]}$.

\begin{lemma}\label{lem:grand-bound}
Let $(\vp,\vq,\vk)$ be as in Section~\ref{sec:interpolation}, and fix $\beta>0$ and $\zeta\in\cuL$ such that $\ozeta\in\cM_{\vq}$. With 
\[
    \lambda=\int_0^1 \xi''(t)\kappa(t)\zeta(t)\de t
\]
we have
\[
F_N^{\Is}(\beta,\cQ(\eta))\leq \beta K \Par^{\Is}(\beta\kappa\zeta)
+3\beta^2 K^2\xi''(1)\eta+K\beta\lambda\eta+K\log 2 .
\]
\end{lemma}
\begin{proof}

Applying Proposition~\ref{prop:interpolation-common} with $\ozeta$ and $(\beta^2\xi,\beta\bh,\beta\lambda)$ in place of $(\xi,\bh, \lambda)$ in the first line, 
\[
\begin{WithArrows}
  \hspace{-2cm}F_N^{\Is}(\beta,&\cQ(\eta))- 
  3\beta^2 K^2\xi''(1)\eta-K\beta\lambda\eta
  \leq \varphi^{\Is}_{\beta^2\xi}(0)
  - \fr{\beta^2 K}{2}\int_{q_0}^1(q-q_0)\xi''(q)\kappa(q)\zeta(q) \diff{q} 
  \Arrow{Props~\ref{prop:X},\ref{prop:XPhi}}\\
  &\leq {\color{red}\frac{1}{N}\sum_{i=1}^N\Phi^{\bbL}_{\beta^2\xi,m_i,\ozeta}(q_0,\beta h+\beta \lambda m_i)} 
  - \fr{\beta^2 K}{2}\int_{q_0}^1(q-q_0)\xi''(q)\kappa(q)\zeta(q) \diff{q}
  +\beta K R( \bh,\bm)
  \Arrow{Lem~\ref{lem:ineqising}}\\
  &\leq \frac{{\color{red}K}}{N}\sum_{i=1}^N\Phi_{\beta^2\xi,m_i,{\color{red}\kappa}\ozeta}(q_0, \beta h+\beta \lambda m_i)
  - \fr{\beta^2 K}{2}\int_{q_0}^1(q-q_0)\xi''(q)\kappa(q)\zeta(q) \diff{q}
  +\beta K R( \bh,\bm)
  \Arrow{\eqref{eq:betascale}}\\
  &= \frac{{\color{red}\beta }K}{N}\sum_{i=1}^N\Phi{\color{red}^{\beta}_{m_i,\beta\kappa\ozeta}}(q_0, {\color{red}h+\lambda m_i})
  - \fr{\beta^2 K}{2}\int_{q_0}^1(q-q_0)\xi''(q)\kappa(q)\zeta(q) \diff{q}
  +\beta K R( \bh,\bm)
  \Arrow{Prop \ref{prop:cPbetalarge}}\\
  &= \frac{\beta K}{N}\sum_{i=1}^N
  \Phi^{{\color{red}\infty}}_{m_i,\beta\kappa\ozeta}(q_0, h+\lambda m_i)
  - \fr{{\color{red}\beta} K}{2}\int_{q_0}^1(q-q_0)
  \xi''(q){\color{red}\beta}\kappa(q)\zeta(q) \diff{q}\\
 &\qquad +\beta K R( \bh,\bm)+{\color{red} K\log 2}
    \Arrow{Lem~\ref{lem:Ising-shift-bound}}\\
    &\leq {\color{red}\beta K \Par^{\Is}(\beta\kappa\zeta)}+ K\log 2.
\end{WithArrows}
\]
Here terms modified from the previous line are in red text.

\end{proof}

All that remains is to approximate an arbitrary $\zeta_*\in\cuL$ by $\beta\kappa\zeta$ on $[q_0,1)$ for $\zeta\in\cM_{\vq}$ and choose parameters appropriately. We do this now.

\begin{proof}[Proof of Proposition~\ref{prop:uniform-multi-opt}, Ising case]

First choose $\zeta_*=\zeta_*(\xi,h,\varepsilon)\in\cuL$ such that
\begin{equation}\label{eq:zetastar}
	\Par^{\Is}(\zeta_*)\leq \inf_{\zeta\in\cuL}\Par^{\Is}(\zeta)+\frac{\varepsilon}{10}=\ALG^{\Is}+\frac{\varepsilon}{10}.
\end{equation}
Since $\int_0^1 \xi''(t)\zeta_*(t)\de t<\infty$, the monotone convergence theorem guarantees 
\[
	\lim_{\beta\to\infty}\int_0^1 \xi''(t)\cdot \lt|\min(\zeta(t),\beta)-\zeta(t)\rt|\de t=0.
\] 
Define $\zeta_{\beta}(t)=\min(\zeta(t),\beta)$. Therefore there exists 
\begin{equation}\label{eq:beta-defn}
    \beta=\beta(\zeta_*,\xi,\varepsilon)=\beta(\xi,h,\varepsilon)\geq\frac{20\log 2}{\eps}
\end{equation} 
sufficiently large so that (recall Proposition~\ref{prop:6.1})
\begin{equation}\label{eq:zetabeta}
	\Par^{\Is}(\zeta_{\beta})-\Par^{\Is}(\zeta_*) \leq 2\int_0^1 \xi''(t)\cdot \lt|\zeta_{\beta}(t)-\zeta(t)\rt|\de t\leq \fr{\eps}{10}.
\end{equation}
For $\delta>0$, let $q_0^{\delta}=q_0$ and $q^{d+1}_{\delta}=\min(q_d^{\delta}+\delta,1)$. This determines $D$ which satisfies $q_{D-1}<q_D=1$.
Since $\zeta_{\beta}\in\cuL$ is bounded and has bounded variation, there exists $\delta=\delta(\xi,\zeta_{\beta},\eps)=\delta(\xi,\bh,\eps)>0$ such that the function
\[
	\zeta_{\beta,\delta}(t)=\begin{cases}
	\zeta_{\beta}(t),&\quad t\in [0,q_0)\\
	\max\lt(\delta,\zeta_{\beta}(q^{\delta}_j)\rt),&\quad t\in [q^{\delta}_j,q^{\delta}_{j+1}), j\geq 0
	\end{cases}
\]
satisfies
\begin{equation}\label{eq:zetabetadelta}
	\Par^{\Is}(\zeta_{\beta})-\Par^{\Is}(\zeta_{\beta,\delta}) \leq 2\int_0^1 \xi''(t)\lt|\zeta_{\beta}(t)-\zeta_{\beta,\delta}(t)\rt| \de t\leq \fr{\eps}{10}.
\end{equation}
(Note in particular that $\delta$ does not depend on $q_0$.) Observe that $\zeta_{\beta,\delta}(t)\in [\delta, \beta]$ holds for all $t\in [0,1]$. Next define
\[
	k_1=k_2=\dots=k_{D}=k_*\equiv\lt\lceil\fr{\beta}{\delta^2}\rt\rceil.
\]
This leads to $p_d^{\delta}=\chi^{-1}(q_d^{\delta})$ with $\delta\leq p_1^{\delta}\leq p_D^{\delta}= 1$ and hence $\kappa(t)=\kappa_d$ for $t\in [q_d^{\delta},q_{d+1}^{\delta})$, where 
\[
	\delta k_*^{D-d}\le \kappa_d\le k_*^{D-d}.
\]
Next define
\[
	\hzeta_{\beta,\delta}(t)\equiv \frac{\zeta_{\beta,\delta}(t)}{\beta \kappa(t)},\quad t\in[q_0,1]
\]
so that $\beta\kappa \hzeta_{\beta,\delta}=\zeta_{\beta,\delta}.$ Note that
\[
\sup_{t\in [0,1]}\hzeta_{\beta,\delta}(t)\leq \frac{\sup_{s\in[0,1]}\zeta_{\beta,\delta}(s)}{\beta }\leq 1.
\]
Additionally $\hzeta_{\beta,\delta}$ is nondecreasing since if $q_d\leq t_d<q^{\delta}_{d+1}\leq t_{d+1}\leq q^{\delta}_{d+2}$, then
\begin{align*}
	\frac{\hzeta_{\beta,\delta}(t_d)}{\hzeta_{\beta,\delta}(t_{d+1})}&=\frac{\zeta_{\beta,\delta}(t_d)}{\zeta_{\beta,\delta}(t_d)}\cdot \frac{\kappa_{d+1}}{\kappa_d}\\
	&\leq \frac{\beta}{\delta^2 k_*}\leq 1
\end{align*}
by definition of $k_*$. Set
\[
    \lambda=\int_0^1 \xi''(t)\kappa(t)\hzeta_{\beta,\delta}(t)\de t
\]
and
\begin{equation}\label{eq:eta-defn}
    \eta=\frac{\eps}{30\beta K \xi''(1)+10\lambda}.
\end{equation}
We now show that using $\hzeta_{\beta,\delta}$ in the interpolation implies Proposition~\ref{prop:uniform-multi-opt}.
Take $\vp, \vq, \vk,D,\beta,\eta$ as above.
\[
\begin{WithArrows}
    \hspace{-2.5cm}\fr{1}{N} \E \lt[\max_{\vbsig \in \cQ^{\Is}(\eta)}
    \cH_N^{\vk,\vp}(\vbsig) \rt]
	&\leq F^{\Is}_N(\beta,\cQ(\eta))/\beta
	\Arrow{Lem~\ref{lem:grand-bound}}\\
	&\leq K \Par^{\Is}(\beta\kappa \hzeta_{\beta,\delta})
	+3\beta K^2\xi''(1)\eta +K\lambda\eta+\frac{K\log 2}{\beta}
	\Arrow{\eqref{eq:beta-defn}, \eqref{eq:eta-defn}}\\
	&\leq  K\cdot\Par^{\Is}(\zeta_{\beta,\delta})+\frac{2 K \eps}{10}
	\Arrow{\eqref{eq:zetabetadelta}}\\
	&\leq K\cdot\Par^{\Is}(\zeta_{\beta})+\frac{3 K \eps}{10}
	\Arrow{\eqref{eq:zetabeta}}\\
	&\leq  K\cdot\Par^{\Is}(\zeta_*)+\frac{4 K \eps}{10}
	\Arrow{\eqref{eq:zetastar}}\\
	&\leq  K\cdot\ALG^{\Is}+\frac{5 K \eps}{10}.
\end{WithArrows}
\]
Moreover the values $D,\eta$ and $K$ above are bounded depending only on $\xi,h$ and $\eps$. Indeed $D \leq \delta^{-1}+1$, $\eta$ is bounded as in \eqref{eq:eta-defn}, and $K=\prod_{d=1}^D k_i=k_*^D=\lt\lceil\fr{\beta}{\delta^2}\rt\rceil^D$. Meanwhile $\beta$ as defined in \eqref{eq:beta-defn} also depends only on $\xi,h,\eps$. This concludes the proof.

\end{proof}

\subsection{Deferred Proofs}
\label{sec:s6-deferred}

Here we give the missing proofs for this section, which are all relatively standard.

\begin{proof}[Proof of Lemma~\ref{lem:recursiontoPDE}]

We assume $\zeta(t)>0$ as the $\zeta(t)=0$ case is clear. We consider only the case $t\in [q_{D-1},1)$ as the remaining cases are identical by induction. Let $\vy=\vy(t)\in\bbR^K$ be the Gaussian random vector
\[
    \vy=\veta_D(\xi''(1)-\xi''(t))^{1/2}.    
\]
Below $A$ always denotes 
\[
    A(\vx+\vy)=\Phi^{\bbL}_{a,\zeta}(1,\vx+\vy)
\]
and for convenience we set $m=\zeta_D=\zeta(t)$ for $t\in [q_{D-1},1)$. First note that since $|\partial_{x(u)} \Phi^{\bbL}_{a,\zeta}(1,\vx)|\leq 1+|a|$ holds, there are no issues of convergence in any of the expectations even though $\vy$ has unbounded support.

By differentiating in the endpoint value $\vx+\vy$ before taking expectation in $\vy$ it follows that
\[
\nabla \Phi^{\bbL}_{a,\zeta} = \frac{\E[\nabla A e^{mA}]}{\E[e^{mA}]}.
\]
This immediately implies that $|\partial_{x(u)} \Phi^{\bbL}_{a,\zeta}(t,\vx)|\leq 1+|a|$. Similarly one has
\[
\partial_{x_ix_j} \Phi^{\bbL}_{a,\zeta} = \frac{\E\left[\partial_{x_ix_j} A+m(\partial_{x_i} A)(\partial_{x_j} A)e^{mA}\right]}{\E[e^{mA}]}-m\left(\frac{\E[\partial_{x_i}e^{mA}]}{\E[e^{mA}]}\right)\left(\frac{\E[\partial_{x_j}e^{mA}]}{\E[e^{mA}]}\right).
\]
Combining, we compute
\[
\langle T,\nabla^2 \Phi^{\bbL}_{a,\zeta}\rangle + m\langle T,(\nabla L)^{\otimes 2}\rangle=\frac{1}{\E[e^{mA}]}\E[\left(\langle T,\nabla^2 A\rangle+m\langle T,(\nabla A)^{\otimes 2}\rangle\right) e^{mA}] 
\]
Next, note that the time-derivative of the covariance of $\vy(t)$ is $M(t)$. Since $M(t)$ is positive semidefinite we can couple together $(\vy(t))_{t\in [q_{L-1},1]}$ via
\[
\vy(t)=\int_{t}^{1} \sqrt{\xi''(r) M(r)} \de \vB_r
\]
where $\vB_r$ is a standard Brownian motion in $\bbR^K$. Applying Ito's formula backward in time now implies
\[
\frac{\diff{}}{\diff{t}}\E e^{mA(\vx,\vy(t))}= -\frac{1}{2}m\E\left[\left(\langle T,\nabla^2 A\rangle +m\langle M(t), (\nabla A)^{\otimes 2}\rangle\right) e^{mA}\right].
\]
Therefore we conclude
\begin{align*}
\partial_t \Phi^{\bbL}_{a,\zeta} &= -\frac{\frac{\de}{\de t}\E e^{mA(\vx,\vy(t))}}{m\E e^{mA(\vx,\vy(t))}}\\
& = -\frac{1}{2}\langle T,\nabla^2 \Phi^{\bbL}_{a,\zeta}\rangle + m\langle T,(\nabla \Phi^{\bbL}_{a,\zeta})^{\otimes 2}\rangle.
\end{align*}

\end{proof}

\begin{proof}[Proof of Proposition~\ref{prop:1dGTIsing}]

Set
\[
    W_s=x+\int_{t_1}^s \zeta(r)\xi''(r)v_r\de r+\int_{t_1}^s \sqrt{\xi''(r)}\de B_r
\]
and
\[
    V_s\equiv \Phi^{\beta}_{a,\zeta}\left(s, W_s\right)-\frac{1}{2}\int_{t_1}^s \zeta(r)\xi''(r)v_r^2\de r.
\]
Ito's formula gives
\[
\de V_t=\left(\partial_t \Phi^{\beta}_{a,\zeta}(t, W_t)+\zeta(t)\xi''(t)v_t\partial_x \Phi^{\beta}_{a,\zeta}(t, W_t)+\frac{1}{2}\xi''(t)\partial_{xx}\Phi^{\beta}_{a,\zeta}(t, W_t)-\frac{1}{2}\zeta(t)\xi''(t)v_t^2\right)\de t+Y_t\de B_t.
\]
Here $Y_t$ is irrelevant and \eqref{eq:1dParisiPDEdefn} lets us rewrite the finite variation part of $\de V_t$ as 
\begin{align*}
\partial_t \Phi^{\beta}_{a,\zeta}(t, X_t)+\zeta(t)\xi''(t)v_t\partial_x \Phi^{\beta}_{a,\zeta}(t, W_t)+\frac{1}{2}\xi''(t)\partial_{xx}\Phi^{\beta}_{a,\zeta}(t, W_t)&-\frac{1}{2}\zeta(t)\xi''(t)v_t^2\\
&= -\frac{1}{2}\zeta(t)\xi''(t)\left(v_t-\partial_x\Phi^{\beta}_{a,\zeta}(t, W_t)\right)^2\\
&\leq 0.
\end{align*}
We conclude that
\[
\Phi^{\beta}_{\zeta}(t_1,x)\geq \mathcal X_{\zeta}^{t_1,t_2}(x,v)
\]
with equality when $v_r=\partial_x\Phi^{\beta}_{\zeta}(r,W_r)$ holds for all $r\in [t_1,t_2]$. By uniqueness of solutions for SDEs with Lipschitz coefficients, this implies $W_r=X_r$.

\end{proof}

\begin{proof}[Proof of Proposition~\ref{prop:GTIsing}]

The proof is similar to the $1$-dimensional case. First, the SDE defining $\vX_t$ has strong and pathwise unique solutions since $\nabla\Phi^{\bbL}_{a,\zeta}(t, \vx)$ is uniformly bounded and Lipschitz in $\vx$. Set
\[
    \vW_s=\vx+\int_{t_1}^{s} \zeta(r)\xi''(r)M(r)\vv_r\de r+\int_{t_1}^{s} \sqrt{\xi''(r)M(r)}\de \vB(r)
\]
and
\[
V^{\bbL}_s\equiv \Phi^{\bbL}_{a,\zeta}\left(s,\vX_s\right)-\frac{1}{2}\int_{t_1}^s \zeta(r)\xi''(r)\langle M(r),\vv_r^{\otimes 2}\rangle\de r.
\]
By Ito's formula,
\[
\de V^{\bbL}_t=\left(\partial_t \Phi^{\bbL}_{\zeta}(t, \vW_t)+\zeta(t)\xi''(t)\vv_t\partial_x \Phi^{\bbL}_{\zeta}(t, \vW_t)+\frac{1}{2}\xi''(t)\partial_{xx}\Phi^{\bbL}_{\zeta}(t, \vX_t)-\frac{\xi''(t)}{2}\langle M(t),\vv_t^{\otimes 2}\rangle\right)\de t+Y_t^{\bbL}\de B_t.
\]
Here $Y_t^{\bbL}$ is again irrelevant. By \eqref{eq:parisi-pde-Rk} the finite variation part of $\de V^{\bbL}_t$ is 
\begin{align*}
    \partial_t \Phi^{\bbL}_{a,\zeta}(t, \vW_t)
    +\lt\langle M(t),\vv_t\otimes \nabla\Phi^{\bbL}_{a,\zeta}(t, \vW_t)\rt\rangle
    +\frac{1}{2}\xi''(t)\partial_{xx}\Phi^{\bbL}_{a,\zeta}(t, \vW_t)
    &-\frac{\xi''(t)}{2}\langle M(t),\vv_t^{\otimes 2}\rangle\\
    &= -\frac{1}{2}\left\langle M(t),\left(\vv_t-\nabla\Phi^{\bbL}_{a,\zeta}(t, \vW_t)\right)^{\otimes 2}\right\rangle\\
    &\leq 0.
\end{align*}
We conclude that
\[
\Phi^{\bbL}_{a,\zeta}(t_1,\vx)\geq \mathcal X_{a,\zeta}^{\bbL,t_1,t_2}(\vx,\vv)
\]
with equality when 
\[
\vv_r=\nabla\Phi^{\bbL}_{a,\zeta}(r, \vW_r)
\]
holds for all $r\in [t_1,t_2]$. Again, uniqueness of solutions to SDEs with Lipschitz coefficients implies $\vW_r=\vX_r$.

\end{proof}

%% file: tex/7-necessity-of-branching-tree.tex
\section{Necessity of Full Branching Trees}
\label{sec:necessity-fully-branching-trees}

In this section we show, roughly speaking, that it is necessary to use a full branching tree to obtain our results within the overlap gap framework. We restrict for convenience to the setting of spherical models with null external field $h=0$ and set $\ALG^{\Sp}_{\xi} = \ALG^{\Sp}_{\xi,0}=\int_0^1 \xi''(t)^{1/2} \diff{t}$ (recall Proposition~\ref{prop:alg-sp-value}) and $\OPT^{\Sp}_{\xi} = \OPT^{\Sp}_{\xi,0}$. 

A consequence of Theorem~\ref{thm:tree-needed}, proved near the end of this section, can be expressed informally as follows for any $\xi$ with $\ALG^{\Sp}_{\xi}<\OPT^{\Sp}_{\xi}$. Recall the canonical bijection between finite ultrametric spaces and edge-weighted rooted trees (or see Subsection~\ref{subsec:trees-ultrametrics} for a reminder). For all finite ultrametric spaces $X$ of diameter at most $\sqrt{2}$ whose corresponding rooted tree does not contain a subdivision of a full binary subtree of depth $D$, with probability at least $1-e^{-\Omega(N)}$ the following holds. There exists an isometric (up to the scaling factor $\sqrt{N}$) embedding $\iota:X\to S_N$ such that
\[
  H_N(\iota(x))\geq (\ALG^{\Sp}_{\xi}+\eps_{\xi,D})N,\quad \forall x\in X.
\]
Here $\eps_{\xi,D}>0$ is a constant depending only on $\xi$ and $D$, and in particular is independent of the size of the ultrametric $X$. 
In other words, to rule out algorithms achieving better than $\ALG^{\Sp}_{\xi} + \eps$ using forbidden ultrametrics, as $\eps \to 0$ it is necessary to take $D\to\infty$, in effect using the full power of Proposition~\ref{prop:uniform-multi-opt}.

The full statement of Theorem~\ref{thm:tree-needed} shows that in fact a super-constant amount of branching must occur at all ``depths'' in $[0,1]$ where $\xi''(t)^{-1/2}$ is strictly convex. We also show in Theorem~\ref{thm:many-branches-needed} that there exists an embedding $\iota$ as above with large average energy
\[
    \frac{1}{|X|}\sum_{x\in X} H_N(\iota(x))\geq (\ALG^{\Sp}_{\xi}+\eps_{\xi,D})N
\]
unless ``almost all of'' $X$ branches a super-constant amount at ``almost all such depths''. Note that this average energy is what the Guerra-Talagrand interpolation actually allows one to upper bound.
Throughout this section we always consider just a single Hamiltonian $H_N$. This corresponds to the case $\vec p\approx(1,1,\dots,1)$, i.e. a correlation function $\chi(p)$ which sharply increases near $p=1$ such as $\chi(p)=p^{100}$.

Our plan to prove Theorem~\ref{thm:tree-needed} is as follows. If $\ALG_{\xi}^{\Sp}<\OPT_{\xi}^{\Sp}$, there exists an interval $[a,b]\subseteq [0,1]$ on which $(\xi'')^{-1/2}$ is strictly convex. 
Let $\bbT$ be the finite rooted tree with leaf set corresponding to the ultrametric space $X$. 
Let $\eps > 0$ be a small constant depending only on $\xi$ and $D$.
We use the algorithm of \cite{subag2018following} to find embeddings of ancestor points $\iota(x_a)$ for each $x\in X$ of norm $\norm{\iota(x_a)}_2=\sqrt{aN}$ which satisfy
\[
    H_N(\iota(x_a))
    \geq 
    \lt(\int_0^a \xi''(t)^{1/2} \diff{t} 
    -\eps\rt)N.
\]
Next we embed the depth $[a,b]$ parts of $\bbT$ so that the resulting depth $b$ ancestor embeddings $\iota(x_b)$ satisfy
\[
    H_N(\iota(x_b))
    \geq 
    \lt(\int_0^b \xi''(t)^{1/2} \diff{t}
    +2\eps\rt)N.
\]
In other words, from radius $\sqrt{aN}$ to $\sqrt{bN}$, the embedded points' energy grows by $\int_a^b \xi''(t)^{1/2} \diff{t} + 3\eps$, which exceeds the maximum possible growth of an overlap concentrated algorithm by a small constant $3\eps$.
This is the main step of our procedure, and it succeeds whenever the portion of $\bbT$ at depths in $[a,b]$ does not contain a full binary tree of depth $D$. 
The proof uses induction on $D$, and the $D=1$ case is described in Figure~\ref{fig:ladder-pic}. 
We remark that our proof is essentially constructive assuming access to an oracle to find many orthogonal near-maximizers of $H_N$ on arbitrary bands as guaranteed by Lemma~\ref{lem:exact-embedding}.  

Finally we again use the algorithm of \cite{subag2018following} to define embeddings of the leaves $\iota(x)\in S_N$ for $x\in X$ with
\[
    H_N(\iota(x))
    \geq 
    \lt(\int_0^1 \xi''(t)^{1/2} \diff{t}
    +\eps\rt)N.
\]

We remark that in previous multi-OGP arguments, ultrametricity of the forbidden configuration does not explicitly enter. 
However in these arguments, it is always \emph{possible} that the structure of replicas identified is an ultrametric. 
Specifically, in a ``star'' multi-OGP \cite{rahman2017independent, gamarnik2017performance, gamarnik2021partitioning} all the replicas are pairwise equidistant. 
For the ``ladder'' OGP implementations of \cite{wein2020independent, bresler2021ksat}, the forbidden structure is defined by applying some stopping rule to choose a finite number of solutions from a ``stably evolving'' sequence of algorithmic outputs. 
In both settings it is possible that the resulting configuration is a star ultrametric with all pairwise nonzero distances equal.
However, the rooted tree corresponding to such an ultrametric does not contain even a full binary tree of depth $D=2$. Therefore Theorem~\ref{thm:tree-needed} strongly suggests that existing OGP arguments are incapable of ruling out Lipschitz $\cA$ from achieving energies down to the algorithmic threshold $\ALG^{\Sp}_{\xi}$.

\subsection{Preparation}


For given $\xi$ and $t\in [0,1]$, define
\[
  \ALG^{\Sp}_{\xi}(t)=\int_0^t \xi''(s)^{1/2} \de s
\]
so that $\ALG^{\Sp}_{\xi}(1)=\ALG^{\Sp}_{\xi}$. Define also 
\[
  \ALG^{\Sp}_{\xi}([a,b])=\ALG^{\Sp}_{\xi}(b)-\ALG^{\Sp}_{\xi}(a).
\]
Define 
\[
\xi_a(t)=\xi(t)-\xi(a)-(t-a)\xi'(a).
\]
Note that $\xi_a(a)=\xi'_a(a)=0$, and $\xi_a''(t)=\xi''(t)$ for all $t$. Define the rescaled mixture function
\[
  \xi_{[a,b]}(t)=\xi_a\left(a+(b-a)t\right).
\]
We derive 
\[
  \ALG^{\Sp}_{\xi_{[a,b]}}=\int_0^1 \sqrt{\xi_{[a,b]}''(t)}\de t= \int_a^b \sqrt{\xi_a''(s)}\de s = \int_a^b \sqrt{\xi''(s)}\de s= \ALG^{\Sp}_{\xi}([a,b]).
\]\
Correspondingly, define
\[
  \OPT^{\Sp}_{\xi}([a,b])=\OPT^{\Sp}_{\xi_{[a,b]}}.
\]

\begin{proposition}\label{prop:concave-OPT-ALG}
Suppose $\frac{\de^2}{\de t^2}(\xi''(t)^{-1/2})>0$ for $t\in [a,b]\subseteq [0,1]$. Then 
\begin{equation}\label{eq:opt-alg}
  \OPT^{\Sp}_{\xi}([a,b])>\ALG_{\xi}([a,b]).
\end{equation}
\end{proposition}

\begin{proof}
    The result follows from Proposition~\ref{prop:alg-sp-value} applied to $\xi_{[a,b]}$. 
\end{proof}

The next proposition follows from the work \cite{subag2018free} and ensures the existence of many approximately orthogonal replicas which each approximately achieve the ground state energy in spherical spin glasses without external field. In Lemma~\ref{lem:exact-embedding} we make several simple modifications to this result, for instance requiring that the replicas be exactly orthogonal.

\begin{proposition}
\label{prop:many-ortho}
Suppose $\frac{\de^2}{\de t^2}(\xi''(t)^{-1/2})>0$ for $t\in [0,1]$. Then for any $C,\varepsilon>0$ and $k\in\bbN$, for $N\geq N_0=N_0(\xi,C,\eps,k)$, with probability at least $1-e^{-CN}$ either $H_N\notin K_N$ (recall Proposition~\ref{prop:gradients-bounded}) or the following holds. There exist $k$ points $\bsig_1,\dots,\bsig_k\in S_N$ with
\[
  |R(\bsig_i,\bsig_j)|\leq \eps,\quad 1\leq i<j\leq k
\]
and 
\[
  H_N(\bsig_i)\geq N(\OPT^{\Sp}_{\xi}-\eps),\quad i\in [k].
\]
\end{proposition}


\begin{proof}

With the absence of external field, it follows from \cite[Lemma 42]{subag2018free} that $0$ is multi-samplable. Let $\cQ_k(\varepsilon) \subseteq B(\bm,\eps)^k \cap S_N^k$ denote the set of $\vbsig$ with $|R(\bsig_i,\bsig_j)|\leq \eps$ for $i\neq j$.
Let $\mu$ be the uniform measure on $S_N$. 
Define
\[
    F_{N,\beta}=\frac{1}{\beta N}\log \int_{S_N} \exp \beta H_N(\bsig) \diff{\mu(\bsig)}
\]
to be the quenched free energy of $H_N$ on $S_N$ at inverse temperature $\beta$ and 
\[
    \wt{F}_{N,\beta}(\bm)=\wt{F}_{N,\beta}(\bm,k_N,\eps)\equiv \frac{1}{\beta Nk_N} \log \int_{\cQ_{k_N}(\eps)} \exp \beta\sum_{i=1}^{k_N} H_N(\bsig_i) \diff{\mu^k(\vbsig)}.
\]
Here $k_N$ grows to $\infty$ with $N$ at a suitably slow rate. By \cite[Proposition 1 and Theorem 3]{subag2018free}\footnote{In the statement of \cite[Theorem 3]{subag2018free}, there are values $\delta_N,\rho_N$ which also shrink with $N$. We are taking $\eps=\rho_N$ a small constant and ignoring the constraint from $\delta$, so our value of $\wt{F}_{N,\beta}(\bm)$ is larger than that of \cite{subag2018free}. Therefore the lower bound on $\wt{F}_{N,\beta}(\bm)$ we use is somewhat weaker than in the results cited.}  it follows that for $N\geq N_0$ sufficiently large,
\[
  \prob\lt[\E\wt{F}_{N,\beta}(\bm)-\E F_{N,\beta}\ge -\eps\rt]\ge 1-e^{-CN}.
\]
Therefore there exists some $\vbsig\in\cQ_{k_N}(\eps)$ satisfying
\[
  \sum_{i=1}^{k_N} H_N(\bsig_i) \geq Nk_N (\OPT^{\Sp}_{\xi}-\eps-o_{\beta}(1)).
\]
Here $o_{\beta}(1)$ is a value tending to $0$ as $\beta\to\infty$, uniformly in everything else. Assuming $H_N\in K_N$, the values $\frac{1}{N}|H_N(\bsig_i)|$ are uniformly bounded by a constant $C_1$ (because $H_N(0)=0$). It follows by Markov's inequality that at least $k_N\lt(\frac{\eps}{10C_1}-\eps-o_{\beta}(1)\rt)$ of the $\bsig_i$ satisfy $H_N(\bsig_i)\geq N(\OPT^{\Sp}_{\xi}-\fr{\eps}{2}-o_{\beta}(1))$. Since $k_N\to\infty$, eventually
\[
    k\leq\lt\lfloor k_N\lt(\frac{\eps}{10C_1}-\eps-o_{\beta}(1)\rt)\rt\rfloor
\]
for suitably large $\beta$, which completes the proof.
\end{proof}

For fixed $\bm$, define the first-order Taylor expansion 
\[
  \overline{H}^{\bm}_N(\bsig)=H_N(\bm)+\langle \nabla H_N(\bm),\bsig-\bm\rangle.
\]
of $H_N$ and write
\[
    H_N=\overline H^{\bm}_N + \hH^{\bm}_N.
\]
For $0\leq a\leq b\leq 1$ with $\bm\in \sqrt{a}\cdot S_N$, define $B(\bm,0,b)=B(\bm,0)\cap \sqrt{b}\cdot S_N$ and its convex hull $B(\bm,0,[a,b])$.

\begin{lemma}
\label{lem:taylor-HN}
For any fixed $\bm$, the law of $\hH^{\bm}_N$ restricted to $B(\bm,0)$ is a Gaussian process with covariance
\begin{equation}
  \E[\hH^{\bm}_N(\bsig^1)\hH^{\bm}_N(\bsig^2)]=N\xi_a(R(\bsig^1,\bsig^2)).
\end{equation}
Moreover the restrictions of $\hH^{\bm}_N$ and $\overline{H}^{\bm}_N$ to $B(\bm,0)$ are independent.
\end{lemma}

\begin{proof}

Note that for all $\bsig^1,\bsig^2\in B(\bm,0)$,
\[
  R(\bsig^1-\bm,\bsig^2-\bm)=R(\bsig^1,\bsig^2)-a.
\]
Since $\xi_a(t)$ has all derivatives non-negative for $t\geq a$, we may sample a centered Gaussian process $\wt{H}_N$ on $B(\bm,0,[a,1])$ with covariance given by
\begin{align*}
    \E[\wt{H}_N(\bsig^1)\wt{H}_N(\bsig^2)]
    &= N\xi_a(R(\bsig^1-\bm,\bsig^2-\bm)+a) \\
    &=N\xi_a(R(\bsig^1,\bsig^2)).
\end{align*}
Next, generate the independent centered Gaussian process $\underline{H}_N$ by
\[
  \E[\underline{H}_N(\bsig^1)\underline{H}_N(\bsig^2)]=N\lt(\xi(a)+\xi'(a)\lt(R(\bsig^1,\bsig^2)-a\rt)\rt).
\]
It follows by adding covariances (with $x=R(\bsig^1,\bsig^2)$ in the definition of $\xi_a$) that 
\[
  \wt{H}_N+\underline{H}_N \stackrel{d}{=}H_N
\]
when restricted to $B(\bm,0)$.
Since $\xi_a(a)=\xi'_a(a)=0$, it follows that $\wt{H}_N(\bm)=0$ and $\nabla \wt{H}_N(\bm)=0$ hold almost surely. 
Therefore $\underline{H}_N=\overline{H}^{\bm}_N$ is the first-order Taylor expansion of $H_N$ around $\bm$, and then also $\wt{H}_N=\hH^{\bm}_N$. Moreover $\wt{H}_N$ and $\underline{H}_N$ are independent by construction. This concludes the proof.
\end{proof}

In the following Lemma~\ref{lem:exact-embedding}, we refine Proposition~\ref{prop:many-ortho} in several simple but convenient ways. In particular, Lemma~\ref{lem:taylor-HN} implies the same result uniformly over all bands $B(\bm,0,b)$; it also guarantees exact orthogonality. Lemma~\ref{lem:exact-embedding} will serve as a useful tool for embedding more complicated ultrametric trees. Roughly speaking, it gives a way to gain on the embedding algorithm of \cite{subag2018following} (stated later as Proposition~\ref{prop:tree-embedding-alg}).

\begin{lemma}\label{lem:exact-embedding}
Suppose $\frac{\de^2}{\de t^2}(\xi''(t)^{-1/2})>0$ for $t\in [a,b]\subseteq [0,1]$. 
Then there exists $\varepsilon>0$ depending only on $\xi,a,b$ such that for any $k$, for $N\geq N_0(\xi,a,b,k)$ sufficiently large and some $c=c(\xi,a,b,k)$, with probability $1-e^{-cN}$ the following holds. 

For any $\bm$ with $||\bm||_N^2=a\leq 1$ and any linear subspace $W\subseteq \R^N$ with $\dim(W)\geq N-k$, there exist $k$ points $\bsig_1,\dots,\bsig_k\in W+\bm$ such that
\begin{equation}\label{eq:ortho-opts}
  R(\bsig_i-\bm,\bsig_j-\bm)=(b-a)\cdot\ind(i=j)\quad\forall i,j\in [k]
\end{equation}
and
\begin{equation}\label{eq:many-opts}
 H_N(\bsig_i)\geq H_N(\bm)+N(\ALG^{\Sp}_{\xi}([a,b])+\varepsilon)\quad \forall i\in [k].
\end{equation}

\end{lemma}

\begin{proof}

Consider a (non-random) $\eta\sqrt{N}$-net $\mathcal N_{\eta}$ on $\sqrt{a}\cdot S_N$ of size at most $(10/\eta)^N$. For any $\bm\in \sqrt{a}\cdot S_N$, the Hamiltonian $\hH_N^{\bm}(\bsig)$ restricted to $B(\bm,0,b)$ has covariance
\begin{align*}
    \E \hH_N^{\bm}(\bsig_1)\hH_N^{\bm}(\bsig_2)
    &= N\xi_a(R(\bsig_1,\bsig_2))\\
    &= N\xi_{[a,b]}\lt(R\lt(\fr{\bsig_1-\bm}{\sqrt{b-a}},
    \fr{\bsig_2-\bm}{\sqrt{b-a}}\rt)\rt).
\end{align*}
Since 
\[
  ||\bsig-\bm||_2= \sqrt{N(b-a)}
\]
for $\bsig\in B(\bm,0,b)$, we conclude that $\hH_N^{\bm}$ is exactly an $N-1$ dimensional spin glass with mixture $\xi_{[a,b]}$ on $B(\bm,0,b)$ up to rescaling the input.

Fix a large constant $M$, and choose $\varepsilon$ sufficiently small depending on $M$. We apply Proposition~\ref{prop:many-ortho} to $\hH_N^{\bm}$ with mixture $\xi_{[a,b]}(t)$ based on the observation just above.
Recall that the constant $C$ in Proposition~\ref{prop:many-ortho} can be arbitrarily large.
It follows by a union bound that with probability $1-e^{-C_1 N}$, for all $\bn\in \cN_{\eta}$ there exist $\wt\bsig_1(\bn),\dots,\wt\bsig_M(\bn)$ satisfying
\[
  |R(\wt\bsig_i(\bn)-\bn,\wt\bsig_j(\bn)-\bn)-(b-a)\cdot \ind(i=j)|\leq \eps\quad\forall 1\leq i<j\leq M
\]
and 
\begin{equation}\label{eq:all-replicas-good}
  \hH_N(\wt\bsig_i(\bn))\geq N(\OPT^{\Sp}_{\xi}([a,b])-\eps)\quad\forall i\in [M].
\end{equation}
For any $\bm\in \sqrt{a}\cdot S_N$, there exists by definition $\bn\in \cN_{\eta}$ with $||\bm-\bn||\leq \eta\sqrt{N}$. Then with $\wt\bsig_i=\wt\bsig_i(\bn)$ as above,
\[
  |R(\wt\bsig_i-\bm,\wt\bsig_j-\bm)-(b-a)\cdot \ind(i=j)|\leq \eps_1\quad\forall 1\leq i<j\leq M
\]
for some $\eps_1=o_{\eps,\eta}(1)$ tending to $0$ as $\eps,\eta\to 0$. Define the linear subspace $\wt W\subseteq W$ by
\[
  \wt W= W\cap \bm^{\perp} \cap (\nabla H_N)^{\perp}
\]
where $(\cdot)^{\perp}$ denotes orthogonal complement.
Let $P_{{\wt W}^{\perp}}$ be the orthogonal projection matrix onto ${\wt W}^{\perp}$. It is easy to see that
\[
  \lt|\lt|\sum_{i=1}^M (\wt\bsig_i-\bm)^{\otimes 2}\rt|\rt|_2^2\leq (1+M\eps)N\leq 2N
\]
for $\eps$ sufficiently small. Then
\begin{align*}
  \sum_{i=1}^M ||P_{{\wt W}^{\perp}}(\wt\bsig_i)||^2_2 &= \lt\langle P_{{\wt W}^{\perp}},\sum_{i=1}^M (\wt\bsig_i-\bm)^{\otimes 2}\rt\rangle\\
  &\leq || P_{{\wt W}^{\perp}}||_2^2\cdot\lt|\lt|\sum_{i=1}^M (\bsig_i-\bm)^{\otimes 2}\rt|\rt|_2^2\\
  &\leq 2(k+2)N.
\end{align*}
By the pigeonhole principle, at most $M-k$ values $i\in [M]$ can satisfy
\[
||P_{{\wt W}^{\perp}}(\bsig_i-\bm)||_2^2 \geq \frac{2(k+2)N}{M-k}.
\]
It follows that there exist a subset $\wt\bsig_{i_1},\dots,\wt\bsig_{i_k}$ with
\[
 ||P_{{\wt W}^{\perp}}(\wt\bsig_{i_j}-\bm)||_2^2 \leq \eta N,\quad j\in [k]
\]
where $\eta\leq \frac{2(k+2)}{M-k}$ is arbitrarily small (by choosing $M$ large depending on $k$). Defining $\bsig_{i_1}',\dots, \bsig_{i_k}'$ by
\[
  \bsig_{i_j}'-\bm=P_{{\wt W}^{\perp}}(\wt\bsig_{i_j}-\bm),
\]
we have
\[
  \bsig_{i_1}',\dots,\bsig_{i_k}'\in \bm+\wt W
\]
satisfying 
\[
  |R(\bsig_{i_j}'-\bm,\bsig_{i_{\ell}}-\bm)-(b-a)\cdot \ind(j=\ell)|\leq \eps_2,\quad j,\ell\in [k]
\]
and 
\[
  ||\bsig_{i_j}'-\wt\bsig_{i_j}||_2^2\leq \eta N,\quad j\in [k].
\]
Here $\eps_2=o_{\eps_1,\eta}(1)$ tends to $0$ as $\eps_1$ and $\eta$ tend to $0$. Using Gram-Schmidt orthonormalization inside the affine subspace $B(\bm,0)$, for $\eps_3=o_{\eps_2}(1)$ we may find $\hbsig_1,\dots,\hbsig_k\in B(\bm,[a,b])\cap W$ satisfying
\[
  R(\hbsig_i-\bm,\hbsig_j-\bm)=(b-a)\cdot \ind(i=j)\quad\forall 1\leq i<j\leq k
\]
and 
\[
  ||\hbsig_j-\bsig_{i_j}'||_2^2 \leq \eps_3 N\quad\forall j\in [k].
\]
Assuming $H_N$ is $C_1\sqrt{N}$-Lipschitz with respect to the $\norm{\cdot}_2$ norm (recall Proposition~\ref{prop:gradients-bounded}), this implies based on \eqref{eq:all-replicas-good} that for some $\eps_4=o_{\eps_3}(C_1+1)$,
\[
  \hH_N(\hbsig_j)\geq N(\OPT^{\Sp}_{\xi}-\eps_4)\quad\forall i\in [k]
\] 
and 
\[
  ||\hbsig_j-\wt\bsig_{i_j}||_2^2 \leq 2(\eps_3+\eta) N\quad\forall j\in [k].
\]
Recalling Proposition~\ref{prop:concave-OPT-ALG}, this completes the proof.
\end{proof}

\subsection{Trees and Ultrametrics}
\label{subsec:trees-ultrametrics}

We recall the well known connection between trees and ultrametric spaces. Here and throughout given a rooted tree $\bbT$ with root $r(\bbT)$ we denote by $\PA(v)$ the parent of $v\in V(\bbT)\backslash\{r(\bbT)\}$

\begin{definition}{\cite{bocker1998recovering}} A dated, rooted tree $\bbT$ with range $[a,b]\subseteq [0,1]$ is a finite tree rooted at $r(\bbT)\in V(\bbT)$ together with a height function
\[
  |\cdot|:V(\bbT)\to [a,b]
\] 
satisfying the following properties.
\begin{itemize}
  \item  $|r(\bbT)|=a$.
  \item $|v|=b$ for all leaves $v\in L(\bbT)$.
  \item $|\PA(v)|< |v|$ for all $v\in V(\bbT)\backslash\{r(\bbT)\}$. 
\end{itemize}
We say $\bbT$ is reduced if no $v\in V(\bbT)$ except possibly $r(\bbT)$ has exactly $1$ child.

\end{definition}

In a rooted tree, let $u\wedge v\in V(\bbT)$ denote the least common ancestor of vertices $u$ and $v$. To any dated rooted tree $\bbT$, we associate a metric $d_T:V(\bbT)\times V(\bbT)\to [0,\sqrt{2}]$ characterized by
\begin{equation}\label{eq:depth-to-d}
  |u\wedge v|=\frac{|u|-d_T(u,v)^2+|v|}{2},\qquad u,v\in V(\bbT).
\end{equation}
When $u,v\in L(\bbT)$ are leaves and $\bbT$ has range $[a,b]$, this becomes
\begin{equation}
  |u\wedge v|=b-\frac{d_T(u,v)^2}{2},\qquad u,v\in L(\bbT).
\end{equation}
Crucially, observe that for $u,v\in L(\bbT)$, the value $d_T(u,v)$ is a strictly decreasing function of $|u\wedge v|$. Therefore $d_T$ defines an ultrametric on $L(\bbT)$, or in fact the set of vertices at any fixed height. The specific decreasing bijection between $|u\wedge v|\in [0,1]$ and $d_T(u,v)\in [0,\sqrt{2}]$ for $u,v\in L(\bbT)$ can in general be arbitrary; the one above is suited for embeddings into Euclidean space since
\begin{equation}\label{eq:overlap-distance}
  R(\bsig^1,\bsig^2)=\frac{R(\bsig^1,\bsig^1)-||\bsig^1-\bsig^2||_N^2+R(\bsig^2,\bsig^2)}{2},\qquad \bsig^1,\bsig^2\in \bbR^N.
\end{equation}
The following type of result is well known and seems to be folklore. 

\begin{proposition}{\cite[Section 6]{rammal1986ultrametricity},\cite{bocker1998recovering}}
For any finite set $X$, \eqref{eq:depth-to-d} defines a bijection between the following two isomorphism classes.
\begin{enumerate}
  \item Dated, rooted reduced trees with range $[0,1]$ and leaf set $X$.
  \item Ultrametric structures on $X$ with diameter at most $\sqrt{2}$. 
\end{enumerate}
\end{proposition}

Any dated, rooted tree can be naturally reduced by removing vertices with a single child and connecting their parent and child. Hence we will consider general dated, rooted trees to give ourselves more flexibility. We are interested in embeddings of the leaves $L(\bbT)$ into level sets $\{\bsig\in \bbR^N:H_N(\bsig)\geq (\ALG+\eps)N\}$ which are isometries up to the scaling factor $\sqrt{N}$. It will be convenient to embed the entire vertex set $V(\bbT)$.

\begin{definition}
A \emph{Euclidean embedding} of a dated, rooted tree $\bbT$ to is a function $\iota:V(\bbT)\to\bbR^N$
satisfying
\[
  R(\iota(u),\iota(v))=|u\wedge v|\qquad \forall u,v\in V(\bbT).
\]
or equivalently (by \eqref{eq:overlap-distance}), 
\[
  ||\iota(u)-\iota(v)||_N=d_T(u,v)\qquad \forall u,v\in V(\bbT).
\]
\end{definition}

The next lemma gives an alternate characterization of Euclidean embeddings. .

\begin{lemma}
\label{lem:isometric-criterion}
$\iota:V(\bbT)\to\bbR^N$ is a Euclidean embedding if and only if the following properties hold. Below we implicitly define $|\PA(r(\bbT))|=0$ and $\iota(\PA(r(\bbT))=0$.
\begin{enumerate}
  \item $\iota(r(\bbT))=a$.
  \item For all $v\in V(\bbT)$,
  \[
    ||\iota(v)-\iota(\PA(v))||_N = |v|-|\PA(v)|.
  \]
  \item For all distinct $u,v\in V(\bbT)$,
  \[
    R(\iota(u)-\iota(\PA(u)),\iota(v)-\iota(\PA(v)))=0.
  \]
\end{enumerate}
\end{lemma}

\begin{proof}

First if $\iota$ is a Euclidean embedding, then clearly the first property holds. The second holds because
\begin{align*}
  ||\iota(v)-\iota(\PA(v))||_N^2
  &=
  R(\iota(v)-\iota(\PA(v)),\iota(v)-\iota(\PA(v)))\\
  & = |v\wedge v|-|v\wedge \PA(v)|  -|\PA(v)\wedge v|+|\PA(v)\wedge \PA(v)| \\
  &=|v|-|\PA(v)|.
\end{align*}
For the third property, we compute
\[
  R(\iota(u)-\iota(\PA(u)),\iota(v)-\iota(\PA(v)))
  = |u\wedge v| - |v\wedge \PA(v)| - |\PA(u)\wedge v| + |\PA(u)\wedge \PA(v)|.
\]
Since $u\neq v$, without loss of generality suppose $v\neq (u\wedge v)$. Then $v$ is an ancestor of neither $u$ nor $\PA(u)$. The third property then follows because 
\[
  u\wedge v = u\wedge \PA(v)
  \qquad
  \text{and}
  \qquad 
  \PA(u)\wedge v = \PA(u)\wedge \PA(v).
\]

In the other direction, let us assume the three properties hold and show $\iota$ is a Euclidean embedding. Consider vertices $u=u_n$ and $v=v_m$ with ancestor paths \[
    (r(\bbT)=u_0,u_1,\dots,u_{n-1}),\qquad (r(\bbT)=v_0,v_1,\dots,v_{m-1}).
\]
Suppose that $u\wedge v=u_d=v_d$, so that $u_j=v_j$ if and only if $j\leq d$. Then we expand
\begin{align*}
R(\iota(u),\iota(v)) &= a+\sum_{\substack{1\leq i \leq n\\1\leq j\leq m}} R\big(\iota(u_{i})-\iota(u_{i-1}),\iota(v_{j})-\iota(v_{j-1})\big)\\
& = a+\sum_{1\leq i\leq d} R\big(\iota(u_{i})-\iota(u_{i-1}),\iota(u_{i})-\iota(u_{i-1})\big)\\
& = a+\sum_{1\leq i\leq d} |u_i|-|u_{i-1}|\\
&=a+|u_{d}|-|r(\bbT)|\\
&=|u\wedge v|.
\end{align*}
\end{proof}

Next, define the depth $D$ rooted binary tree $\bbT^2_D$ on vertex set
\[
  V(\bbT^2_D)=\emptyset \cup [2]\cup [2]^2\cup\dots\cup [2]^D.
\]
As usual, a vertex $v\in [2]^j$ is the parent of $u\in [2]^{j+1}$ if and only if $v$ is an initial substring of $u$. We say the rooted tree $\bbT$ \emph{contains} $\bbT^2_D$ if there exists an ancestry-preserving injection
\[
  \phi:V(\bbT^2_D)\to V(\bbT)
\]
(which need not preserve the root). Define the branching depth $D(\bbT)$ to be the largest $D$ such that $\bbT$ contains $\bbT^2_D$. For $v\in V(\bbT)$, define $D(v)=D(\bbT_v)$ where $\bbT_v\subseteq \bbT $ is the subtree rooted at $v$.

\begin{proposition}

For any rooted tree $\bbT$, the set 
\begin{equation}
    \label{eq:def-VD}
    V_D=\{v\in V(\bbT):D(v)=D(\bbT)\}
\end{equation}
is a simple path graph with one endpoint at $r(\bbT)$.

\end{proposition}

\begin{proof}

Let $D=D(\bbT)$. Clearly $V_D$ is closed under ancestry and contains $r(\bbT)$. Suppose for sake of contradiction that $V_D$ is not a path with $r(\bbT)$ as an endpoint. Then $V_D$ contains vertices $v$ and $w$ neither of which is an ancestor of the other. But if the disjoint subtrees rooted at $v$ and $w$ each contain $\bbT^2_D$, then $\bbT$ contains $\bbT^2_{D+1}$, a contradiction. 
\end{proof}

We will use the following slight generalization of the main result of \cite{subag2018following}. It can be seen as the ``default'' embedding procedure which ensures energy $\ALG^{\Sp}_{\xi}$ at the leaves, while Lemma~\ref{lem:exact-embedding} gives a tool to improve over this embedding on intervals $[a,b]$ where $\xi''(t)^{-1/2}$ is convex.

\begin{proposition}\label{prop:tree-embedding-alg}
For any $\varepsilon>0$ and $k\in\Z^+$, there exist $c$ and $N_0$ depending on $\xi,\varepsilon,k$ such that with probability $1-e^{-cN}$ the following holds for all $N\geq N_0$. 

For any $\bm$ with $||\bm||_N^2=q\leq 1$, any dated, rooted tree $\bbT$ of order $|V(\bbT)|\leq k$ with range $[q,1]$, and any linear subspace $W\subseteq \R^N$ with $\dim(W)\geq N-k$, there is an embedding $\iota$ of $X$ into $W+\bm$ such that 
\begin{equation} \label{eq:alg-distances}
||\iota(u)-\iota(v)||_N = d(u,v)\quad \forall u,v\in V(\bbT)
\end{equation}
and
\begin{equation} \label{eq:alg-energies}
H_N(\iota(x))\geq H_N(\bm)+N\cdot (\ALG^{\Sp}_{\xi}(\norm{\iota(u)}_N^2-\ALG^{\Sp}_{\xi}(\norm{\bm}_N^2)-\varepsilon)\quad \forall v\in V(\bbT).
\end{equation}

\end{proposition}

\begin{proof}

The proof is essentially contained in \cite[Theorem 4 and Remark 6]{subag2018following}. The restriction to $W+\bm$ has no effect on the proof whenever $k\leq o(N)$. Indeed, a GOE matrix has $\Omega_{\varepsilon}(N)$ eigenvalues at least $2-\eps$ with probability $1-e^{-\Omega_{\eps}(N^2)}$. This property implies existence of an eigenvalue at least $2-\eps$ when a GOE matrix is restricted to any subspace of dimension at least $N-\Omega_{\varepsilon}(N)$ by the Courant-Fisher theorem. Repeating the proof of \cite[Theorem 4]{subag2018following} with this minor modification establishes the result.
\end{proof}

The following simple lemma is a slightly more general statement of Proposition~\ref{prop:tree-embedding-alg}. It will be used to extend partial embeddings of ancestor-closed subsets of $V(\bbT)$ to all of $V(\bbT)$.

\begin{lemma}
\label{lem:euclidean-tree-embedding-alg}
For any $\eps>0$ and finite dated rooted tree $\bbT$, there exist $c$ and $N_0$ depending on $\xi,\eps,T$ such that with probability $1-e^{-cN}$ the following holds for all $N\geq N_0$.

For any ancestor-closed subset $U\subseteq V(\bbT)$, let $\iota_U:U\to \bbR^N$ be a Euclidean embedding. Then there is a Euclidean embedding $\iota:V(\bbT)\to\bbR^N$ extending $\iota_U$ such that for any $v\in V(\bbT)$ with lowest $U$-ancestor $u\in U$,
\begin{equation}
\label{eq:tree-embedding-again}
  H_N(\iota(v))\geq H_N(\iota(u))+N\cdot (\ALG^{\Sp}_{\xi}(|v|)-\ALG^{\Sp}_{\xi}(|u|)-\eps).
\end{equation}
\end{lemma}

\begin{proof}

The result follows by repeated application of Proposition~\ref{prop:tree-embedding-alg}. 
Indeed, $V(\bbT)\backslash U$ consists of a disjoint union of subtrees $\bbT_i$ rooted at vertices $u_1,\dots,u_k$ in $U$. 
For each $j\in [k]$, given a Euclidean embedding $\iota_U^{j-1}$ of 
\[
  U_{j-1}=U\cup\lt(\bigcup_{1\leq i\leq j-1} \bbT_i\rt),
\] 
we extend it to $\bbT_j$ as follows. 
Let 
\[
  W_j=\Span(\iota(u)_{u\in U_{j-1}})^{\perp}
\]
be the orthogonal complement of the span of the already-embedded vertices.
Then applying Proposition~\ref{prop:tree-embedding-alg} with $W=W_j$ and $\bm=\iota(u_j)$, we obtain a Euclidean embedding of $\bbT_j$ into $W_j+\iota(u_j)$, which joins with $\iota_U^{j-1}$ to form an embedding $\iota_U^j$ on $U_j$. 
It follows from Lemma~\ref{lem:isometric-criterion} that $\iota_U^{j}$ is again a Euclidean embedding of $U_j$. 
Moreover \eqref{eq:alg-energies} ensures that \eqref{eq:tree-embedding-again} is satisfied for each $v\in \bbT_j$. 
Repeating this inductively for each $j\in [k]$ completes the proof.
\end{proof}






\subsection{Proof of Theorems~\ref{thm:tree-needed} and \ref{thm:many-branches-needed}}

We now show that full binary trees are necessary for our results, in the sense that trees $\bbT$ not containing $\bbT_D^2$ fail as obstructions to energy $(\ALG^{\Sp}_{\xi}+\eps_{\xi,D})N$ for some $\eps_{\xi,D}>0$ independent of $|V(\bbT)|$. The main arguments are devoted to proving Lemma~\ref{lem:tree-needed}, which implies the two main theorems. Lemma~\ref{lem:tree-needed} is proved by induction on $D$, and a representative case for $D=1$ is depicted in Figure~\ref{fig:ladder-pic}.

\begin{figure}[H]
\centering
\begin{framed}
\includegraphics[width=0.45\linewidth]{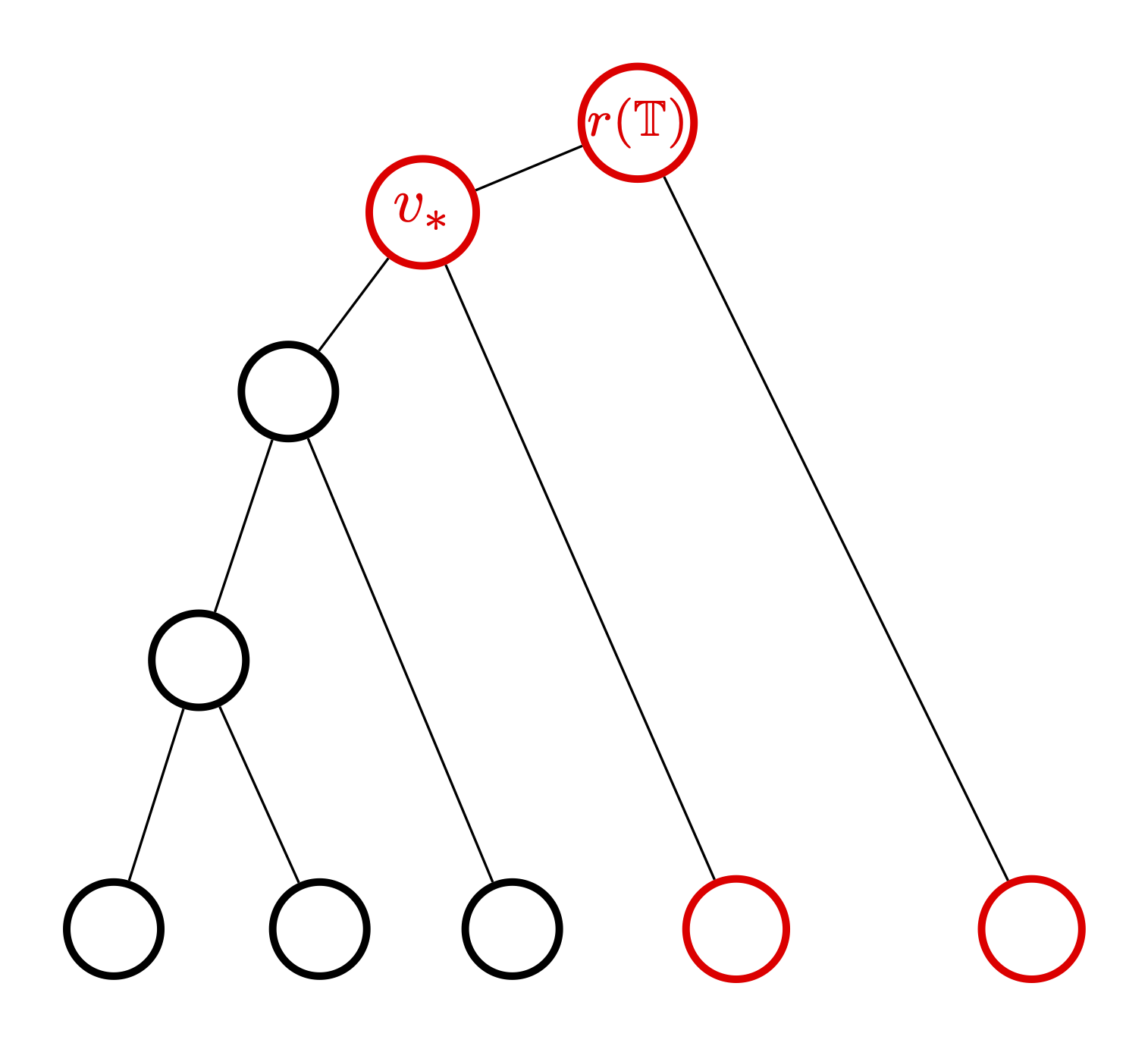}
\caption{A stylized instance of Lemma~\ref{lem:tree-needed} in the case $D=1$ and $[a,b]=[0,1]$ is displayed. By definition of branching depth, when $D=1$ the non-leaves of $\bbT$ consist of a single path. We choose a vertex $v_*$ along this path with small depth $|v_*|=a_2$, and embed $v_*$ to have energy at least $(\ALG(a_2)+2\eps)N$ using Lemma~\ref{lem:exact-embedding}. The leaves with parent on the segment connecting $v_*$ to $r(\bbT)$ (shown in red) can be embedded one at a time using Lemma~\ref{lem:exact-embedding}. The remaining subtree under $v_*$ is embedded all at once using Proposition~\ref{prop:tree-embedding-alg}. This results in a Euclidean embedding $\iota:V(\bbT)\to\bbR^N$ satisfying $H_N(\iota(v))\geq (\ALG+\eps)N$ for all $v\in L(\bbT)$. For $D>1$, we repeat this idea recursively.}
\label{fig:ladder-pic}
\end{framed}

\end{figure}

In the proofs below, we will repeatedly use the principle that $\bbT$ can be subdivided by placing additional vertices on the edges of $\bbT$. This only makes constructing an Euclidean embedding more difficult. In particular, we may assume that all leaves have an ancestor of any given height. We sometimes make this explicit by considering the subgraph $\bbT_{[a,a']}$ of a tree $\bbT$ with range $[a,b]$, for $a<a'<b$. Precisely, $\bbT_{[a,a']}$ is the subgraph of vertices with heights in $[a,a']$, where we implicitly assume via subdivision that each leaf in $L(\bbT)$ has ancestors at heights exactly $a$ and $a'$. Note that unless $a=0$, $\bbT_{[a,a']}$ is not in general a tree but is a disjoint union of dated rooted trees each with range $[a,a']$. We similarly define $\bbT_{[a]}$ to consist of the subset of $V(\bbT)$ at heights exactly $a$, which without loss of generality contains exactly one ancestor of each leaf of $\bbT$.

\begin{lemma}
\label{lem:tree-needed}
Fix a mixture $\xi$, and suppose $\frac{\de^2}{\de t^2}(\xi''(t)^{-1/2})>0$ for $t\in [a,b]\subseteq [0,1]$. Fix $D\in\bbN$ and sufficiently small constants $c,\eps>0$ depending only on $\xi,a,b$ and $D$. Then for any $a_1\in \lt[a,\frac{a+b}{2}\rt]$, any $k\in \bbN$, and any dated rooted tree $\bbT$ with range $[a,b]$, with probability $1-O(e^{-cN})$ over the random choice of $H_N$, the following holds. 

For any $\bm\in \sqrt{a_1}\cdot S_N$ and any linear subspace $W\subseteq\R^N$ with $\dim(W)\geq N-k$, there exists a Euclidean embedding 
\[
  \iota:V(\bbT)\to W+\bm
\]
with $\iota(r({\bbT}))=\bm$ such that for all $v\in L(\bbT)$,
\begin{equation}\label{eq:good-embedding-lemma}
  H_N(\iota(v))\geq H_N(\bm)+(\ALG^{\Sp}_{\xi}(b)-\ALG^{\Sp}_{\xi}(a_1))+\eps)N.
\end{equation}
\end{lemma}

\begin{proof}
We proceed by induction on $D$.

\paragraph{Base Case}

In the base case $D=0$, the tree $\bbT$ contains only a single leaf $v$. It then suffices to find a single point $\bsig\in W+\bm$ with $\norm{\bsig}_N^2=b$ such that 
\[
  H_N(\bsig)\geq H_N(\bm)+(\ALG^{\Sp}_{\xi}(b)-\ALG^{\Sp}_{\xi}(a') +\eps)N.
\]
Indeed such a $\bsig$ exists by Lemma~\ref{lem:exact-embedding}. 

\paragraph{Inductive Step}

For $D\geq 1$, assume the result holds for branching depths up to $D-1$. Our strategy is to first embed the path $V_D$ (recall \eqref{eq:def-VD}), and then apply the inductive hypothesis on the remainder of $\bbT$ to complete the embedding. 
We will assume in the remainder of the proof that $H_N$ is $C_1\sqrt{N}$-Lipschitz with respect to the $\norm{\cdot}_2$ norm for some constant $C_1$ as in Proposition~\ref{prop:gradients-bounded}. 

Define $a_2,a_3\in \lt[a_1,\frac{3a+b}{4}\rt]$ such that 
\[
  \sqrt{a_3^2-a_2^2}=\sqrt{a_2^2-a_1^2}\leq \frac{\eps}{4C_1}.
\] 
Let $t=\max_{v\in V_D}|v|$ denote the maximum height of any $v\in V_D$. (Note that $t$ is not affected by adding extraneous vertices to $\bbT$.)

\paragraph{Case $1$: $t\leq a_2$}

Let $v_t\in V_D$ satisfy $|v_t|=t$. Take
\[
  \iota:V_D:\to W+\bm
\]
to be an arbitrary Euclidean embedding with codomain $W+\bm$. Without loss of generality, we may assume that the children of each $v\in V_D$ have height at most $a_3$. Next, extend $\iota$ to a still arbitrary Euclidean embedding on $Q_D$, which consists of $V_D$ together with all children of vertices in $V_D$.

For each vertex $v\in Q_D$, the Lipschitz property implies 
\begin{align*}
  H_N(\iota(v))&\geq H_N(\bm)-C_1\sqrt{a_3^2-a_1^2}N\\
  &\geq H_N(\bm)-\eps_1 N.
\end{align*}
Observe that any $v\in Q_D\backslash V_D$ satisfies $D(v)\leq D-1$. Because of this, we can now apply the inductive hypothesis to each subtree $\bbT_v$ rooted at some $v\in Q_D\backslash V_D$ in an arbitrary order over the $v$'s. More precisely, suppose a Euclidean embedding mapping a subset $U\subseteq V(\bbT)$ to $W+\bm$ is given, and that $U$ contains no strict descendants of $v\in Q_D\backslash V_D$. Then we know that $|v|\leq a_3\leq \frac{3a+b}{4}$. Define the affine subspace 
\[
  W_v=\Span(\iota(u)_{u\in U})^{\perp}.
\] 
Then by the inductive hypothesis (using the same values $a,b$), there exists $\varepsilon_2$ depending only on $\xi,a,b,D-1$ such that $\iota$ extends to a Euclidean embedding
\[
  \iota:V\cup \bbT_v\to W+\bm
\]
such that $\iota(u)\in W_v$ for all $u\in \bbT_v$, and which satisfies
\[
  H_N(\iota(u))\geq H_N(\iota(v))+\lt(\ALG^{\Sp}_{\xi}(b)-\ALG^{\Sp}_{\xi}(|v|+\eps_2)\rt)N,\qquad \forall u\in L(\bbT_v).
\]
In particular, the above procedure can be repeated for each $v$, resulting in an embedding $\iota$ defined on all of $V(\bbT)$. Finally for any $u\in L(\bbT)$, we must have $u\in L(\bbT_v)$ for some $v$ as above, and so
\begin{align*}
  H_N(\iota(u))
  &\geq H_N(\iota(v))+(\ALG^{\Sp}_{\xi}(b)-\ALG^{\Sp}_{\xi}(|v|+\eps_2))N\\
  &\geq H_N(\bm) +(\ALG^{\Sp}_{\xi}(b)-\ALG^{\Sp}_{\xi}(|v|)+\eps_2-\eps_1)N\\
  &\geq H_N(\bm) +(\ALG^{\Sp}_{\xi}(b)-\ALG^{\Sp}_{\xi}(a_3)+\eps_1)N  +\lt(\eps_2-2\eps_1-\xi'(1)\sqrt{a_3^2-a_1^2}\rt)N.
\end{align*}
Note that
\[
  2\eps_1+\xi'(1)\sqrt{a_3^2-a_1^2}\leq \eps_1\cdot \lt(2+\fr{\xi'(1)}{C_1}\rt).
\]
Since $\eps_2$ depended only on $\xi,a,b,D$ and $\eps_1$ was chosen sufficiently small depending on the same parameters, we may assume that 
\[
  \eps_2-2\eps_1-\xi'(1)\sqrt{a_3^2-a_1^2}\geq 0.
\]
Choosing $\eps=\eps_1$ finishes Case 1 of the inductive step.

\paragraph{Case $2$: $t\geq a_2$}

Let $v_*\in V(\bbT)$ denote the unique vertex on $V_D$ at height $a_2$ -- such a $v_*$ exists without loss of generality. Then applying Lemma~\ref{lem:exact-embedding} on $\bbT_{[a_1,a_2]}$, it follows that there exists $\bsig\in W+\bm$ with $||\bsig||_N^2=a_2$ such that 
\[
  H_N(\bsig)\geq H_N(\bm)+(\ALG^{\Sp}_{\xi}(a_2)-\ALG^{\Sp}_{\xi}(a_1)+\eps_2)N
\]
for some $\eps_2$ depending only on $\xi,a,b$. Set $\iota(v_*)=\bsig$. Next we apply Proposition~\ref{prop:tree-embedding-alg} to the subtree $\bbT_{v_*}$ rooted at $v_*$, obtaining a Euclidean embedding 
\[
  \iota:V_D\cup \bbT_{v_*}\to W+\bm
\]
such that
\[
  H_N(\iota(x))\geq H_N(\bm)+(\ALG^{\Sp}_{\xi}(a_2)-\ALG^{\Sp}_{\xi}(a_1)+\eps_3)N
\]
for $\eps_3=\eps_2/2$. Extending to $\iota$ to the remainder of $V(\bbT)$ proceeds exactly as in Case $1$.
\end{proof}

Below we use Lemma~\ref{lem:tree-needed} to show that to rule out energies greater than $\ALG^{\Sp}$, $\bbT$ must have large branching depth when restricted to any height interval on which $\xi''(t)^{-1/2}$ is convex.

\begin{theorem}
\label{thm:tree-needed}
Fix $\xi$ and suppose $\frac{\de^2}{\de t^2}(\xi''(t)^{-1/2})>0$ for all $t\in [a,b]\subseteq [0,1]$. Fix $D\in\bbN$ and sufficiently small constants $c,\eps>0$ depending only on $\xi,a,b$ and $D$. Then for any dated rooted tree $\bbT$ with range $[0,1]$ such that every connected component of $\bbT_{[a,b]}$ has branching depth at most $D$, with probability $1-O(e^{-cN})$ over the random choice of $H_N$, there exists a Euclidean embedding 
\[
  \iota:V(\bbT)\to\bbR^N
\]
such that for all $v\in L(\bbT)$,
\begin{equation}\label{eq:good-embedding}
  H_N(\iota(v))\geq (\ALG^{\Sp}_{\xi}(|v|)+\eps)N.
\end{equation}
\end{theorem}

\begin{proof}[Proof of Theorem~\ref{thm:tree-needed}]

We let $\eps>0$ be sufficiently small throughout the argument. By Proposition~\ref{prop:tree-embedding-alg}, there exists a Euclidean embedding $\iota:\bbT_{[0,a]}\to \bbR^N$ such that for all $v\in \bbT_{[a]}$,
\begin{equation}\label{eq:0-a}
    H_N(\iota_a(v))\geq (\ALG^{\Sp}_{\xi}(a)-\eps)N.
\end{equation}
Here as usual we assume without loss of generality that all leaves in $\bbT$ have an ancestor at height $a$.
Next we extend $\iota$ to a Euclidean embedding
\[
\iota:\bbT_{[0,b]}\to\bbR^N
\]
such that for all $v\in V(\bbT_{[b]})$ with ancestor $u$ at height $a$,
\begin{equation}\label{eq:a-b}
    H_N(\iota(v))\geq H_N(\iota(u)) + (\ALG^{\Sp}_{\xi}(b)-\ALG^{\Sp}_{\xi}(a)+3\eps)N.
\end{equation}
In fact the existence of such an extension follows directly from Lemma~\ref{lem:tree-needed} for $\eps$ sufficiently small. Here as before we repeatedly apply Lemma~\ref{lem:tree-needed} to individual subtrees in $\bbT_{[a,b]}$, using the subspace $W$ in the statement to ensure the orthogonality constraints are satisfied. 

Finally extend $\iota$ to $\bbT_{[b,1]}$ using Lemma~\ref{lem:euclidean-tree-embedding-alg} on each component. The result is that for any $x\in L(\bbT)$ with ancestor $v$ at height $b$,
\begin{equation}\label{eq:b-1}
  H_N(\iota(x))\geq H_N(\iota(v)) + (\ALG^{\Sp}_{\xi}(1)-\ALG^{\Sp}_{\xi}(b)-\eps)N.
\end{equation}
Combining \eqref{eq:0-a}, \eqref{eq:a-b}, and \eqref{eq:b-1} completes the proof.
\end{proof}

In the Guerra-Talagrand interpolation used for our main argument, it was only possible to directly estimate the \emph{average} energy of the replicas instead of the minimum. In the following final result, we show that to force the average of $H_N(v)$ over the leaves $v\in L(\bbT)$ to be small, it is necessary to use a tree which branches a superconstant amount in any height interval $[a,b]$ as above, \emph{on a set of components of $\bbT_{[a,b]}$ ancestral to almost all leaves}.



Let us illustrate the difference between Theorems~\ref{thm:tree-needed} and \ref{thm:many-branches-needed} by an example. Form $\bbT$ by starting with a full symmetric tree as in Proposition~\ref{prop:uniform-multi-opt} and adding many children of the root as additional leaves. Then by construction (recall also Proposition~\ref{prop:multi-opt-tail}), with probability $1-e^{-\Omega(N)}$ any Euclidean embedding $\iota:\bbT\to \bbR^N$ satisfies
\[
    \min_{v\in L(\bbT)} H_N(\iota(v))\leq (\ALG+\eps)N
\]
for $\eps>0$ as in Proposition~\ref{prop:uniform-multi-opt} arbitrarily small given $\xi$. However the same is not true for the average energy. Indeed, Theorem~\ref{thm:tree-needed} with $D=1$ implies that the additional leaves in $\bbT$ can be embedded to each have energy at least $(\ALG+2\eps')N$ where $\eps'>0$ depends only on $\xi$. If the additional leaves form a sufficiently large fraction of $L(\bbT)$, then any Euclidean extension $\iota$ to all of $\bbT$ satisfies
\[
    \frac{1}{|L(\bbT)|}\sum_{v\in L(\bbT)} H_N(\iota(v))\geq (\ALG+\eps')N
\]
assuming $H_N\in K_N$.

\begin{theorem}
\label{thm:many-branches-needed}
Fix a mixture $\xi$ and $\delta > 0$, and suppose $\frac{\de^2}{\de t^2}(\xi''(t)^{-1/2})>0$ for $t\in [a,b]\subseteq [0,1]$.
Fix $D\in\bbN$ and sufficiently small constants $c,\eps>0$ depending only on $\xi,a,b,D$ and $\delta$. 
Consider any tree $\bbT$ with range $[0,1]$ and $|L(\bbT)|=n$ leaves such that for at least $\delta n$ of the leaves $v\in |L(\bbT)|$, the subtree of $\bbT_{[a,b]}$ consisting of ancestors of $v$ has branching depth at most $D$. With probability $1-O(e^{-cN})$ over the random choice of $H_N$, there exists a Euclidean embedding 
\[
  \iota:V(\bbT)\to\bbR^N
\]
such that 
\begin{equation}\label{eq:good-embedding-average}
  \fr{1}{L(\bbT)} \sum_{v\in L(\bbT)} H_N(\iota(v))
  \geq 
  (\ALG^{\Sp}_{\xi}+\eps)N.
\end{equation}

\end{theorem}

\begin{proof}

Take $\eps_0>0$ sufficiently small. For $v\in L(\bbT)$ and $t\in [0,1]$, let $v^t$ denote the ancestor of $v$ at height $v$. As before, Proposition~\ref{prop:tree-embedding-alg} shows that there exists a Euclidean embedding $\iota:\bbT_{[0,a]}\to \bbR^N$ such that for all $v^a\in \bbT_{[a]}$,
\begin{equation}\label{eq:0-a-v2}
    H_N(\iota_a(v^a))\geq (\ALG^{\Sp}_{\xi}(a)-\delta\eps_0)N.
\end{equation}
Let $\wt{\bbT}_{[a,b]}\subseteq \bbT_{[a,b]}$ consist of all connected components in $\bbT_{[a,b]}$ of branching depth at most $D$. 
Next we extend $\iota$ to a Euclidean embedding
\[
\iota:\bbT_{[0,a]}\cup \wt{\bbT}_{[a,b]} \to\bbR^N
\]
such that for all $v^b\in L(\wt{\bbT}_{[a,b]})$ with ancestor $v^a$ at height $a$,
\begin{equation}\label{eq:a-b-v2}
    H_N(\iota(v^b))\geq H_N(\iota(v^a)) + (\ALG^{\Sp}_{\xi}(b)-\ALG^{\Sp}_{\xi}(a)+4\eps_0)N.
\end{equation}
Lemma~\ref{eq:tree-embedding-again} allows $\iota$ to extend to the remainder of $V(\bbT_{[a,b]})$ such that
\begin{equation}\label{eq:a-b-v3}
    H_N(\iota(v^b))\geq H_N(\iota(v^a)) + (\ALG^{\Sp}_{\xi}(b)-\ALG^{\Sp}_{\xi}(a)-\delta\eps_0)N.
\end{equation}
holds for all $v\in V(\bbT_{[a,b]})$. Since at least $\delta |L(\bbT)|$ leaves $v$ satisfy $v^a\in \wt{\bbT}_{[a,b]}$, \eqref{eq:a-b-v2} and \eqref{eq:a-b-v3} imply
\begin{equation}
\label{eq:a-b-average}
\frac{1}{|L(\bbT)|}\sum_{v\in L(\bbT)}
  H_N(\iota(v^a))-H_N(\iota(v^b))\geq (\ALG^{\Sp}_{\xi}(b)-\ALG^{\Sp}_{\xi}(a)+3\delta\eps_0)N
\end{equation}
As before we finish by extending $\iota$ to $\bbT_{[b,1]}$, using Lemma~\ref{lem:euclidean-tree-embedding-alg} one component at a time. Then for any $v\in L(\bbT)$,
\begin{equation}\label{eq:b-1-v2}
  H_N(\iota(v))\geq H_N(\iota(v^b)) + (\ALG^{\Sp}_{\xi}(1)-\ALG^{\Sp}_{\xi}(b)-\delta\eps_0)N.
\end{equation}
By combining \eqref{eq:0-a-v2}, \eqref{eq:a-b-average} and \eqref{eq:b-1-v2}, it follows that the average of $H_N(\iota(v))$ over $v\in L(\bbT)$ is 
\begin{align*}
  \frac{1}{|L(\bbT)|}\sum_{v\in L(\bbT)}[H_N(\iota(v))]
  &\geq \frac{1}{|L(\bbT)|}\sum_{v\in L(\bbT)}
  \bigg(H_N(\iota(v))-H_N(\iota(v^b))
   +H_N(\iota(v^b))-H_N(\iota(v^a))+H_N(\iota(v^a)) \bigg)\\
  &\geq \ALG^{\Sp}_{\xi}(1)+\delta\eps_0.
\end{align*}
Taking $\eps=\delta\eps_0$ completes the proof.
\end{proof}

%% file: tex/8-stable-algorithms.tex
\section{Overlap Concentration of Standard Optimization Algorithms}
\label{sec:overlap-conc-of-algs}

In this section we prove using Gaussian concentration of measure and Kirszbraun's theorem that approximately $\tau$-Lipschitz functions $\cA:\sH_N\to B_N$ are $(\lambda,e^{-c_{\lambda, \tau}N})$ overlap concentrated. 
We also show that common optimization algorithms such as gradient descent, AMP, and Langevin dynamics are approximately Lipschitz.

\subsection{Overlap Concentration of Approximately Lipschitz Algorithms}
\label{subsec:approx-lip}

Recall that we identify each Hamiltonian $H_N$ with its disorder coefficients $(\Gp{p})_{p\in 2\bbN}$, which we concatenate into an infinite vector $\bg = \bg(H_N)$. 
We can define a (possibly infinite) distance on these Hamiltonians by 
\begin{equation}
    \label{eq:hamiltonian-norm}
    \norm{H_N - H_N'}_N = 
    \fr{1}{\sqrt{N}} 
    \norm{\bg(H_N) - \bg(H_N')}_2.
\end{equation}
We consider algorithms $\cA : \sH_N \to B_N$ that are $\tau$-Lipschitz with respect to the $\norm{\cdot}_N$ norms, i.e. $\cA$ satisfying
\begin{equation}
    \label{eq:lp}
    \norm{\cA(H_N) - \cA(H'_N)}_N
    \le 
    \tau \norm{H_N - H'_N}_N.
\end{equation}
for all $H_N, H'_N \in \sH_N$.
This is the same notion of Lipschitz as in Theorem~\ref{thm:main-lip}, though the current scaling with $\norm{\cdot}_N$ norms will be more convenient for proofs.

We will show overlap concentration for the following class of algorithms that relax the Lipschitz condition to a high probability set of inputs.

\begin{definition}
    \label{defn:approx-lp}
    Let $\tau, \nu > 0$. 
    An algorithm $\cA: \sH_N \to B_N$ is $(\tau,\nu)$-approximately Lipschitz if there exists a $\tau$-Lipschitz $\cA' : \sH_N \to B_N$ with
    \begin{equation}
        \label{eq:approx-lp}
        \P\lt[
            \cA(H_N)
            =
            \cA'(H_N)
        \rt]
        \ge 
        1-\nu.
    \end{equation}
\end{definition}

\begin{proposition}
    \label{prop:lp-overlap-conc}
    If $\cA : \sH_N \to B_N$ is $\tau$-Lipschitz, then for all $\lambda > 0$ it is $\lt(\lambda, \exp\lt(-\fr{\lambda^2}{8\tau^2}N\rt)\rt)$ overlap concentrated.
\end{proposition}
\begin{proof}
    We write $\cA(\bg)$ to mean $\cA(H_N)$ for the Hamiltonian $H_N$ with disorder coefficients $\bg = \bg(H_N)$.
    Let $\cA_i(\bg)$ denote the $i$-th coordinate of $\cA(\bg)$, so $\cA(\bg) = (\cA_1(\bg), \ldots, \cA_N(\bg))$. 
    Define the gradient matrix $\nabla \cA(\bg) \in \bbR^{\bbN \times N}$ by
    \[
        \nabla \cA(\bg)
        = 
        \begin{bmatrix}
            \nabla \cA_1(\bg) 
            & 
            \nabla \cA_2(\bg) 
            &
            \cdots
            &
            \nabla \cA_N(\bg) 
        \end{bmatrix}.
    \]
    Because $\cA$ is $\tau$-Lipschitz, we have for all $\bg, \bg' \in \bbR^{\bbN}$ that
    \[
        \tau \ge \fr{\norm{\cA(\bg) - \cA(\bg')}_N}{\norm{\bg - \bg'}_N}.
    \]
    By taking the limit $\bg' \to \bg$ from the best direction, we conclude that for all $\bg \in \bbR^{\bbN}$, 
    \begin{equation}
        \label{eq:lip-implies-smax}
        \tau \ge \smax(\nabla \cA(\bg)),
    \end{equation}
    where $\smax$ denotes the largest singular value.
    
    Consider any $p\in [0,1]$. 
    We can generate $p$-correlated $\gp{1}, \gp{2} \in \bbR^{\bbN}$ by generating i.i.d. $\gb{0}, \gb{1}, \gb{2} \in \bbR^{\bbN}$, each with i.i.d. standard Gaussian entries, and setting, for $i=1,2$,
    \[
        \gp{i} = \sqrt{p} \gb{0} + \sqrt{1-p} \gb{i}.
    \]
    We will apply Gaussian concentration to the function
    \[
        F(\gb{0}, \gb{1}, \gb{2}) 
        = 
        R\lt(\cA(\gp{0}), \cA(\gp{1})\rt),
    \]
    which is a function of i.i.d. standard Gaussians.
    For each $i\in \bbN$, let $\nabla \cA_{\cdot, i}(\bg)$ denote the $i$-th row of $\nabla \cA(\bg)$, i.e.
    \[
        \nabla \cA_{\cdot, i}(\bg) = 
        \begin{bmatrix}
            \pderiv{\cA_1}{\bg_i}(\bg)
            &
            \pderiv{\cA_2}{\bg_i}(\bg)
            &
            \cdots
            &
            \pderiv{\cA_N}{\bg_i}(\bg)
        \end{bmatrix}.
    \]
    We can compute that
    \begin{align}
        \label{eq:fderiv-g0}
        \pderiv{F}{\gb{0}_i}(\gb{0}, \gb{1}, \gb{2}) 
        &= 
        \fr{\sqrt{p}}{N} \lt[
            \nabla \cA_{\cdot, i}(\gp{1}) \cA(\gp{2}) + 
            \nabla \cA_{\cdot, i}(\gp{2}) \cA(\gp{1})
        \rt], \\
        \label{eq:fderiv-g1}
        \pderiv{F}{\gb{1}_i}(\gb{0}, \gb{1}, \gb{2}) 
        &= 
        \fr{\sqrt{1-p}}{N} \nabla \cA_{\cdot, i}(\gp{1}) \cA(\gp{2}), \\
        \label{eq:fderiv-g2}
        \pderiv{F}{\gb{2}_i}(\gb{0}, \gb{1}, \gb{2}) 
        &= 
        \fr{\sqrt{1-p}}{N} \nabla \cA_{\cdot, i}(\gp{2}) \cA(\gp{1}).
    \end{align}
    By the inequality $(x+y) \le 2x^2 + 2y^2$, \eqref{eq:fderiv-g0} implies
    \[
        \pderiv{F}{\gb{0}_i}(\gb{0}, \gb{1}, \gb{2})^2
        \le 
        \fr{2p}{N^2} \lt[
            \lt(\nabla \cA_{\cdot, i}(\gp{1}) \cA(\gp{2})\rt)^2 + 
            \lt(\nabla \cA_{\cdot, i}(\gp{2}) \cA(\gp{1})\rt)^2
        \rt].
    \]
    Similarly, \eqref{eq:fderiv-g1} and \eqref{eq:fderiv-g2} imply
    \begin{align*}
        \pderiv{F}{\gb{1}_i}(\gb{0}, \gb{1}, \gb{2})^2 
        &\le 
        \fr{2(1-p)}{N^2} \lt(\nabla \cA_{\cdot, i}(\gp{1}) \cA(\gp{2})\rt)^2, \\
        \pderiv{F}{\gb{2}_i}(\gb{0}, \gb{1}, \gb{2})^2
        &\le 
        \fr{2(1-p)}{N^2} \lt(\nabla \cA_{\cdot, i}(\gp{2}) \cA(\gp{1})\rt)^2.
    \end{align*}
    Summing over the last three inequalities and over $i\in \bbN$ gives
    \begin{align*}
        \norm{\nabla F(\gb{0}, \gb{1}, \gb{2})}_2^2
        &\le 
        \fr{2}{N^2} 
        \sum_{i\in \bbN}
        \lt(\nabla \cA_{\cdot, i}(\gp{1}) \cA(\gp{2})\rt)^2
        +
        \fr{2}{N^2} 
        \sum_{i\in \bbN}
        \lt(\nabla \cA_{\cdot, i}(\gp{2}) \cA(\gp{1})\rt)^2 \\
        &= 
        \fr{2}{N^2}
        \norm{\nabla \cA(\gp{1}) \cA(\gp{2})}_2^2 
        +
        \fr{2}{N^2}
        \norm{\nabla \cA(\gp{2}) \cA(\gp{1})}_2^2.
    \end{align*}
    Since $\cA(\gp{1}), \cA(\gp{2}) \in B_N$, this implies
    \[
        \norm{\nabla F(\gb{0}, \gb{1}, \gb{2})}_2^2
        \le 
        \fr{2}{N} \smax \lt(\nabla \cA(\gp{1})\rt)^2 +
        \fr{2}{N} \smax \lt(\nabla \cA(\gp{2})\rt)^2
        \le 
        \fr{4\tau^2}{N}
    \]
    for all $\gb{0}, \gb{1}, \gb{2} \in \bbR^{\bbN}$.
    The last inequality uses \eqref{eq:lip-implies-smax}.
    By Gaussian concentration,
    \[
        \P \lt[
            \lt|F(\gb{0}, \gb{1}, \gb{2}) - 
            \E F(\gb{0}, \gb{1}, \gb{2})\rt| 
            \ge \lambda
        \rt] 
        \le 
        \exp\lt(- \fr{\lambda^2}{8\tau^2}N\rt).
    \]
    Note that Gaussian concentration of measure applies in infinite-dimensional abstract Weiner spaces as explained just before \cite[Theorem 2.7]{Ledoux} regarding Equation (2.10) therein. Alternatively if one wishes to avoid infinite-dimensional Gaussian measures, it suffices to prove the present proposition for the (still $\tau$-Lipschitz) conditional expectations
    \[
        \cA_p(H_N)=\E[\cA(H_N)|\Gp{2},\dots,\Gp{p}]
    \]
    and observe that $\lim_{p\to\infty} \cA_p(H_N)=\cA(H_N)$ holds almost surely and in $L^1$. 
\end{proof}

\begin{proposition}
    \label{prop:approx-lp-overlap-conc}
    Suppose $\cA : \sH_N \to B_N$ is $(\tau, \nu)$-approximately Lipschitz.
    Then, for any $\lambda > 0$, it is $\lt(\lambda, \exp\lt(-\fr{(\lambda-4\nu)_+^2}{8\tau^2}N\rt) + 2\nu\rt)$ overlap concentrated.
\end{proposition}
\begin{proof}
    If $\lambda \le 4\nu$ the result is trivial, so suppose $\lambda > 4\nu$. 
    Let $\cA'$ be such that \eqref{eq:approx-lp} holds.
    
    Let $p\in [0,1]$, and let $\HNp{1}, \HNp{2}$ be $p$-correlated.
    We have 
    \[
        \lt|
            R\lt(\cA (\HNp{1}), \cA (\HNp{2})\rt)
            -
            R\lt(\cA' (\HNp{1}), \cA' (\HNp{2})\rt)
        \rt|
        \le 2
    \]
    pointwise.
    Furthermore, \eqref{eq:approx-lp} implies that 
    \begin{equation}
        \label{eq:approx-lp-match-event}
        \cA (\HNp{1}) = \cA' (\HNp{1})
        \qquad
        \text{and}
        \qquad
        \cA (\HNp{2}) = \cA' (\HNp{2})
    \end{equation}
    with probability at least $1-2\nu$.
    So,
    \begin{equation}
        \label{eq:approx-lp-mean}
        \lt|
            \E R\lt(\cA (\HNp{1}), \cA (\HNp{2})\rt)
            -
            \E R\lt(\cA' (\HNp{1}), \cA' (\HNp{2})\rt)
        \rt|
        \le 4 \nu.
    \end{equation}
    By Proposition~\ref{prop:approx-lp-overlap-conc}, we have
    \begin{equation}
        \label{eq:conc-event}
        \lt|
            R\lt(\cA' (\HNp{1}), \cA' (\HNp{2})\rt) - 
            \E R\lt(\cA' (\HNp{1}), \cA' (\HNp{2})\rt)
        \rt|
        \le 
        \lambda-4\nu
    \end{equation}
    with probability at least $1 - \exp\lt(-\fr{(\lambda-4\nu)^2}{8\tau^2}N\rt)$.
    The events \eqref{eq:approx-lp-match-event} and \eqref{eq:conc-event} occur simultaneously with probability at least $1 - \exp\lt(-\fr{(\lambda-4\nu)^2}{8\tau^2}N\rt) - 2\nu$.
    On this event, \eqref{eq:approx-lp-mean} and \eqref{eq:conc-event} imply
    \[
        \lt|
            R\lt(\cA (\HNp{1}), \cA (\HNp{2})\rt) - 
            \E R\lt(\cA (\HNp{1}), \cA (\HNp{2})\rt)
        \rt|
        \le 
        \lambda,
    \]
    as desired.
\end{proof}

\subsection{Standard Deterministic Optimization Algorithms are Approximately Lipschitz}

Fix constants $T_0, T, k_0 \in \bbN$ and $r\in [1, \sqrt{2})$.
We take as initialization a sequence $(\bx^{-T_0},\dots,\bx^{-1})$ of vectors in $B_N$, which is independent of the Hamiltonian $H_N$. 
We consider rather general $k_0$-th order optimization algorithms which compute
\begin{equation}
    \label{eq:opt-form}
    \bx^{t+1}
    =
    f_t\lt(
        \lt(\bx^s\rt)_{-T_0 \le s \le t},
        \lt(\nabla^k H_N(\bx^s)\rt)_{1\le k\le k_0,-T_0\le s\le t}
    \rt),
    \quad 
    0\le t\le T-1
\end{equation}
and output $\bx^T$.
Here, $(f_0,f_1,\ldots,f_{T-1})$ is a deterministic sequence of functions such that $f_0,\ldots,f_{T-2}$ have codomain $rB_N$, $f_{T-1}$ has codomain $B_N$, and these functions are all Lipschitz in the sense that there exist constants $c_0,\ldots,c_{T-1} > 0$ such that
\begin{align}
    \notag
    &\norm{f_t \lt(
        \lt(\bx^s\rt)_{-T_0 \le s \le t},
        \lt(A^s_k\rt)_{1\le k\le k_0, -T_0 \le s\le t}
    \rt)
    -
    f_t \lt(
        \lt(\by^s\rt)_{-T_0 \le s \le t},
        \lt(B^s_k\rt)_{1\le k\le k_0, -T_0 \le s\le t}
    \rt)}_N \\
    \label{eq:ft-lipschitz}
    &\le 
    c_t \lt[
        \sum_{s=-T_0}^t \norm{\bx^s - \by^s}_N +
        \sum_{k=1}^{k_0} \sum_{s=-T_0}^t \norm{A^s_k - B^s_k}_{\op}
    \rt].
\end{align}

As we review below, the majority of standard convex optimization algorithms fall into this class. However we remark that some optimization algorithms for highly smooth and convex functions, such as Newton's method and the recent advances \cite{gasnikov2019near,nesterov2020superfast}, do not fall into this class. This is because they require inverting a Hessian matrix or solving another inverse problem each iteration. 

\begin{example}
    Projected gradient descent is of the form in \eqref{eq:opt-form} via
    \[
        f_k=\rho\lt(\bx^k-\eta_k\nabla H_N(\bx^k)\rt).
    \]
    Here $\rho$ is the projection map onto either $B_N$ or $C_N$ and the learning rate parameters $(\eta_1,\dots,)$ are arbitrary constants. 
    Other variants such as accelerated gradient descent, ISTA, and FISTA (see e.g. \cite{bubeck2014convex}) can similarly be expressed in the form \eqref{eq:opt-form}.
\end{example}

\begin{example}
    Approximate message passing (AMP) with arbitrary Lipschitz non-linearities can be expressed in the form of \eqref{eq:opt-form}. 
    Given a deterministic sequence of Lipschitz functions $f_t:\bbR^{t+1}\to \bbR$ for each $t\ge 0$, the AMP iterates are given by
    \begin{align}
        \label{eq:ampdef}
        \bx^{t+1}
        &=
        \nabla \wtH_N(f_t(\bx^0,\dots,\bx^t)) - 
        \sum_{s=1}^t d_{t,s}f_{s-1}(\bx^0,\dots,\bx^{s-1}), \\
        d_{t,s} 
        &= 
        \xi''\lt(R \lt(
            f_{t}(\bx^0,\dots,\bx^{t}),
            f_{s-1}(\bx^0,\dots,\bx^{s-1})
        \rt)\rt)
        \cdot 
        \E\lt[
            \fr{\partial f_t}{\partial X^s}(X^0,\dots,X^t)
        \rt].
    \end{align}
    Here $X^0\sim p_0$ is a uniformly bounded random variable, and $\bx^0$ has i.i.d. coordinates generated from the same law. 
    The non-linearities $f_t$ are applied entry-wise as functions $f_t:\R^{N\times (t+1)}\to \R^N$. 
    The sequence $(X^{t})_{t\ge 1}$ is an independent centered Gaussian process with covariance $Q_{t,s}=\E[X^{t}X^s]$ defined recursively by
    \begin{equation}
        Q_{t+1, s+1}
        =
        \xi'\lt(
            \E
            \lt[
                f_{t}\lt(X^{0}, \ldots, X^{t}\rt) f_{s}\lt(X^{0}, \ldots, X^{s}\rt)
            \rt]
        \rt), 
        \quad t,s \ge 0.
    \end{equation}
    It is not difficult to see that the iteration \eqref{eq:ampdef} is captured by \eqref{eq:opt-form}, by defining the non-linearities $f_{t}(\bx^0,\dots,\bx^{t})$ as additional iterates $\bx^{\ell}$ so that their gradients can be evaluated.
\end{example}

\begin{theorem}
    \label{thm:opt-algs-approx-lp}
    For any functions $f_0,\ldots,f_{T-1}$ as above and any initialization $(\bx^{-T_0},\ldots,\bx^{-1})$ of vectors in $B_N$, there exist constants $\tau, c$ such that the map $H_N\to \bx^T$ defined by the iteration \eqref{eq:opt-form} is $(\tau,\nu)$-approximately Lipschitz with $\nu = e^{-cN}$.
\end{theorem}

\begin{proof}
    We will first show the existence of $\tau$ such that the map $H_N \to \bx^T$, with domain restricted to $K_N$ (recall Proposition~\ref{prop:hnp-op}), is $\tau$-Lipschitz with respect to the $\norm{\cdot}_N$ norms.
    Consider running the iteration \eqref{eq:opt-form} on two Hamiltonians $H_N, H'_N \in K_N$ with the same initializaton $(\bx^{-T_0},\ldots,\bx^{-1})$; call the respective iterates $\bx^0, \ldots, \bx^T$ and $\by^0, \ldots, \by^T$. 
    A straightforward induction using Proposition~\ref{prop:gradstable} and \eqref{eq:ft-lipschitz} gives constants $C_0, \ldots, C_T$ such that for $0\le t\le T$, 
    \[
        \norm{\bx^t - \by^t}_N \le C_t \norm{H_N - H'_N}_N.
    \]
    In particular, we may take $\tau = C_T$.
    
    By Kirszbraun's theorem, there exists a $\tau$-Lipschitz $\cA'$ such that $\cA(H_N) = \cA'(H_N)$ for $H_N \in K_N$.
    By Proposition~\ref{prop:hnp-op}, there exists $c$ such that $\P(H_N \in K_N) \ge 1 - e^{-cN}$.
    Therefore $\cA$ is $(\tau, \nu)$-approximately Lipschitz for $\nu = e^{-cN}$.
\end{proof}

The following corollary follows immediately from Theorem~\ref{thm:opt-algs-approx-lp} and Proposition~\ref{prop:approx-lp-overlap-conc}.
\begin{corollary}
    For any functions $f_0,\ldots,f_{T-1}$ as above and any initialization $(\bx^{-T_0},\ldots,\bx^{-1})$ of vectors in $B_N$, for every $\lambda > 0$ there exists a constant $c_\lambda$ such that for sufficiently large $N$, the map $H_N\to \bx^T$ defined by the iteration \eqref{eq:opt-form} is $(\lambda,e^{-c_\lambda N})$ overlap concentrated.
\end{corollary}

\subsection{Reflected Langevin Dynamics are Approximately Lipschitz}

Here we show that a natural version of Langevin dynamics, run for bounded time, is approximately Lipschitz for almost any realization of the driving Brownian motion and hence falls into the scope of our main results. 
The Langevin dynamics for a Hamiltonian $H_N$ are given by the diffusion
\[
    \de X_t=\fr{\beta}{2}\nabla H_N\de t+ \de B_t.
\]
When $X_t$ can range over all of space, the SDE above may explode to infinity in finite time. We therefore modify the na\"ive dynamics above by enforcing an inward-normal reflecting boundary for the convex body $\cK=rB_N$ or $\cK=rC_N$. We refer the reader to \cite{pilipenko2014introduction} for the relevant definitions. In short, the result is a stochastic differential equation of the form
\begin{equation}\label{eq:SDEreflection}
    \de X_t=\fr{\beta}{2}\nabla H_N(X_t)\de t+ \de B_t-v_t\de \ell_t.
\end{equation}
Here $\ell_t$ is non-decreasing and only increases at times when $X_t\in\partial\cK$. Meanwhile $v_t\in \bbR^N$ is contained in the outward normal cone of $X_t\in\partial \cK$ for all $t$. Note that there may be several inequivalent choices for such a reflected process; our results apply to any of these choices. The Langevin dynamics we consider consists of solving \eqref{eq:SDEreflection} for a constant time $T$ starting from $X_0$ which is independent of $H_N$, and then projecting $X_T$ onto $B_N$ or $C_N$.

The corresponding Skorokhod problem was shown to have a Lipschitz solution for convex polyhedra such as $rC_N$ in \cite[Proposition 2.2]{dupuis1991lipschitz}. 
In this case, solving \eqref{eq:SDEreflection} reduces to solving an SDE with Lipschitz coefficients as explained in \cite[Section 2.2]{pilipenko2014introduction}. As a result, the solutions to \eqref{eq:SDEreflection} from different starting points $X_0$ (but with a shared Brownian motion) can be coupled together to give a continuous stochastic flow (see \cite[Chapter 5, Section 13]{rogers1994diffusions}). 
In the case of a smooth boundary such as $B_N$, although the Skorokhod problem does not have a Lipschitz solution, the results of \cite{lions1984stochastic} imply the existence of a stochastic flow as explained in \cite{burdzy2009differentiability}.

\begin{lemma}
\label{lem:hardshrinkage}
    Let $X_t,Y_t$ solve \eqref{eq:SDEreflection} inside a convex body $\cK$ with the same Brownian motion. Then 
    \[
        \int_0^t 
        \lt\la 
            X_t-Y_t,
            v_t^X\de \ell_t^X -
            v_t^Y\de \ell_t^Y 
        \rt\ra
        \ge 0.
    \]
    Here $(v_t^X,\ell_t^X)$ denote the reflecting boundary terms for $X_t$ and similarly for $Y_t$.
\end{lemma}

\begin{proof}
    Recall that $\ell_t^X,\ell_t^Y$ are increasing. Moreover $\la X_t-Y_t,v_t^X\ra\ge 0$ whenever $X_t\in\partial\cK$ by the definition of the normal cone, and similarly $\la Y_t-X_t,v_t^Y\ra\ge 0$ whenever $Y_t\in\partial\cK$. The result follows.
\end{proof}

\begin{theorem}
    Both variants of Langevin dynamics above define, for any initialization $X_0\in B_N$ and for almost every path $(B_t)_{t\in [0,T]}$, a $(\tau,\nu)$ approximately Lipschitz map $\cA:\sH_N\to B_N$ with $\tau=O_{\xi,h,T}(1)$ and $\nu\le e^{-\Omega(N)}$.
\end{theorem}

\begin{proof}
    Fix Hamiltonians
    \[
        H_N^X,H_N^Y\in K_N\subseteq \sH_N
    \]
    satisfying $\norm{H_N^X-H_N^Y}_N = \Delta$.
    Let $X_t,Y_t$ be the solutions to \eqref{eq:SDEreflection} driven by a shared Brownian motion with $H_N^X$ and $H_N^Y$ for $H_N$ respectively, and with shared initial condition $X_0=Y_0$. 
    We will show that 
    \[
        \norm{X_T-Y_T}_N\le C\Delta
    \]
    holds almost surely for some constant $C=C(\xi,h,T)$. This suffices to imply the result. (Note that $\cA$ might not be defined on all of $\sH_N$, but it suffices for it to be well-defined and Lipschitz on $K_N$.)
    
    First observe that $X_t-Y_t$ is a finite variation process, i.e. it has no Brownian component. With $\ell^X$ and $\ell^Y$ the corresponding finite variation processes in \eqref{eq:SDEreflection}, Ito's formula gives
    \begin{align*}
        \fr{1}{2}\de \norm{X_t-Y_t}_2^2
        &=
        \lt\la
            X_t-Y_t,
            \de X_t - \de Y_t
        \rt\ra \de t \\
        &= 
        \la 
            X_t-Y_t,
            -v_t^X\de\ell_t^X + 
            v_t^Y\de\ell_t^Y
        \ra 
        \de t + 
        \beta\la 
            X_t-Y_t,
            \nabla H_N^X(X_t) - 
            \nabla H_N^Y(Y_t)
        \ra \de t.
    \end{align*}
    Integrating and using Lemma~\ref{lem:hardshrinkage}, we find
    \begin{align*}
        \norm{X_t-Y_t}_2^2 
        &\le 
        \int_0^t 
        \la 
            X_s-Y_s,
            -v_s^X\de\ell_s^X + 
            v_s^Y\de\ell_s^Y
        \ra
        \de s + 
        \beta \int_0^t 
        \la 
            X_s-Y_s,
            \nabla H_N^X(X_s) - 
            \nabla H_N^Y(Y_s)
        \ra \de s \\
        &\le 
        \beta \int_0^t 
        \la 
            X_s-Y_s,
            \nabla H_N^X(X_s) - 
            \nabla H_N^Y(Y_s)
        \ra \de s.
    \end{align*}
    By Proposition~\ref{prop:gradstable} with $C=C_1'$,
    \[
        \norm{\nabla H_N^X(X_t)-\nabla H_N^Y(Y_t)}_N
        \le 
        C(\Delta+\norm{X_t-Y_t}_N).
    \]
    Using AM-GM and rescaling, we obtain for each $t\in [0,T]$ the self-bounding inequality
    \begin{align*}
        \norm{X_t-Y_t}_N^2
        &\le 
        C
        \int_0^t 
        \Delta\norm{X_s-Y_s}_N +
        \norm{X_s-Y_s}_N^2 
        \de t\\
        &\le 
        2C
        \int_0^t 
        \Delta^2 + \norm{X_s-Y_s}_N^2 
        \de t\\
        &\le 2C\Delta^2 T+ 2C\int_0^t \norm{X_s-Y_s}_N^2 \de t.
    \end{align*}
    Gr\"onwall's inequality now implies $\norm{X_T-Y_T}_N^2\le 2C\Delta^2T e^{2CT}$. This concludes the proof.
\end{proof}

%% file: tex/ack.tex
\section*{Acknowledgements}

Brice Huang was supported by an NSF graduate research fellowship, a Siebel scholarship, NSF awards DMS-2022448 and CCF-1940205, and NSF TRIPODS award 1740751.
Mark Sellke was supported by an NSF graduate research fellowship, the William R. and Sara Hart Kimball Stanford graduate fellowship, and NSF award CCF-2006489. 
Part of this work was conducted while the authors were visiting the Simons Institute for the Theory of Computing. 
We thank Guy Bresler, Wei-Kuo Chen, Andrea Montanari, Cristopher Moore, Dmitry Panchenko, Yi Sun, and Lenka Zdeborov\'a for helpful references and comments.

%% file: tex/a1-deferred-proofs.tex
\section{Bounds on Hamiltonian Derivatives}
\label{sec:basic-estimates}

In this section we will prove high-probability bounds on the derivatives of $H_N$, including Proposition~\ref{prop:gradients-bounded}.
We write $H_N(\bsig) = \la \bh, \bsig \ra + \wtH_N(\bsig)$ for
\[
    \wtH_N(\bsig) 
    = 
    \sum_{p\in 2\bbN} \gamma_p H_{N,p}(\bsig),
\]
where the $p$-tensor component is
\[
    H_{N,p}(\bsig) 
    = 
    \fr{1}{N^{(p-1)/2}} 
    \la \Gp{p}, \bsig^{\otimes p} \ra.
\]
By slight abuse of notation, we also denote $\fr{1}{N^{(p-1)/2}} \Gp{p} = H_{N,p}$.

\begin{proposition}
    \label{prop:hnp-op}
    There exists universal constants $c, C > 0$ such that for all sufficiently large $N$,
    \[
        \norm{H_{N,p}}_{\op} \le C \sqrt{p}
    \]
    for all $p\in 2\bbN$ with probability at least $1 - e^{-cN}$
\end{proposition}
\begin{proof}
    By \cite[Equation B.6]{arous2020geometry} with $k=p$, we have for some universal constant $K$ and all $p\in 2\bbN$,
    \[
        \P \lt[
            \norm{H_{N,p}}_{\op}
            \ge 
            2K \sqrt{p}
        \rt]
        \le 
        e^{-K^2 pN/2}.
    \]
    Take $C = 2K$. 
    The result follows by a union bound over $p\in 2\bbN$.
\end{proof}

\begin{proof}[Proof of Proposition~\ref{prop:gradients-bounded}]
    Let $K_N \subseteq \sH_N$ be the set of Hamiltonians $H_N$ satisfying the conclusion of Proposition~\ref{prop:hnp-op}.
    We will take
    \[
        C_{k} 
        = 
        C 
        \sum_{p\in 2\bbN, p\ge k} 
        \gamma_p r^{p-k} 
        p^{\uk} \sqrt{p}
        +
        h \ind\{k=1\},
    \]
    where $C$ is given by Proposition~\ref{prop:hnp-op} and $p^{\uk} = p(p-1)\cdots (p-k+1)$ denotes the $k$-th falling power of $p$.
    This is finite because $r<\sqrt{2}$ and $\sum_{p\in 2\bbN} \gamma_p^2 2^p < \infty$ implies $\limsup_{p\to\infty} \fr{\gamma_{p+2}}{\gamma_p} \le \fr{1}{2}$.

    If $H_N \in K_N$, for each $\bsig^1,\ldots,\bsig^k \in S_N$ we have
    \begin{align*}
        \fr{1}{N} 
        \lt\la 
            \nabla^k \wtH_N(\bx), 
            \bsig^1 \otimes \cdots \otimes \bsig^k
        \rt\ra
        &= 
        \sum_{p\in 2\bbN, p\ge k}
        \fr{\gamma_p}{N}
        \lt\la 
            \nabla^k H_{N,p}(\bx),
            \bsig^1 \otimes \cdots \otimes \bsig^k
        \rt\ra \\
        &=
        \sum_{p\in 2\bbN, p\ge k}
        \fr{\gamma_p p^{\uk}}{N}
        \lt\la 
            H_{N,p},
            \bx^{\otimes (p-k)}
            \otimes
            \bsig^1 \otimes \cdots \otimes \bsig^k
        \rt\ra \\
        &\le 
        \sum_{p\in 2\bbN, p\ge k}
        \gamma_p r^{p-k} p^{\uk}
        \norm{H_{N,p}}_{\op} \\
        &\le
        C 
        \sum_{p\in 2\bbN, p\ge k} 
        \gamma_p r^{p-k} 
        p^{\uk} \sqrt{p},
    \end{align*}
    by Proposition~\ref{prop:hnp-op}.
    Thus
    \[
        \norm{\nabla^k \wtH_N(\bx)}_{\op}
        \le 
        C 
        \sum_{p\in 2\bbN, p\ge k} 
        \gamma_p r^{p-k} 
        p^{\uk} \sqrt{p}.
    \]
    For $k\ge 2$, $\nabla^k H_N(\bx) = \nabla^k \wtH_N(\bx)$, and for $k=1$, $\norm{\nabla H_N(\bx)}_{\op} \le \norm{\nabla \wtH_N(\bx)}_{\op} + h$.
    This proves the first claim.
    Similarly, 
    \begin{align*}
        &\fr{1}{N} 
        \lt\la 
            \nabla^k H_N(\bx) - \nabla^k H_N(\by), 
            \bsig^1 \otimes \cdots \otimes \bsig^k
        \rt\ra \\
        &= 
        \sum_{p\in 2\bbN, p\ge k}
        \fr{\gamma_p}{N}
        \lt\la 
            \nabla^k H_{N,p}(\bx) - \nabla^k H_{N,p}(\by),
            \bsig^1 \otimes \cdots \otimes \bsig^k
        \rt\ra \\
        &=
        \sum_{p\in 2\bbN, p\ge k}
        \fr{\gamma_p p^{\uk}}{N}
        \lt\la 
            H_{N,p},
            \lt(\bx^{\otimes (p-k)} - \by^{\otimes (p-k)}\rt)
            \otimes
            \bsig^1 \otimes \cdots \otimes \bsig^k
        \rt\ra \\
        &=
        \sum_{p\in 2\bbN, p\ge k}
        \fr{\gamma_p p^{\uk}}{N}
        \sum_{j=0}^{p-k-1}
        \lt\la 
            H_{N,p},
            (\bx - \by) 
            \otimes 
            \bx^{\otimes (p-k-1-j)}
            \otimes 
            \by^{\otimes j}
            \otimes
            \bsig^1 \otimes \cdots \otimes \bsig^k
        \rt\ra \\
        &\le 
        \sum_{p\in 2\bbN, p\ge k}
        \gamma_p r^{p-k-1}
        p^{\uk} (p-k) 
        \norm{\bx-\by}_N
        \norm{H_{N,p}}_{\op} \\
        &\le 
        C_{k+1} \norm{\bx-\by}_N,
    \end{align*}
    so $\norm{\nabla^k H_N(\bx) - \nabla^k H_N(\by)}_{\op} \le C_{k+1} \norm{\bx-\by}_N$, proving the second claim.
\end{proof}

\begin{proposition}
    \label{prop:gradstable}
    Fix a model $(\xi,h)$ and a constant $r\in [1,\sqrt{2})$.
    Let $K_N$ be given by Proposition~\ref{prop:gradients-bounded}.
    There exists a sequence of constants $(C'_k)_{k\ge 1}$ independent of $N$ such that for all $H_N, H'_N \in K_N$ and $\bx, \by \in \bbR^N$ with $\norm{\bx}_N, \norm{\by}_N \le r$,
    \[
        \norm{\nabla^k H_N(\bx)-\nabla^k H_N'(\by)}_{\op}
        \le 
        C'_k
        \lt(\norm{\bx-\by}_N + \norm{H_N-H_N'}_N\rt),
    \]
    where $\norm{H_N-H_N'}_N$ is defined by \eqref{eq:hamiltonian-norm}.
\end{proposition}
Note that when $\norm{H_N-H_N'}_N$ is infinite, this proposition is vacuously true.
\begin{proof}
    We have that
    \[
        \norm{\nabla^k H_N(\bx)-\nabla^k H_N'(\by)}_{\op}
        \le 
        \norm{\nabla^k H_N(\bx)-\nabla^k H'_N(\bx)}_{\op}
        +
        \norm{\nabla^k H'_N(\bx)-\nabla^k H'_N(\by)}_{\op},
    \]
    and by \eqref{eq:gradient-lipschitz}, 
    \[
        \norm{\nabla^k H'_N(\bx)-\nabla^k H'_N(\by)}_{\op}
        \le 
        C_{k+1} \norm{\bx-\by}_N.
    \]
    For all $\bsig^1, \ldots, \bsig^k \in S_N$, we have
    \begin{align*}
        \fr{1}{N}
        \lt\la
            \nabla^k H_N(\bx) - \nabla^k H'_N(\bx),
            \bsig^1 \otimes \cdots \otimes \bsig^k
        \rt\ra
        &= 
        \sum_{p\in 2\bbN, p\ge k} 
        \fr{\gamma_p}{N}
        \lt\la 
            \nabla^k H_{N,p}(\bx) - \nabla^k H'_{N,p}(\bx),
            \bsig^1 \otimes \cdots \otimes \bsig^k
        \rt\ra \\
        &= 
        \sum_{p\in 2\bbN, p\ge k} 
        \fr{\gamma_p p^{\uk}}{N}
        \lt\la 
            H_{N,p} - H'_{N,p},
            \bx^{\otimes (p-k)} \otimes 
            \bsig^1 \otimes \cdots \otimes \bsig^k
        \rt\ra \\
        &\le 
        \sum_{p\in 2\bbN, p\ge k} 
        \gamma_p r^{p-k} p^{\uk}
        \norm{H_{N,p} - H'_{N,p}}_{\op}.
    \end{align*}
    Moreover,
    \begin{align*}
        \norm{H_{N,p} - H'_{N,p}}_{\op}
        &= 
        \fr{1}{N^{(p+1)/2}}
        \max_{\bsig^1,\ldots,\bsig^p \in S_N}
        \la \Gp{p} - {\mathbf{G}'}^{(p)}, \bsig^1 \otimes \cdots \otimes \bsig^p \ra \\
        &\le 
        \fr{1}{\sqrt{N}} \norm{\Gp{p} - {\mathbf{G}'}^{(p)}}_2 \\
        &\le 
        \norm{H_N - H'_N}_N.
    \end{align*}
    Thus we have 
    \[
        \fr{1}{N}
        \lt\la
            \nabla^k H_N(\bx) - \nabla^k H'_N(\bx),
            \bsig^1 \otimes \cdots \otimes \bsig^k
        \rt\ra
        \le 
        \sum_{p\in 2\bbN, p\ge k} 
        \gamma_p r^{p-k} p^{\uk}
        \cdot 
        \norm{H_N - H'_N}_N.
    \]
    Because this holds for all $\bsig^1, \ldots, \bsig^k \in S_N$, we have
    \[
        \norm{\nabla^k H_N(\bx)-\nabla^k H'_N(\bx)}_{\op}
        \le 
        \sum_{p\in 2\bbN, p\ge k} 
        \gamma_p r^{p-k} p^{\uk}
        \cdot 
        \norm{H_N - H'_N}_N.
    \]
    The result follows by taking $C'_{k}$ to be the larger of $C_{k+1}$ and $\sum_{p\in 2\bbN, p\ge k} \gamma_p r^{p-k} p^{\uk}$.
\end{proof}

%% file: tex/a2-spherical-minimizer.tex
\section{Explicit Formula for the Spherical Algorithmic Threshold}
\label{sec:alg-sp-value}

In this section, we will prove Proposition~\ref{prop:alg-sp-value}, which gives an explicit formula for $\ALG^{\Sp}_{\xi,h}$. 

We first remark that the $\hq$ defined in the second case of Proposition~\ref{prop:alg-sp-value} exists and is unique.
Define $f(q) = q\xi''(q) - \xi'(q) = \sum_{p\in 2\bbN}p(p-2) \gamma_p^2 q^{p-1}$. 
If we are in the second case of the proposition, then $h^2 + \xi'(1) < \xi''(1)$, so $f(1) > h^2$. 
Since $f(0) = 0 \le h^2$, existence of $\hq$ follows from the Intermediate Value Theorem.
Moreover, $f(1) > h^2 \ge 0$ implies $\gamma_p > 0$ for some $p>2$, so $f(q)$ is strictly increasing for $q\in [0,1]$.
This implies uniqueness.

Recall that the spherical Parisi functional $\Par^{\Sp}$ \eqref{eq:def-parisi-functional-sp} is defined in terms of a function $B_\zeta(t) = B - \int_t^1 \xi''(q)\zeta(q) \diff{q}$. 
As $(B, \zeta)$ ranges over $\cuK(\xi)$, $B_\zeta(t)$ ranges over all continuous, nondecreasing functions from $[0,1]$ to $\bbR_{>0}$.
We can thus reparametrize the minimizaton \eqref{eq:def-alg-sp} as one over continuous and nondecreasing $B : [0,1] \to \bbR_{>0}$. 
By slight abuse of notation, for continuous and nondecreasing $B : [0,1] \to \bbR_{>0}$ define
\[
    \Par^{\Sp}(B) 
    =
    \Par_{\xi, h}^{\Sp}(B) 
    = 
    \fr12 \lt[
        \fr{h^2}{B(0)} + 
        \int_0^1 \lt(
            \fr{\xi''(q)}{B(q)} + B(q) 
        \rt) \diff{q}.
    \rt]
\]

\begin{proof}[Proof of Proposition~\ref{prop:alg-sp-value}]
    We first handle the case $h=0$. 
    By AM-GM, 
    \[
        \Par^{\Sp}(B) 
        =
        \fr12 
        \int_0^1 \lt(
            \fr{\xi''(q)}{B(q)} + B(q) 
        \rt) \diff{q}
        \ge 
        \int_0^1 \xi''(q)^{1/2} \diff{q}.
    \]
    Equality holds when $B(q) = \xi''(q)^{1/2}$ for all $q\in [0,1]$.
    However, this requires $B(0) = 0$, so this objective is not attained, though approximations to this $B$ get arbitrarily close.
    Thus $\ALG^{\Sp} = \int_0^1 \xi''(q)^{1/2} \diff{q}$.
    We will show this $\ALG^{\Sp}$ equals the value claimed.
    If $\gamma_p > 0$ for some $p>2$, then $\xi'(1) < \xi''(1)$, so we are in the second case of the proposition.
    Since $\hq = 0$, we are done.
    Otherwise, $\gamma_p = 0$ for all $p>2$, and $\xi'(1) = \xi''(1)$. 
    Then $\xi''(q)$ is constant, so $\ALG^{\Sp} = \xi''(1)^{1/2} = \xi'(1)^{1/2}$ as claimed.
    
    Otherwise, $h>0$.
    We extend the definition of $\hq$ to
    \[
        \hq = \sup\lt\{
            q\in [0,1]: 
            h^2 + \xi'(q) \ge q \xi''(q)
        \rt\}.
    \]
    This gives $\hq = 1$ in the first case of the proposition, and matches the definition of $\hq$ in the second case.
    Note that $\hq > 0$.
    Define 
    \[
        \hB = \lt(\fr{h^2 + \xi'(\hq)}{\hq}\rt)^{1/2}.
    \]
    We will prove both cases simultaneously by showing that for any continuous and nondecreasing $B : [0,1] \to \bbR_{>0}$, we have
    \[
        \Par^{\Sp}(B) 
        \ge 
        \hq^{1/2} \lt(h^2 + \xi'(\hq)\rt)^{1/2} + 
        \int_{\hq}^1 \xi''(q)^{1/2} \diff{q},
    \]
    with equality if and only if
    \[
        B(q) = 
        \begin{cases}
            \hB & q \le \hq, \\
            \xi''(q)^{1/2} & q > \hq.
        \end{cases}
    \]
    It is easy to check that this $B$ is continuous and nondecreasing (i.e. if $\hq < 1$, then $\hB = \xi''(\hq)^{1/2}$) and that it corresponds to the equality cases claimed in the proposition.
    By AM-GM,
    \begin{equation}
        \label{eq:hq-to-1}
        \fr12 \int_{\hq}^1 \lt(
            \fr{\xi''(q)}{B(q)} + B(q)
        \rt)\diff{q}
        \ge 
        \int_{\hq}^1 \xi''(q)^{1/2} \diff{q},
    \end{equation}
    with equality if and only if $B(q) = \xi''(q)^{1/2}$ on $(\hq,1]$.
    Define the truncated Parisi operator
    \[
        \Par^{\Sp, \hq}(B)
        = 
        \fr12 \lt[
            \fr{h^2}{B(0)} + 
            \int_0^{\hq} \lt(
                \fr{\xi''(q)}{B(q)} + B(q) 
            \rt) \diff{q}
        \rt].
    \]
    Let $\wtB : [0, \hq] \to \bbR_{>0}$ be given by $\wtB(q) = \hB$, and note that $\Par^{\Sp, \hq}(\wtB) = \hq^{1/2} \lt(h^2 + \xi'(\hq)\rt)^{1/2}$. 
    We will show that for continuous and nondecreasing $B : [0, \hq] \to \bbR_{>0}$, we have $\Par^{\Sp, \hq}(B) \ge \Par^{\Sp, \hq}(\wtB)$, with equality if and only if $B \equiv \wtB$ on $[0, \hq]$.
    Along with \eqref{eq:hq-to-1}, this implies the conclusion.
    We consider two cases.
    
    \paragraph{Case 1: $B(0) < \hB$.}
    Define
    \[
        \wtq = \sup\lt\{
            q \in [0, \hq] : 
            B(q) \le \hB
        \rt\}.
    \]
    It is possible that $\wtq = \hq$.
    For $q\in [\wtq, \hq]$, we have $B(q) \ge \hB$, so 
    \[
        \int_{\wtq}^{\hq} \lt(
            \fr{\xi''(q)}{B(q)}
            + B(q)
        \rt)
        -
        \int_{\wtq}^{\hq} \lt(
            \fr{\xi''(q)}{\hB}
            + \hB
        \rt)
        =
        \int_{\wtq}^{\hq} \lt(
            \fr{1}{\hB} - \fr{1}{B(q)}
        \rt) 
        \lt(B(q) \hB - \xi''(q)\rt)
        \diff{q}.
    \]
    Because
    \[
        B(q)\hB \ge \hB^2 \ge \fr{h^2 + \xi'(\hq)}{\hq} \ge \xi''(\hq) \ge \xi''(q),
    \]
    we have
    \begin{equation}
        \label{eq:wtq-hq}
        \int_{\wtq}^{\hq} \lt(
            \fr{\xi''(q)}{B(q)}
            + B(q)
        \rt)
        \ge
        \int_{\wtq}^{\hq} \lt(
            \fr{\xi''(q)}{\hB}
            + \hB
        \rt).
    \end{equation}
    Moreover, for $q\in [0, \wtq]$, we have $B(q) \le \hB$, so
    \begin{align*}
        2\lt(\Par^{\Sp, \wtq}(B) - \Par^{\Sp, \wtq}(\wtB)\rt)
        &= 
        h^2 \lt(\fr{1}{B(0)} - \fr{1}{\hB}\rt)
        -
        \int_0^{\wtq}
        \lt(B(q)\hB - \xi''(q)\rt)
        \lt(\fr{1}{B(q)} - \fr{1}{\hB}\rt) 
        \diff{q} \\
        &\ge
        h^2 \lt(\fr{1}{B(0)} - \fr{1}{\hB}\rt)
        -
        \int_0^{\wtq}
        \lt(\hB^2 - \xi''(q)\rt)
        \lt(\fr{1}{B(q)} - \fr{1}{\hB}\rt) 
        \diff{q} \\
        &=
        h^2 \lt(\fr{1}{B(0)} - \fr{1}{\hB}\rt)
        -
        \int_0^{\wtq}
        \lt(\fr{h^2 + \xi'(\hq)}{\hq} - \xi''(q)\rt)
        \lt(\fr{1}{B(q)} - \fr{1}{\hB}\rt) 
        \diff{q} \\
        &\ge 
        h^2 \lt(\fr{1}{B(0)} - \fr{1}{\hB}\rt)
        -
        \int_0^{\wtq}
        \lt(\fr{h^2 + \xi'(\hq)}{\hq} - \xi''(q)\rt)
        \lt(\fr{1}{B(0)} - \fr{1}{\hB}\rt) 
        \diff{q} \\
        &=
        \lt(\fr{1}{B(0)} - \fr{1}{\hB}\rt) 
        \lt[
            h^2 
            -
            \int_0^{\wtq}
            \lt(\fr{h^2 + \xi'(\hq)}{\hq} - \xi''(q)\rt)
            \diff{q} 
        \rt] \\
        &\ge 
        \lt(\fr{1}{B(0)} - \fr{1}{\hB}\rt) 
        \lt[
            h^2 
            -
            \int_0^{\hq}
            \lt(\fr{h^2 + \xi'(\hq)}{\hq} - \xi''(q)\rt)
            \diff{q} 
        \rt] \\
        &= 0.
    \end{align*}
    Thus $\Par^{\Sp, \wtq}(B) \ge \Par^{\Sp, \wtq}(\wtB)$, with equality only if $\wtq = \hq$ and $B(q) = \hB$ for all $q\in [0, \wtq]$.
    Combining this with \eqref{eq:wtq-hq} gives that $\Par^{\Sp, \hq}(B) \ge \Par^{\Sp, \hq}(\wtB)$, with equality only if $B \equiv \wtB$ on $[0, \hq]$.
    
    \paragraph{Case 2: $B(0) \ge \hB$.}
    In this case, $B(q)\ge \hB$ for all $q\in [0, \hq]$.
    So,
    \begin{align*}
        2\lt(\Par^{\Sp, \hq}(B) - \Par^{\Sp, \hq}(\wtB)\rt)
        &= 
        -h^2 \lt(\fr{1}{\hB} - \fr{1}{B(0)}\rt)
        +
        \int_0^\hq 
        \lt(B(q)\hB - \xi''(q)\rt)
        \lt(\fr{1}{\hB} - \fr{1}{B(q)}\rt)
        \diff{q} \\
        &\ge
        -h^2 \lt(\fr{1}{\hB} - \fr{1}{B(0)}\rt)
        +
        \int_0^\hq 
        \lt(\hB^2 - \xi''(q)\rt)
        \lt(\fr{1}{\hB} - \fr{1}{B(q)}\rt)
        \diff{q} \\
        &= 
        -h^2 \lt(\fr{1}{\hB} - \fr{1}{B(0)}\rt)
        +
        \int_0^\hq 
        \lt(\fr{h^2 + \xi'(\hq)}{\hq} - \xi''(q)\rt)
        \lt(\fr{1}{\hB} - \fr{1}{B(q)}\rt)
        \diff{q} \\
        &\ge 
        -h^2 \lt(\fr{1}{\hB} - \fr{1}{B(0)}\rt)
        +
        \int_0^\hq 
        \lt(\fr{h^2 + \xi'(\hq)}{\hq} - \xi''(q)\rt)
        \lt(\fr{1}{\hB} - \fr{1}{B(0)}\rt)
        \diff{q} \\
        &=
        \lt(\fr{1}{\hB} - \fr{1}{B(0)}\rt) \lt[
            -h^2 
            +
            \int_0^\hq 
            \lt(\fr{h^2 + \xi'(\hq)}{\hq} - \xi''(q)\rt)
            \diff{q} 
        \rt] \\
        &= 0.
    \end{align*}
    For equality to hold, we must have $B(q) = \wtB$ for all $q\in [0, \hq]$, so $B \equiv \wtB$ on $[0, \hq]$.
\end{proof}

%% file: final.bbl
\newcommand{\etalchar}[1]{$^{#1}$}
\begin{thebibliography}{KMRT{\etalchar{+}}07}

\bibitem[ABA13]{auffinger2013complexity}
Antonio Auffinger and G{\'e}rard Ben~Arous.
\newblock Complexity of random smooth functions on the high-dimensional sphere.
\newblock {\em The Annals of Probability}, 41(6):4214--4247, 2013.

\bibitem[ABA{\v{C}}13]{auffinger2013random}
Antonio Auffinger, G{\'e}rard Ben~Arous, and Ji{\v{r}}{\'\i} {\v{C}}ern{\'y}.
\newblock Random matrices and complexity of spin glasses.
\newblock {\em Communications on Pure and Applied Mathematics}, 66(2):165--201,
  2013.

\bibitem[ABE{\etalchar{+}}05]{arora2005non}
Sanjeev Arora, Eli Berger, Hazan Elad, Guy Kindler, and Muli Safra.
\newblock On non-approximability for quadratic programs.
\newblock In {\em Proceedings of 46th FOCS}, pages 206--215. IEEE, 2005.

\bibitem[AC15]{auffinger2015parisi}
Antonio Auffinger and Wei-Kuo Chen.
\newblock {The Parisi formula has a unique minimizer}.
\newblock {\em Communications in Mathematical Physics}, 335(3):1429--1444,
  2015.

\bibitem[AC17a]{auffinger2017energy}
Antonio Auffinger and Wei-Kuo Chen.
\newblock On the energy landscape of spherical spin glasses.
\newblock {\em Advances in Mathematics}, 330, 02 2017.

\bibitem[AC17b]{auffinger2017parisi}
Antonio Auffinger and Wei-Kuo Chen.
\newblock Parisi formula for the ground state energy in the mixed $p$-spin
  model.
\newblock {\em The Annals of Probability}, 45(6b):4617--4631, 2017.

\bibitem[ACO08]{achlioptas2008phasetransitions}
Dimitris Achlioptas and Amin Coja-Oghlan.
\newblock Algorithmic barriers from phase transitions.
\newblock In {\em Proceedings of 49th FOCS}, pages 793--802, 2008.

\bibitem[AJK{\etalchar{+}}21]{anari2021entropic}
Nima Anari, Vishesh Jain, Frederic Koehler, Huy~Tuan Pham, and Thuy-Duong
  Vuong.
\newblock Entropic independence in high-dimensional expanders: Modified
  log-{S}obolev inequalities for fractionally log-concave polynomials and the
  {I}sing model.
\newblock {\em arXiv preprint arXiv:2106.04105}, 2021.

\bibitem[ALS22a]{abbe2022binary}
Emmanuel Abbe, Shuangping Li, and Allan Sly.
\newblock Binary perceptron: efficient algorithms can find solutions in a rare
  well-connected cluster.
\newblock In {\em Proceedings of the 54th Annual ACM SIGACT Symposium on Theory
  of Computing}, pages 860--873, 2022.

\bibitem[ALS22b]{abbe2022proof}
Emmanuel Abbe, Shuangping Li, and Allan Sly.
\newblock Proof of the contiguity conjecture and lognormal limit for the
  symmetric perceptron.
\newblock In {\em 2021 IEEE 62nd Annual Symposium on Foundations of Computer
  Science (FOCS)}, pages 327--338. IEEE, 2022.

\bibitem[AM03]{achlioptas2003almost}
Dimitris Achlioptas and Cristopher Moore.
\newblock Almost all graphs with average degree 4 are 3-colorable.
\newblock {\em Journal of Computer and System Sciences}, 67(2):441--471, 2003.

\bibitem[AM20]{alaoui2020algorithmic}
Ahmed~El Alaoui and Andrea Montanari.
\newblock Algorithmic thresholds in mean field spin glasses.
\newblock {\em arXiv preprint arXiv:2009.11481}, 2020.

\bibitem[AMS21a]{alaoui2021local}
Ahmed~El Alaoui, Andrea Montanari, and Mark Sellke.
\newblock Local algorithms for maximum cut and minimum bisection on locally
  treelike regular graphs of large degree.
\newblock {\em arXiv preprint arXiv:2111.06813}, 2021.

\bibitem[AMS21b]{ams20}
Ahmed~El Alaoui, Andrea Montanari, and Mark Sellke.
\newblock Optimization of mean-field spin glasses.
\newblock {\em Ann. Probab.}, 49(6):2922--2960, 2021.

\bibitem[BADG06]{arous2006cugliandolo}
G{\'e}rard Ben~Arous, Amir Dembo, and Alice Guionnet.
\newblock {Cugliandolo-Kurchan equations for dynamics of spin-glasses}.
\newblock {\em Probability Theory and Related Fields}, 136(4):619--660, 2006.

\bibitem[BAGJ20]{arous2020bounding}
G{\'e}rard Ben~Arous, Reza Gheissari, and Aukosh Jagannath.
\newblock Bounding flows for spherical spin glass dynamics.
\newblock {\em Communications in Mathematical Physics}, 373(3):1011--1048,
  2020.

\bibitem[BASZ20]{arous2020geometry}
G{\'e}rard Ben~Arous, Eliran Subag, and Ofer Zeitouni.
\newblock Geometry and temperature chaos in mixed spherical spin glasses at low
  temperature: the perturbative regime.
\newblock {\em Communications on Pure and Applied Mathematics},
  73(8):1732--1828, 2020.

\bibitem[BBH{\etalchar{+}}12]{barak2012hypercontractivity}
Boaz Barak, Fernando~G.S.L. Brand{\~a}o, Aram~W Harrow, Jonathan Kelner, David
  Steurer, and Yuan Zhou.
\newblock Hypercontractivity, sum-of-squares proofs, and their applications.
\newblock In {\em Proceedings of the forty-fourth annual ACM symposium on
  Theory of computing}, pages 307--326. ACM, 2012.

\bibitem[BCKM98]{bouchaud1998out}
Jean-Philippe Bouchaud, Leticia~F Cugliandolo, Jorge Kurchan, and Marc
  M\'ezard.
\newblock Out of equilibrium dynamics in spin-glasses and other glassy systems.
\newblock {\em Spin glasses and random fields}, pages 161--223, 1998.

\bibitem[B{\v{C}}NS22]{belius2021triviality}
David Belius, Ji{\v{r}}{\'\i} {\v{C}}ern{\`y}, Shuta Nakajima, and Marius~A
  Schmidt.
\newblock {Triviality of the Geometry of Mixed $p$-Spin Spherical Hamiltonians
  with External Field}.
\newblock {\em Journal of Statistical Physics}, 186(1):1--34, 2022.

\bibitem[BD98]{bocker1998recovering}
Sebastian B{\"o}cker and Andreas~W.M. Dress.
\newblock Recovering symbolically dated, rooted trees from symbolic
  ultrametrics.
\newblock {\em Advances in mathematics}, 138(1):105--125, 1998.

\bibitem[BH21]{bresler2021ksat}
Guy Bresler and Brice Huang.
\newblock {The Algorithmic Phase Transition of Random k-SAT for Low Degree
  Polynomials}.
\newblock In {\em Proceedings of 62nd FOCS}, pages 298--309, 2021.

\bibitem[BIL{\etalchar{+}}15]{baldassi2015subdominant}
Carlo Baldassi, Alessandro Ingrosso, Carlo Lucibello, Luca Saglietti, and
  Riccardo Zecchina.
\newblock Subdominant dense clusters allow for simple learning and high
  computational performance in neural networks with discrete synapses.
\newblock {\em Physical review letters}, 115(12):128101, 2015.

\bibitem[Bor75]{borell1975brunn}
Christer Borell.
\newblock The {B}runn-{M}inkowski inequality in {G}auss space.
\newblock {\em Inventiones mathematicae}, 30(2):207--216, 1975.

\bibitem[Bub15]{bubeck2014convex}
S{\'e}bastien Bubeck.
\newblock Convex optimization: Algorithms and complexity.
\newblock {\em Foundations and Trends in Machine Learning}, 8(3-4):231--357,
  2015.

\bibitem[Bur09]{burdzy2009differentiability}
Krzysztof Burdzy.
\newblock Differentiability of stochastic flow of reflected brownian motions.
\newblock {\em Electronic Journal of Probability}, 14:2182--2240, 2009.

\bibitem[CCM21]{celentano2021high}
Michael Celentano, Chen Cheng, and Andrea Montanari.
\newblock The high-dimensional asymptotics of first order methods with random
  data.
\newblock {\em arXiv preprint arXiv:2112.07572}, 2021.

\bibitem[CGPR19]{chen2019suboptimality}
Wei-Kuo Chen, David Gamarnik, Dmitry Panchenko, and Mustazee Rahman.
\newblock Suboptimality of local algorithms for a class of max-cut problems.
\newblock {\em The Annals of Probability}, 47(3):1587--1618, 2019.

\bibitem[Cha09]{chatterjee2009disorder}
Sourav Chatterjee.
\newblock Disorder chaos and multiple valleys in spin glasses.
\newblock {\em arXiv preprint arXiv:0907.3381}, 2009.

\bibitem[Che17]{chen2017variational}
Wei-Kuo Chen.
\newblock {Variational representations for the Parisi functional and the
  two-dimensional Guerra--Talagrand bound}.
\newblock {\em The Annals of Probability}, 45(6A):3929--3966, 2017.

\bibitem[CHL18]{chen2018energy}
Wei-Kuo Chen, Madeline Handschy, and Gilad Lerman.
\newblock On the energy landscape of the mixed even p-spin model.
\newblock {\em Probability Theory and Related Fields}, 171(1-2):53--95, 2018.

\bibitem[CIS76]{cirel1976norms}
Boris~S. Cirel'son, Ildar~A. Ibragimov, and V.N. Sudakov.
\newblock Norms of {G}aussian sample functions.
\newblock In {\em Proceedings of the Third Japan—USSR Symposium on
  Probability Theory}, pages 20--41. Springer, 1976.

\bibitem[CK94]{cugliandolo1994out}
Leticia~F. Cugliandolo and Jorge Kurchan.
\newblock {On the out-of-equilibrium relaxation of the Sherrington-Kirkpatrick
  model}.
\newblock {\em Journal of Physics A: Mathematical and General}, 27(17):5749,
  1994.

\bibitem[CLR03]{crisanti2003complexity}
Andrea Crisanti, Luca Leuzzi, and Tommaso Rizzo.
\newblock {The complexity of the spherical $p$-spin spin glass model,
  revisited}.
\newblock {\em The European Physical Journal B-Condensed Matter and Complex
  Systems}, 36(1):129--136, 2003.

\bibitem[CLR05]{crisanti2005complexity}
Andrea Crisanti, Luca Leuzzi, and Tommaso Rizzo.
\newblock Complexity in mean-field spin-glass models: Ising p-spin.
\newblock {\em Physical Review B}, 71(9):094202, 2005.

\bibitem[COE15]{cojaoghlan2015independent}
Amin Coja-Oghlan and Charilaos Efthymiou.
\newblock On independent sets in random graphs.
\newblock {\em Random Structures \& Algorithms}, 47(3):436--486, 2015.

\bibitem[CPS18]{chen2018generalized}
Wei-Kuo Chen, Dmitry Panchenko, and Eliran Subag.
\newblock {The Generalized TAP Free Energy}.
\newblock {\em Communications on Pure and Applied Mathematics}, 2018.

\bibitem[CS17]{chen2017parisi}
Wei-Kuo Chen and Arnab Sen.
\newblock Parisi formula, disorder chaos and fluctuation for the ground state
  energy in the spherical mixed p-spin models.
\newblock {\em Communications in Mathematical Physics}, 350(1):129--173, 2017.

\bibitem[CS21]{chatterjee2019average}
Sourav Chatterjee and Leila Sloman.
\newblock Average {G}romov hyperbolicity and the {P}arisi ansatz.
\newblock {\em Advances in Mathematics}, 376:107417, 2021.

\bibitem[DEZ15]{ding2015multiple}
Jian Ding, Ronen Eldan, and Alex Zhai.
\newblock On multiple peaks and moderate deviations for the supremum of a
  gaussian field.
\newblock {\em The Annals of Probability}, 43(6):3468--3493, 2015.

\bibitem[DI91]{dupuis1991lipschitz}
Paul Dupuis and Hitoshi Ishii.
\newblock On lipschitz continuity of the solution mapping to the skorokhod
  problem, with applications.
\newblock {\em Stochastics: An International Journal of Probability and
  Stochastic Processes}, 35(1):31--62, 1991.

\bibitem[DMS17]{dembo2017extremal}
Amir Dembo, Andrea Montanari, and Subhabrata Sen.
\newblock Extremal cuts of sparse random graphs.
\newblock {\em The Annals of Probability}, 45(2):1190--1217, 2017.

\bibitem[DS19]{ding2019capacity}
Jian Ding and Nike Sun.
\newblock Capacity lower bound for the ising perceptron.
\newblock In {\em Proceedings of the 51st Annual ACM SIGACT Symposium on Theory
  of Computing}, pages 816--827, 2019.

\bibitem[EKZ21]{eldan2021spectral}
Ronen Eldan, Frederic Koehler, and Ofer Zeitouni.
\newblock A spectral condition for spectral gap: fast mixing in
  high-temperature ising models.
\newblock {\em Probability Theory and Related Fields}, pages 1--17, 2021.

\bibitem[FS18]{fefferman2018sharp}
Charles Fefferman and Pavel Shvartsman.
\newblock Sharp finiteness principles for {L}ipschitz selections.
\newblock {\em Geometric and Functional Analysis}, 28(6):1641--1705, 2018.

\bibitem[Fyo13]{fyodorov2013high}
Yan~V. Fyodorov.
\newblock High-dimensional random fields and random matrix theory.
\newblock {\em arXiv preprint arXiv:1307.2379}, 2013.

\bibitem[Gam21]{gamarnik2021survey}
David Gamarnik.
\newblock The overlap gap property: A topological barrier to optimizing over
  random structures.
\newblock {\em Proceedings of the National Academy of Sciences}, 118(41), 2021.

\bibitem[GDG{\etalchar{+}}19]{gasnikov2019near}
Alexander Gasnikov, Pavel Dvurechensky, Eduard Gorbunov, Evgeniya Vorontsova,
  Daniil Selikhanovych, C{\'e}sar~A Uribe, Bo~Jiang, Haoyue Wang, Shuzhong
  Zhang, S{\'e}bastien Bubeck, et~al.
\newblock Near optimal methods for minimizing convex functions with lipschitz $
  p $-th derivatives.
\newblock In {\em Conference on Learning Theory}, pages 1392--1393. PMLR, 2019.

\bibitem[GJ21]{gamarnik2019overlap}
David Gamarnik and Aukosh Jagannath.
\newblock The overlap gap property and approximate message passing algorithms
  for $ p $-spin models.
\newblock {\em The Annals of Probability}, 49(1):180--205, 2021.

\bibitem[GJW20]{gamarnik2020optimization}
David Gamarnik, Aukosh Jagannath, and Alexander~S. Wein.
\newblock Low-degree hardness of random optimization problems.
\newblock In {\em Proceedings of 61st FOCS}, pages 131--140. IEEE, 2020.

\bibitem[GJW21]{gamarnik2021circuit}
David Gamarnik, Aukosh Jagannath, and Alexander~S. Wein.
\newblock Circuit lower bounds for the $p$-spin optimization problem.
\newblock {\em arXiv preprint arXiv:2109.01342}, 2021.

\bibitem[GK21a]{gamarnik2021computing}
David Gamarnik and Eren~C. K{\i}z{\i}lda{\u{g}}.
\newblock {Computing the partition function of the Sherrington--Kirkpatrick
  model is hard on average}.
\newblock {\em The Annals of Applied Probability}, 31(3):1474--1504, 2021.

\bibitem[GK21b]{gamarnik2021partitioning}
David Gamarnik and Eren~C. K{\i}z{\i}lda\u{g}.
\newblock Algorithmic obstructions in the random number partitioning problem.
\newblock {\em arXiv preprint arXiv:2103.01369}, 2021.

\bibitem[GS14]{gamarnik2014limits}
David Gamarnik and Madhu Sudan.
\newblock Limits of local algorithms over sparse random graphs.
\newblock In {\em Proceedings of the 5th conference on Innovations in
  theoretical computer science}, pages 369--376. ACM, 2014.

\bibitem[GS17]{gamarnik2017performance}
David Gamarnik and Madhu Sudan.
\newblock Performance of sequential local algorithms for the random
  {NAE}-{$K$}-sat problem.
\newblock {\em SIAM Journal on Computing}, 46(2):590--619, 2017.

\bibitem[Jag17]{jagannath2017approximate}
Aukosh Jagannath.
\newblock Approximate ultrametricity for random measures and applications to
  spin glasses.
\newblock {\em Communications on Pure and Applied Mathematics}, 70(4):611--664,
  2017.

\bibitem[JT16]{jagannath2016dynamic}
Aukosh Jagannath and Ian Tobasco.
\newblock A dynamic programming approach to the parisi functional.
\newblock {\em Proceedings of the American Mathematical Society},
  144(7):3135--3150, 2016.

\bibitem[Kiv21]{kivimae2021ground}
Pax Kivimae.
\newblock The ground state energy and concentration of complexity in spherical
  bipartite models.
\newblock {\em arXiv preprint arXiv:2107.13138}, 2021.

\bibitem[KMRT{\etalchar{+}}07]{krzakala2007gibbs}
Florent Krzakala, Andrea Montanari, Federico Ricci-Tersenghi, Guilhem
  Semerjian, and Lenka Zdeborov{\'a}.
\newblock Gibbs states and the set of solutions of random constraint
  satisfaction problems.
\newblock {\em Proceedings of the National Academy of Sciences},
  104(25):10318--10323, 2007.

\bibitem[Led01]{Ledoux}
M.~Ledoux.
\newblock {The concentration of measure phenomenon}.
\newblock In {\em Mathematical Surveys and Monographs}, volume~89. {American
  Mathematical Society, Providence, RI}, 2001.

\bibitem[LMP15]{lunardi2015infinite}
Alessandra Lunardi, Michele Miranda, and Diego Pallara.
\newblock Infinite dimensional analysis.
\newblock In {\em 19th Internet Seminar}, volume 2016, 2015.

\bibitem[LS84]{lions1984stochastic}
Pierre-Louis Lions and Alain-Sol Sznitman.
\newblock Stochastic differential equations with reflecting boundary
  conditions.
\newblock {\em Communications on Pure and Applied Mathematics}, 37(4):511--537,
  1984.

\bibitem[McK21]{mckenna2021complexity}
Benjamin McKenna.
\newblock Complexity of bipartite spherical spin glasses.
\newblock {\em arXiv preprint arXiv:2105.05043}, 2021.

\bibitem[Mon19]{mon18}
Andrea Montanari.
\newblock {Optimization of the Sherrington-Kirkpatrick Hamiltonian}.
\newblock In {\em IEEE Symposium on the Foundations of Computer Science, FOCS},
  November 2019.

\bibitem[Nes21]{nesterov2020superfast}
Yurii Nesterov.
\newblock Superfast second-order methods for unconstrained convex optimization.
\newblock {\em Journal of Optimization Theory and Applications}, pages 1--30,
  2021.

\bibitem[Pan13a]{panchenko2013parisi}
Dmitry Panchenko.
\newblock {The Parisi ultrametricity conjecture}.
\newblock {\em Annals of Mathematics}, pages 383--393, 2013.

\bibitem[Pan13b]{panchenko2013sherrington}
Dmitry Panchenko.
\newblock {\em {The Sherrington-Kirkpatrick model}}.
\newblock Springer Science \& Business Media, 2013.

\bibitem[Pan14]{panchenko2014parisi}
Dmitry Panchenko.
\newblock The {P}arisi formula for mixed $ p $-spin models.
\newblock {\em The Annals of Probability}, 42(3):946--958, 2014.

\bibitem[Pan18]{panchenko2018k}
Dmitry Panchenko.
\newblock On the $k$-sat model with large number of clauses.
\newblock {\em Random Structures \& Algorithms}, 52(3):536--542, 2018.

\bibitem[Par79]{parisi1979infinite}
Giorgio Parisi.
\newblock Infinite number of order parameters for spin-glasses.
\newblock {\em Physical Review Letters}, 43(23):1754, 1979.

\bibitem[Par06]{parisi2006computing}
Giorgio Parisi.
\newblock Computing the number of metastable states in infinite-range models.
\newblock {\em arXiv preprint arXiv:cond-mat/0602349}, 2006.

\bibitem[Pil14]{pilipenko2014introduction}
Andrey Pilipenko.
\newblock {\em An introduction to stochastic differential equations with
  reflection}, volume~1.
\newblock Universit{\"a}tsverlag Potsdam, 2014.

\bibitem[PX21]{perkins2021frozen}
Will Perkins and Changji Xu.
\newblock Frozen 1-rsb structure of the symmetric ising perceptron.
\newblock In {\em Proceedings of the 53rd Annual ACM SIGACT Symposium on Theory
  of Computing}, pages 1579--1588, 2021.

\bibitem[PY95]{przeslawski1995lipschitz}
Krzysztof Przes{\l}awski and David Yost.
\newblock Lipschitz retracts, selectors, and extensions.
\newblock {\em Michigan Mathematical Journal}, 42(3):555--571, 1995.

\bibitem[RTV86]{rammal1986ultrametricity}
Rammal Rammal, G{\'e}rard Toulouse, and Miguel~Angel Virasoro.
\newblock Ultrametricity for physicists.
\newblock {\em Reviews of Modern Physics}, 58(3):765, 1986.

\bibitem[Rue87]{ruelle1987mathematical}
David Ruelle.
\newblock A mathematical reformulation of {D}errida's {REM} and {GREM}.
\newblock {\em Communications in Mathematical Physics}, 108(2):225--239, 1987.

\bibitem[RV17]{rahman2017independent}
Mustazee Rahman and B\'alint Vir\'ag.
\newblock Local algorithms for independent sets are half-optimal.
\newblock {\em The Annals of Probability}, 45(3):1543--1577, 2017.

\bibitem[RW94]{rogers1994diffusions}
L.~Chris~G. Rogers and David Williams.
\newblock {\em {Diffusions, Markov processes and martingales: Volume 2, It{\^o}
  calculus}}, volume~2.
\newblock Cambridge university press, 1994.

\bibitem[Sel21a]{sellke2020approximate}
Mark Sellke.
\newblock Approximate {Ground} {States} of {Hypercube} {Spin} {Glasses} are
  {Near} {Corners}.
\newblock {\em Comptes Rendus. Math\'ematique}, 359(9):1097--1105, 2021.

\bibitem[Sel21b]{sellke2021optimizing}
Mark Sellke.
\newblock Optimizing mean field spin glasses with external field.
\newblock {\em arXiv preprint arXiv:2105.03506}, 2021.

\bibitem[Shv84]{shvartsman1984lipshitz}
Pavel Shvartsman.
\newblock Lipshitz selections of multivalued mappings and traces of the
  {Z}ygmund class of functions to an arbitrary compact, dokl. acad. nauk sssr
  276 (1984), 559-562.
\newblock In {\em English transl. in Soviet Math. Dokl}, volume~29, pages
  565--568, 1984.

\bibitem[Shv02]{shvartsman2002lipschitz}
Pavel Shvartsman.
\newblock Lipschitz selections of set-valued mappings and {H}elly’s theorem.
\newblock {\em The Journal of Geometric Analysis}, 12(2):289--324, 2002.

\bibitem[SK75]{sherrington1975solvable}
David Sherrington and Scott Kirkpatrick.
\newblock Solvable model of a spin-glass.
\newblock {\em Physical review letters}, 35(26):1792, 1975.

\bibitem[Sub17]{subag2017complexity}
Eliran Subag.
\newblock The complexity of spherical $ p $-spin models—a second moment
  approach.
\newblock {\em The Annals of Probability}, 45(5):3385--3450, 2017.

\bibitem[Sub18]{subag2018free}
Eliran Subag.
\newblock Free energy landscapes in spherical spin glasses.
\newblock {\em arXiv preprint arXiv:1804.10576}, 2018.

\bibitem[Sub21]{subag2018following}
Eliran Subag.
\newblock {Following the Ground States of Full-RSB Spherical Spin Glasses}.
\newblock {\em Communications on Pure and Applied Mathematics},
  74(5):1021--1044, 2021.

\bibitem[SZ21]{subag2021concentration}
Eliran Subag and Ofer Zeitouni.
\newblock Concentration of the complexity of spherical pure p-spin models at
  arbitrary energies.
\newblock {\em Journal of Mathematical Physics}, 62(12):123301, 2021.

\bibitem[Tal06a]{talagrand2006spherical}
Michel Talagrand.
\newblock Free energy of the spherical mean field model.
\newblock {\em Probability Theory and Related Fields}, 134:339--382, 03 2006.

\bibitem[Tal06b]{talagrand2006parisi}
Michel Talagrand.
\newblock {The Parisi formula}.
\newblock {\em Annals of Mathematics}, pages 221--263, 2006.

\bibitem[Tal11a]{talagrand2011mean1}
Michel Talagrand.
\newblock {\em Mean Field Models for Spin Glasses. Volume I: Basic Examples},
  volume~54.
\newblock Springer Science \& Business Media, 2011.

\bibitem[Tal11b]{talagrand2011mean2}
Michel Talagrand.
\newblock {\em Mean Field Models for Spin Glasses. Volume II: Advanced
  Replica-Symmetry and Low Temperature}, volume~55.
\newblock Springer Science \& Business Media, 2011.

\bibitem[Wei22]{wein2020independent}
Alexander~S Wein.
\newblock Optimal low-degree hardness of maximum independent set.
\newblock {\em Mathematical Statistics and Learning}, 2022.

\bibitem[Zei15]{zeitouni2015gaussian}
Ofer Zeitouni.
\newblock Gaussian {F}ields {N}otes for {L}ectures.
\newblock
  \texttt{https://www.wisdom.weizmann.ac.il/\char`\~zeitouni/notesGauss.pdf},
  2015.

\bibitem[ZK07]{zdeborova2007phase}
Lenka Zdeborov{\'a} and Florent Krz{\k{a}}ka{\l}a.
\newblock Phase transitions in the coloring of random graphs.
\newblock {\em Physical Review E}, 76(3):031131, 2007.

\end{thebibliography}
